\documentclass[11pt,reqno]{amsart}
\usepackage[margin=2.6cm]{geometry}
\usepackage{amsmath,amssymb,amsthm,mathrsfs}
\usepackage[T1]{fontenc}
\usepackage[utf8]{inputenc}
\usepackage{lmodern}
\usepackage{microtype}
\usepackage{float}
\usepackage{xcolor}
\usepackage{longtable,array,booktabs}

\usepackage[unicode,colorlinks=true,linkcolor=blue!50!black,citecolor=blue!45!black,urlcolor=blue!60!black]{hyperref}

\theoremstyle{plain}
\newtheorem{theorem}{Theorem}[section]
\newtheorem{proposition}[theorem]{Proposition}
\newtheorem{lemma}[theorem]{Lemma}
\newtheorem{corollary}[theorem]{Corollary}
\theoremstyle{definition}

\theoremstyle{remark}
\newtheorem{remark}[theorem]{Remark}

\numberwithin{equation}{section}

\DeclareMathOperator{\Div}{div}
\DeclareMathOperator{\grad}{grad}
\DeclareMathOperator{\curl}{curl}
\DeclareMathOperator{\Curl}{Curl}
\DeclareMathOperator{\tr}{tr}
\DeclareMathOperator{\supp}{supp}

\newcommand{\breg}{\mathfrak b}
\newcommand{\sreg}{\mathfrak s}
\newcommand{\preg}{\mathfrak p}
\newcommand{\rpar}{\mathfrak r}

\newcommand{\Ufield}{\mathsf U}
\newcommand{\Kmat}{\mathsf K}
\newcommand{\Mmat}{\mathsf M}
\newcommand{\eframe}{\mathsf e}   
\newcommand{\Vfield}{\mathsf V}
\newcommand{\vLag}{\mathsf v}
\newcommand{\BLag}{\mathsf B}
\newcommand{\bLag}{\mathsf b}
\newcommand{\gmetric}{\mathbf g}
\newcommand{\Bzero}{\mathbf B_0}
\newcommand{\Binf}{\mathbf B_\infty}

\newcommand{\Cprop}{\mathsf C}
\newcommand{\Sprop}{\mathsf S}
\newcommand{\Real}{\mathbb R}
\newcommand{\Rn}{\Real^n}
\newcommand{\Rst}{\Real^{1+n}}
\newcommand{\Riesz}{\mathrm R}

\makeatletter
\def\l@section{\@tocline{1}{6pt plus1pt}{1.2pc}{2.2pc}{}}
\def\l@subsection{\@tocline{2}{0pt}{3.4pc}{2pc}{}}
\makeatother

\title[Low-regularity solutions of Ideal MHD]{Ideal MHD below the classical well-posedness threshold}
\author[]{Matteo Giardi$^\dagger$}
\thanks{$\dagger$ M.G. was supported by the International Max Planck Research School Mathematics
in the Sciences (IMPRS MiS) and thanks his advisor L\'aszl\'o
Sz\'ekelyhidi Jr. for helpful comments on an earlier version of this
manuscript.}

\date{}

\begin{document}
\begin{abstract}
We establish local existence and uniqueness of solutions for the ideal incompressible magnetohydrodynamics system posed on
$[0,T]\times\Rn$ with $n\ge2$ and initial conditions given by a nonzero constant magnetic field $\Bzero$ and arbitrary divergence-free velocity $v_0\in H^s$, in the range $(n+1)/2<s\le n/2+1$. The proof uses a Lagrangian wave--Hodge reformulation and exploits an Alfvén null--structure hidden in the pressure forcing. This work builds on Zhang \cite{Zhang2026} and provides a bridge from the Lagrangian theory to the Eulerian Cauchy
problem. In particular, the constructed solutions are induced by a bi-Lipschitz and volume-preserving map.
\end{abstract}

\maketitle

\tableofcontents

\section{Introduction}
\label{sec:introduction}

We study the ideal incompressible magnetohydrodynamic equations (MHD)
posed on $[0,T]\times\Rn_x$. This is a classical macroscopic model describing
the evolution of electrically conducting fluids such as plasmas and
liquid metals. The
dynamics of the velocity field $v$ is given by the momentum equation with Lorentz force and total pressure gradient $\nabla p$ ensuring incompressibility. The magnetic field $B$ is subject to the induction equation,
which couples the Maxwell--Faraday law with Ohm's law for the electric
field. In Eulerian variables
$(t,x)$ this reads 
\begin{equation}\label{eq:mhd}\tag{MHD}
\begin{cases}
\partial_tv+\nabla_v v-\nabla_B B+\nabla p=0,\\
\partial_tB+\nabla_v B-\nabla_B v=0,\\
\Div_x v=\Div_x B=0.
\end{cases}
\end{equation}
We fix initial data $(v,B)|_{t=0}=(v_0,\Bzero)$ and throughout this paper,
$\Bzero\equiv\Binf\in\Rn\setminus\{0\}$ will be a constant non-zero vector.
We write $v_{\rm A}:=|\Binf|$ for the corresponding Alfv\'en speed.

Set $b:=B-\Binf$, we say that a pair $(v,B)$ with
$v,b\in C([0,T],L^2_{\rm loc})$ is a weak solution of \eqref{eq:mhd} having data
$(v_0,\Binf)$ if $\Div_{x}v=\Div_{x}b=0$ in $\mathcal D'$ and
\[
 \int_0^T  \int_{\Rn}
 \Big(v\cdot\partial_t\varphi+(v\otimes v-B\otimes B):\nabla\varphi\Big) dx dt
 +\int_{\Rn}v_0\cdot\varphi(0) dx=0
\]
for every divergence-free
$\varphi\in C_c^\infty([0,T)\times\Rn;\Rn)$, while
\[
 \int_0^T  \int_{\Rn}
 \Big(B\cdot\partial_t\psi+(B\otimes v-v\otimes B):\nabla\psi\Big) dx dt
 +\int_{\Rn}\Binf\cdot\psi(0) dx=0
\]
for every $\psi\in C_c^\infty([0,T)\times\Rn;\Rn)$.\\

We aim to prove the following theorem whose full statement will be given later in \S \ref{ssec:mainthm}.
\begin{theorem}[existence and uniqueness] Let $n\ge2$. Fix
$\Bzero\equiv\Binf\in\Rn\setminus\{0\}$ and a divergence free vector field $v_0\in H^{s}(\Rn)$, with 
\[
    \frac{n+1}2<s\le\frac n2+1.
\] 
Then there exist a time $T_*>0$ and a distributional
solution $(v,B,p)$ of \eqref{eq:mhd} on $[0,T_*]\times\Rn$ with initial data
$(v_0,\Binf)$, such that  
    $v,B-\Binf,p \in C([0,T_*],H^{s}(\Rn))$. The solution is also unique in this class among the ones sharing the same initial data.
\end{theorem}

We now set some more notation, motivate the result and recall previous contributions to the matter.\\

\textbf{Brief notation overview.} We denote by $x$ and $y$ the two variables given by the change of coordinates $x=X(t,y)$, with $X$ the Lagrangian flow of the velocity $v$, and call $w:=X-\mathrm{id}$ the \emph{displacement} associated to $X$. The
Eulerian vector fields $v,B$ determine the following vector fields along $X$:
\begin{equation}
    \vLag:=v\circ X=\partial_tX, \quad \BLag:=B\circ X=\mathcal BX, \quad \bLag:=b\circ X=\BLag-\Binf.
\end{equation}
We sometimes call $y$ the label coordinate and $x$ the Cartesian or Eulerian coordinate. We use $d=d_{y}$ for exterior differentiation in the label variable $y$ and for the componentwise
differential of maps $F=(F^i):\Real^n_y\to \Real_x^n$.
We reserve $\nabla$ for the flat Eulerian connection, so that $\nabla_v B=(v\cdot\nabla)B$ has its classical meaning we write $\nabla^{y}$ for the flat
label connection instead and use $\eframe_i=\left.\partial_{x^i}\right|_{X}$ for the Cartesian vector fields along $X$ and $E_i=dX^{-1}\eframe_i=X^*\partial_{x^i}$ for their genuine pullbacks to label space, defined in \eqref{eq:component-derivative}. 
Fourier transforms are taken with the unitary normalisation. We use a hat for the spatial
transform in $y$ and a tilde for the spacetime transform in $(t,y)$. Local variations of these conventions are specified when they occur. Harmless dimensional constants are
denoted by $c_n$ or $C$. 
Matrix and tensor norms are Frobenius norms unless stated otherwise. Finally, $\mathrm{id}$ denotes
the identity map and $\mathrm{Id}$ the identity matrix. For reference, we collect the most recurrent objects in \S\ref{ssec:table}.\\

\textbf{Comparison with the classical local theory and literature.} To prove local existence and uniqueness, one usually starts with some kind of, say Galerkin-type, approximation and then tries to prove high-order energy estimates uniform in the approximating parameter to gain some compactness and pass to the limit. The key technical tool is the standard Kato--Ponce~\cite{KatoPonce1988} commutator bound
\[
\|[\Lambda^s,v\cdot\nabla]u\|_{L^2}
\lesssim
\|\nabla v\|_{L^\infty}\|u\|_{H^s}
+\|v\|_{H^s}\|\nabla u\|_{L^\infty}, 
\]
where $\Lambda^s=(1-\Delta)^{s/2}$. This, exploiting the incompressibility properties and skew-symmetries of \eqref{eq:mhd} once multiplied by $v,\, b$ leads to
\begin{equation}\label{eq:standardHsintro}
    \frac{d}{dt}
 \Bigl(\|v(t)\|_{H^s}^2+\|b(t)\|_{H^s}^2\Bigr)
 \lesssim
 \Bigl(\|\nabla v(t)\|_{L^\infty}
       +\|\nabla b(t)\|_{L^\infty}\Bigr)
 \Bigl(\|v(t)\|_{H^s}^2+\|b(t)\|_{H^s}^2\Bigr).
\end{equation}
When $s>\frac n2+1$, the Sobolev embedding
$H^s\hookrightarrow W^{1,\infty}$ then provides the missing bound to close the estimate and leads to the
classical local well-posedness theory, see~\cite{ChenMiaoZhang2010} for a more general statement in Triebel–Lizorkin spaces. Note that above this threshold the velocity is Lipschitz and so is its flow.

This work is concerned with the low-regularity range
\[
 \frac{n+1}{2}<s\leq\frac n2+1,
\]
where the $H^s$ embedding above is not strong enough to control the Lipschitz norm of the velocity and
the classical argument fails. We build on Zhang~\cite{Zhang2026}. There, the author, for a nonzero constant initial
magnetic field, identified in Lagrangian coordinates a degenerate nonlinear wave equation. The key observation is that the nonlinear forcing of the wave operator hides a null-form structure which, as in the classical theory~\cite{KlainermanMachedon1993,FoschiKlainerman2000}, can be exploited to improve on the required regularity for local well-posedness. This led to the first construction below the classical threshold, namely existence and uniqueness for an auxiliary Lagrangian problem when $2\leq n\leq4$ in the regularity range above. 

This null-structure is the mathematical manifestation of the
\emph{Alfv\'en waves} discovered in \cite{Alfven1942}. Its role in the long-time dynamics of ideal MHD was already exploited by Bardos--Sulem--Sulem~\cite{BardosSulemSulem1988}, who proved global existence and scattering for small perturbations of a strong constant magnetic field. The global dynamics near constant magnetic backgrounds have since been studied extensively; see, for instance, He--Xu--Yu~\cite{HeXuYu2018}, Cai--Lei~\cite{CaiLei2018}, and the references therein. More relevant to us, are the aforementioned \cite{Zhang2026}, which stems from the previous work \cite{Zhang2023elastic} of the same author and that of Andersson and Kapitanski~\cite{AnderssonKapitanski2023} addressing a
wave-map and Lagrangian formulation for neo-Hookean elasticity, respectively.

This paper rewrites the Lagrangian mechanism of \cite{Zhang2026} into an invariant wave--Hodge formulation \eqref{eq:WaveHodgeliterature} and bridges the gap with the Eulerian Cauchy local existence and uniqueness theory in every dimension $n\geq2$. Our result is self-contained. 

Let
\begin{equation}\label{eq:Bmult}
 \mathcal B:=\Bzero\cdot\nabla^{y}
            =\Binf\cdot\nabla^{y},
 \qquad
 \widehat{\mathcal B f}(\xi)
   =\mathrm i\lambda(\xi)\widehat f(\xi),
 \qquad
 \lambda(\xi):=\Binf\cdot\xi .
\end{equation}
Regularity along this preferred direction, which we call longitudinal or parallel, plays a crucial role in the
analysis and is measured by the fixed-time scale $Y^{\sreg,\preg}$
defined below in \eqref{eq:Y}. Here, $\sreg$ is the isotropic regularity index and $\preg$ the parallel regularity index.

For the displacement $w=X-\mathrm{id}$, the ideal MHD system will be
rewritten in terms of the Alfv\'en and wave operators
\[
 \mathcal A^\pm:=\partial_t\pm\mathcal B,
 \qquad
 \Box:=\mathcal A^+\mathcal A^-
      =\partial_t^2-\mathcal B^2
\]
as the wave--Hodge system
\begin{equation}\label{eq:WaveHodgeliterature}
 \Box w=-F,\qquad
    \begin{cases}
        \curl_{\gmetric}F=0,\\
        \Div_{\gmetric}F=\mathcal Q[dw],
    \end{cases}
\end{equation}
for the metric $\gmetric=(dX)^{\top}dX$ and with $\mathcal{Q}$ carrying the Alfv\'en null--structure mentioned above. 

The free evolution generated by $\Box$ is
non-dispersive: the two carried
half-waves are simply translated in time in opposite directions along $\Binf$. We call these space-time variables null or Alfv\'en coordinates.  
Nevertheless, away from the characteristic set the symbol of $\Box$ is
elliptic along these directions, and can be used to gain parallel regularity on the source. Near the characteristic set, the nonlinear forcing must instead be handled through the Alfv\'en null structure hidden in $\mathcal{Q}$, which cancels the worst
co-propagating interactions.  In this 1+1 hyperbolic setting, this mechanism is particularly apparent and can be made explicit with the following well-known heuristics, see for example \cite{Tao2006Dispersive}. Let us informally write 
\[
 \mathcal Q[dw]\sim Q_0(dw,dw)\sim V^+\otimes V^-,
 \qquad
 V^\pm:=\mathcal A^\pm(dw)=d(\mathcal A^\pm w).
\]
We will treat \eqref{eq:WaveHodgeliterature} (after taking a derivative) as a semilinear problem, and as a first approximation we may think of $d w$ as solving
\[
 \Box dw=d \Box w=0.
\]
For such a free wave, after a fixed rotation for which $\Binf=v_{\rm A}e_1$ and splitting $y=y_1e_1+y'$ one sees that
\begin{equation}\label{eq:motivationintro}
 \mathcal A^\mp V^\pm=d(\Box w)=0
 \qquad\Longrightarrow\qquad
 V^\pm(t,y)=F^\pm(y_1\pm v_{\rm A}t,y'),
\end{equation}
for some fixed profiles $F^\pm$ determined by the initial data. Now consider the null-coordinates change of variables
\[
 (t,y_1)\longmapsto(u,v):=(y_1+v_{\rm A}t,y_1-v_{\rm A}t),
 \qquad
 dtdy_1=\frac1{2v_{\rm A}}dudv.
\]
This is particularly helpful when trying to estimate the non-linearity in $L^2_{t,y_1}$ at fixed $y'$. Indeed, the transport condition \eqref{eq:motivationintro} gives

\begin{equation}\label{eq:keycomputation}
\begin{split}
    \int_\Real\int_\Real |V^+(t,y)|^2|V^-(t,y)|^2d td y_1 &=\int_{\Real}\int_{\Real}
 |F^+(y_1+v_{\rm A}t,y')|^2\,|F^-(y_1-v_{\rm A}t,y')|^2dtdy_1\\
 &=
 \frac1{2v_{\rm A}}
 \int_{\Real}|F^+(u,y')|^2du
 \int_{\Real}|F^-(v,y')|^2dv .
\end{split}
\end{equation}
Thus, the $L^2_{t,y_1}$-norm of the product is controlled by the product of
the two longitudinal $L^2$ norms without requiring $H^s\hookrightarrow L^\infty$. This is exactly the missing ingredient in \eqref{eq:standardHsintro}. The precise lower bound for $s$ then comes from the transverse numerology that follows: if $w$ is assumed in $H^s$ the differentiated unknown $dw$ in $Q_0$ carries $\sigma:=s-1$ isotropic derivatives
and distributing the $H^{\sigma}$ weight in the product on top of \eqref{eq:keycomputation}, enforces
\[
 \sigma>\frac{n-1}{2} \quad\Longrightarrow\quad H^\sigma_{y'}\hookrightarrow L^\infty_{y'},
 \qquad\text{equivalently}\qquad
 s>\frac{n-1}{2}+1=\frac{n+1}{2}.
\]
and leads to 
\begin{equation}\label{eq:simplequadraticintro}
    \|V^+V^-\|_{L^2_tH^\sigma_y}
 \lesssim v_A^{-1/2}
 \|F^+\|_{H^\sigma}\|F^-\|_{H^\sigma}.
\end{equation}
This is just the classical Sobolev well-posedness threshold with one
spatial dimension removed.  Theorem~\ref{thm:null-form} replaces the exact
free-transport relations \eqref{eq:motivationintro} by the corresponding
modulation estimate, while the wave--Hodge argument reinstates the
coefficients and contractions in $\mathcal Q$. We remark that the key computation \eqref{eq:keycomputation} also underlies the uniqueness argument of Proposition \ref{prop:unconditional-uniqueness} and thus forms the backbone of this work and \cite{Zhang2026}. See \S\ref{ssec:overview} for more details.

Loosely speaking, \eqref{eq:mhd} behaves like
incompressible Euler in the coordinates transverse to $\Binf$, while the Alfv\'en wave-geometry and
the null-form cancellations supply the missing longitudinal control. In our main Theorem \ref{thm:main} we call regime~(R1) the one in which only the initial $H^s$ regularity on the initial velocity is assumed. Here the lifespan may depend on $v_0$ through the modulus \eqref{eq:datum-modulus}. If
\[
 \nabla v_0\in Y^{s-1,\rpar},
 \qquad 0<\rpar\leq1,
\]
then the minimal existence time provided by \ref{thm:main} can be bounded in terms of the $H^s$ and anisotropic $Y^{s-1,\rpar}$ norm of the data and its gradient. We call this regime~(R2).

A related anisotropic condition already appears in the literature.
Ifrim, Pineau, Tataru, and Taylor~\cite{ifrim} in the
setting of free-boundary MHD, require
\[
 \nabla_Bv,\nabla_BB\in H^{s-1/2}
\]
for an $H^s$ initial data, where $s>\frac n2+1$ in their result.  Although the settings and techniques differ, both
conditions reflect the `smoothing' role of derivatives along the magnetic field.

The threshold in \cite{ifrim} is sharp for the class treated there, which
reduces to the hydrodynamic case when $B\equiv0$. In fact, the authors explicitly observe, ``the best we
could hope for in the case of MHD is also $s>\frac{n+2}{2}$'' \cite[p.~5]{ifrim}. Our result shows that once the background field is
bounded away from zero $\Bzero\equiv\Binf\ne0$, the threshold drops to $s>\frac{n+1}2$, at least in the problem set on the full space. In particular, Theorem~\ref{thm:main} produces solutions at a regularity where, for $n=2,3$, the corresponding Euler system is strongly
ill-posed by Bourgain--Li~\cite{BourgainLi2015Invent,BourgainLi2021IMRN}. Note, however, that failure of continuity of the solution map does not necessarily mean that an existence and uniqueness theory cannot be established. At an analytic level, the distinction with the hydrodynamic theory is apparent from the fact that all the relevant constants below degenerate as $v_{\rm A}\downarrow0$, as can already be seen from \eqref{eq:keycomputation} above, and the Euler equation cannot be recovered in the limit. In this sense, \cite{Zhang2026} and the present work isolate the regime in which ideal
magnetohydrodynamics has a genuinely different behaviour.\\

\textbf{The wave-adapted spaces and parameters.} The various operators and observations above motivate the scale of spaces we introduce next. These are an anisotropic version of the classical Bourgain $X^{\sreg,\breg}$ spaces widely used in non-linear dispersive equations, see for example \cite{Tao2006Dispersive}.  

Throughout 
\[
    n\ge2, \qquad \frac{n+1}2<s\le\frac n2+1, \qquad \breg\in(\frac12,1), \qquad\varepsilon:=1-\breg, \qquad \sreg\ge0.
\]
We also require $T\le1$ for convenience. We refer to $(n,s,\breg,v_{\rm A})$ as the
standing parameters.  The bilinear estimates of
\S\ref{ssec:null-form} are proved for the wider range $\breg\in(\frac12,1]$,
which costs nothing and is recorded there. Moreover, we let $\Lambda:=(1-\Delta_{y})^{1/2}$ and associate to the Alfvén and magnetic derivations the self-adjoint operators
\begin{equation}\label{eq:self-adjoint-derivatives}
    D_t:=-\mathrm i\partial_t,\qquad
    D_{\parallel}:=-\mathrm i\mathcal B, \qquad
    D_{\pm}:=-\mathrm i\mathcal A^{\pm}=D_t\pm D_{\parallel},
\end{equation}
where $\parallel$ denotes the full direction $\Binf$, without normalising its
magnitude, so that the parallel and Alfv\'en/null directional symbols are given by
\begin{equation}\label{eq:taupm}
    \lambda=\lambda(\xi)=\Binf\cdot\xi, \qquad\tau_+:=\tau+\lambda,\qquad \tau_-:=\tau-\lambda.
\end{equation}
To adapt the Bourgain spaces, similarly to \cite{Zhang2026}, we define the modulation weight
\begin{equation}\label{eq:modulation}
 m:=\min\big(\langle\tau_+\rangle,\langle\tau_-\rangle\big)
   =\big\langle\big||\tau|-|\lambda|\big|\big\rangle,
\end{equation}
and the corresponding modulation-weighted spacetime and fixed-time scales
\begin{equation}\label{eq:Y}
\|f\|_{Y^{\sreg,\preg}_{\breg}}:=\big\|\langle D_{\parallel}\rangle^{\preg}m^{\breg}\Lambda^{\sreg}f\big\|_{L^{2}(\mathbb{R}^{1+n})},
\qquad
\|f\|_{Y^{\sreg,\preg}}:=\big\|\langle D_{\parallel}\rangle^{\preg}\Lambda^{\sreg}f\big\|_{L^{2}(\mathbb{R}^{n})}.
\end{equation}
The wave-weight is then
\begin{equation}\label{eq:Omega}
 \Omega:=\frac{\langle\tau_+\rangle\langle\tau_-\rangle}{m^{1-\breg}}
        =\frac{\langle\tau_+\rangle\langle\tau_-\rangle}{m^{\varepsilon}},
\end{equation}
and we define the spacetime wave-adapted norm by
\begin{equation}\label{eq:Hsymbol}
\|f\|_{H^{\sreg,\preg}_{\breg}}:=
\big\|\langle D_{\parallel}\rangle^{\preg-1}
\Omega(D_+,D_-)\Lambda^{\sreg}f\big\|_{L^2(\mathbb{R}^{1+n})}.
\end{equation}
A detailed explanation of these spaces is given in
\S\ref{ssec:overview}.

Finally, we define the longitudinal modulus $\omega_{v_0}$ of a vector field $v_0\in H^s$ by
\begin{equation}\label{eq:datum-modulus}
 \omega_{v_0}(T)
 :=\big\|\min(1,T\langle D_{\parallel}\rangle)
                 \Lambda^{s-1}\nabla v_0\big\|_{L^2},
 \qquad 0<T\le1.
\end{equation}
It satisfies $\omega_{v_0}(T)\to0$ as $T\downarrow0$, and if
$\nabla v_0\in Y^{s-1,\rpar}$, $0<\rpar\le1$, then
\begin{equation}\label{eq:omegav0bound}
    \omega_{v_0}(T)\le T^{\rpar}\|\nabla v_0\|_{Y^{s-1,\rpar}}
\end{equation}
which follows immediately from $\min(1,r)\leq r^\rpar$ for $r\geq 0$ and $0<\rpar\leq 1$.\\
 
\textbf{Conserved quantities.} A field tending to $\Binf\ne0$ carries infinite energy, so the constant part is subtracted. The conserved quantities of interest, namely total energy and cross-helicity, are then
\begin{equation}\label{eq:energy}
E:=\tfrac12\int\big(|v|^2+|b|^2\big) dx,\qquad
H_{\rm c}:=\int v \cdot b dx.
\end{equation}
Here both integrals are finite for a solution in the class of
Theorem~\ref{thm:main}. 

\subsection{AI and LLM use declaration} The v1 of this paper was partly written using Claude Opus 4.1 and this v2 is a cleanup of that and ChatGPT Astra was used to tighten some of the arguments. From the v1 to the v2 there is no real technical difference or additional claim, the result is the same. Some of the proofs were simplified in the interest of readability, more than technical wizardry per se.

The v1 was born in an attempt to understand the work of Zhang \cite{Zhang2026} in a more geometric way, out of personal taste and given that geometry was crucial in my thesis work on Onsager's conjecture for ideal MHD.  I then quickly realised that, on the one hand, some fixes were needed and on the other, the flow map $X$ could be made bi-Lipschitz (see the discussion about $\delta$ in the overview \S \ref{ssec:overview}) and the full Eulerian theorem was within reach with an additional energy argument for uniqueness, mimicking the overall null-structure philosophy.

Once the ideas seemed to give a complete proof, I decided to speed up the writing process using my basic version of Claude. This helped greatly, as many of the estimates are small adaptations from classical Fourier-based nonlinear wave technology and the arguments needed only a few tweaks. I didn't initially disclose the LLM use out of insecurity and due to the scope of its use. This v2 is also an attempt to change this. Rest assured that proper citations and credit were given explicitly and to the best of my (limited) knowledge.  

\subsection{The main result}\label{ssec:mainthm}
With the preliminary definitions in place, we state our main result. 

\begin{theorem}[local existence and uniqueness below the classical threshold]\label{thm:main}
Fix $\breg\in(\frac12,1)$. Let $n\ge2$, let $\frac{n+1}2<s\le\frac n2+1$, let
$\Bzero\equiv\Binf\in\Rn\setminus\{0\}$, and let $v_0\in H^{s}(\Rn)$ be divergence-free. There are
constants $C\ge1$ and $\delta\in(0,1)$, depending only on
$n,s,\breg$ and $v_{\rm A}$, and there exist a time $T_*>0$, a distributional
solution $(v,B, p)$ of \eqref{eq:mhd} on $[0,T_*]\times\Rn$ with data
$(v_0,\Binf)$, with the following properties:
\begin{enumerate}
\item for $b:=B-\Binf$, the fields $v$, $b$ and $p$ lie in
$C([0,T_*],H^{s}(\Rn))$. The pressure is normalised by
\begin{equation}\label{eq:rieszmain}
    p=\Riesz_i\Riesz_j
 \big(v^iv^j-\Binf^ib^j-b^i\Binf^j-b^ib^j\big),
\end{equation}
where $\Riesz_i$ are the Euclidean Riesz transforms;
\item the solution is carried by a Lagrangian flow
$X:=\mathrm{id}+w:[0,T_*]\times\Rn\to\Rn$ with $X(0,\cdot)=\rm{id}$, uniformly bi-Lipschitz
and volume-preserving, $\det dX\equiv1$, such that 
\[
    v=(\partial_tX)\circ X^{-1} \qquad \textit{and} \qquad B=(\mathcal BX)\circ X^{-1}.
\]
Moreover, there is a spatially exact spacetime 1-form
$\Ufield\in H^{s-1,1}_{\breg}(\Rst)$ such that
$\Ufield=dw$ on $[0,T_*]\times\Rn$ and for constants fixed by
the standing parameters,
\begin{equation}\label{eq:main-adapted-bound}
 \|\Ufield\|_{H^{s-1,1}_{\breg}}
 \le C\|v_0\|_{H^s},
 \qquad
 \sup_{t\in[0,T_*]}\|dw(t)\|_{Y^{s-1,1}}\le\delta;
\end{equation}
where $\delta$ is chosen so that the second bound implies
$\|dw\|_{L^\infty_{t,y}}\le\frac12$;
\item the renormalised energy $E$ and cross helicity $H_{\rm c}$ of \eqref{eq:energy} are
conserved on $[0,T_*]$;
\item the solution is unique among all distributional solutions with the same
initial data for which $v$ and $b$ belong to
$C([0,T_*],H^s)$. 
\end{enumerate}
The existence time $T_*$ is governed by the longitudinal regularity of the datum. We
distinguish the following two regimes.
\begin{enumerate}
\item[\textbf{\emph{(R1)}}]  $T_*$ depends on $n,s,\breg,
v_{\rm A}$, on $\|v_0\|_{H^{s}}$ and on the function
$T\mapsto\omega_{v_0}(T)$ defined in \eqref{eq:datum-modulus};
\item[\textbf{\emph{(R2)}}] if $\|\nabla v_0\|_{Y^{s-1,\rpar}}<\infty$
for some $\rpar\in(0,1]$, then $T_*$ depends on $n,s,\breg,\rpar,v_{\rm A}$
and on $\|v_0\|_{H^{s}}+\|\nabla v_0\|_{Y^{s-1,\rpar}}$ alone. Moreover,
the spacetime field in~(2) satisfies
\[
 \|\Ufield\|_{H^{s-1,1+\rpar}_{\breg}}
 \le C\|\nabla v_0\|_{Y^{s-1,\rpar}}.
\]
\end{enumerate}
\end{theorem}

\subsection{Overview of the proof}\label{ssec:overview}

The starting points are the Lagrangian flow equation for the velocity~$v$, namely
\[
 v\circ X=\partial_t X
\]
and the fact that reading the Maxwell--Faraday--Ohm law in \eqref{eq:mhd} as Lie transport of the magnetic field $B$ along $v$, one has
\[
 B=X_*\Binf \qquad\leadsto \qquad B\circ X=\mathcal BX,
 \qquad \mathcal B=\Binf\cdot\nabla^y.
\]
The momentum equation along $X$ can then be seen as a \emph{constant-coefficients} wave equation:  
\[
 (\partial_tv+\nabla_vv)\circ X=\partial_t\big(v\circ X\big)=\partial_t^2X, \qquad (\nabla_BB)\circ X= \mathcal B(B\circ X)=\mathcal B^2X
\]
that is
\[
    \partial_t^2X-\mathcal B^2X=-(\nabla p)\circ X.
\]
Setting
\[
 w:=X-\mathrm{id},\qquad
 F:=(\nabla p)\circ X,\qquad
 \Box:=\partial_t^2-\mathcal B^2,
\]
and using $\mathcal B^2\mathrm{id}=0$, we obtain the
wave--Hodge system:
\begin{equation}\label{eq:whintro}
 \Box w=-F,\qquad 
    \begin{cases}
        \curl_{\gmetric}F=0,\\
        \Div_{\gmetric}F=\mathcal Q[dw],
    \end{cases}
\end{equation}
for the metric $\gmetric=(dX)^{\top}dX$. The Hodge r.h.s.~contains the Alfv\'en null-form
\[
 \mathcal Q[dw]
 =-\Kmat^a_i\Kmat^b_jQ_0(\partial_aw^j,\partial_bw^i),
 \qquad
 Q_0(f,g)=\partial_tf \partial_tg-\mathcal Bf \mathcal Bg, \qquad \Kmat=(d X)^{-1}.
\]
Lemma~\ref{lem:weak-pullback} studies the weak geometric calculus behind \eqref{eq:whintro}, while
Lemma~\ref{lem:wave-hodge-reduction} supplies both the algebraic identity for the null--form and the derivation of \eqref{eq:whintro} from
smooth ideal MHD. 

Existence of a solution for this auxiliary problem is proved by a fixed-point argument run at the level of the differentiated equation \eqref{eq:diffwave} for some exact matrix 1-form
$\Ufield$. $w$ is reconstructed only afterwards, and the constraint $dw=\Ufield$ will hold on the interval of definition. At the very end, the return to the Eulerian side from the constructed weak auxiliary solution $X=\mathrm{id}+w$ and Eulerian uniqueness are addressed.\\

\textbf{Wave-adapted spaces, linear and nonlinear wave theory.} Let us first explain what the function spaces \eqref{eq:Y},\eqref{eq:Hsymbol} measure.  The
derivatives along the characteristic directions and Fourier variables of $\Box$ are
\[
 \mathcal A^\pm=\partial_t\pm\mathcal B,
 \qquad \tau_\pm=\tau\pm\Binf\cdot\xi.
\]
With
\[
 m=\min(\langle\tau_+\rangle,\langle\tau_-\rangle),
 \qquad
 \Omega=\frac{\langle\tau_+\rangle\langle\tau_-\rangle}{m^{1-\breg}},
\]
the definition \eqref{eq:Hsymbol} says that
$H_{\breg}^{\sreg,\preg}$ measures
\[
 \langle D_\parallel\rangle^{\preg-1}
 \Omega(D_+,D_-)\Lambda^{\sreg}f.
\]
The factor $\Lambda^{\sreg}$ quantifies isotropic spatial regularity,
$\langle D_\parallel\rangle^{\preg-1}$ adjusts the regularity along
$\Binf$, think of $\preg=1$ for the moment. The weight $m$ is the regularised
distance to the characteristic set $\{\tau_+=0\}\cup\{\tau_-=0\}$ where
$\Box$ is not elliptic. Note that for this degenerate operator, it is a pair of hyperplanes rather
than a cone. The product in the numerator of \eqref{eq:Omega} is the modulus of the
Fourier symbol of the regularised wave operator
\begin{equation}\label{eq:regularized-wave}
 \Box^{\mathrm{reg}}:=(1+\mathcal A^+)(1+\mathcal A^-)
 =\Box+2\partial_t+1.
\end{equation}
Indeed,
\[
 \widetilde{\Box^{\mathrm{reg}}f}
 =(1+\mathrm i\tau_+)(1+\mathrm i\tau_-)\widetilde f,
\]
and hence Plancherel's identity gives
\begin{equation}\label{eq:Omega-regularized-wave}
 \|\Omega(D_+,D_-)f\|_2
 =\|m^{-\varepsilon}\Box^{\mathrm{reg}}f\|_2.
\end{equation}
In particular, $\Omega$ measures one derivative along each Alfv\'en characteristic,
up to a fixed $\varepsilon$ loss. 

This is the usual modulation mechanism
for a wave equation, adapted in \cite{Zhang2026} to a
strongly anisotropic operator. The no-endpoint conditions $\frac12<\breg <1$ are spent as follows.

\textit{Linear wave estimates.} The upper endpoint $\breg <1$, equivalently $\varepsilon=1-\breg>0$, is needed to gain a positive power of $T$ in the
Duhamel/forced term coming from the solution operator of the linear equation
\[
    \Box u=G,
\]
leading to the linear estimate \eqref{eq:linear}. We now recall how this classical
$H^{\sreg,\preg}_{\breg}$ bound is obtained. For the zero-data Duhamel solution $u$ localised at time scale $T$,
the cutoff estimates and duality in the Sobolev embedding
$H^\varepsilon\hookrightarrow L^{2/(1-2\varepsilon)}$,
applied after the shifts $\tau\mapsto\tau\pm\lambda$,
give, respectively (see Theorem~\ref{thm:linear}),
\[
 \|m^{-\varepsilon}\Box^{\mathrm{reg}}u\|_2 \lesssim_\varepsilon T^\varepsilon
 \|\Box^{\mathrm{reg}}u\|_2\lesssim T^\varepsilon\|G\|_2.
\]
leading via \eqref{eq:Omega-regularized-wave} to the desired gain.

\textit{Wave non-linearity and null-structure.} The condition $\breg>\frac12$ is used in the non-linear part of the argument to gain from the null structure using the `stored regularity' in the modulation weight $\Omega$. The r.h.s.~of $\Box$ in \eqref{eq:whintro} contains the bilinear form $Q_0$ which shares with $\Box$ the Alfv\'en factorization:
\[
 Q_0(f,g)=\tfrac12\big(\mathcal A^+f \mathcal A^-g
                         +\mathcal A^-f \mathcal A^+g\big).
\]
Each term pairs a $+$ derivative
of one input with a $-$ derivative of the \emph{other}, while both carry a wave class weight
\[
    \Omega=\frac{\langle\tau_+\rangle\langle\tau_-\rangle}{m^{\varepsilon}}
\]
ensuring one derivative per type is controlled, in particular the product in $Q_0$, although each factor loses one Alfv\'en derivative, still has both available and stays in the class. Quantitatively, when performing estimates in the wave class $H_{\breg}^{s-1,\preg}$ in which the linear estimate \eqref{eq:linear} from above is read,  dividing $|\tau_\pm|$ coming from the differentiation $\mathcal{A}^\pm$ by
the available weight $\Omega=\frac{\langle\tau_+\rangle\langle\tau_-\rangle}{m^{\varepsilon}}$ leaves the decay $\langle\tau_\mp\rangle^{-\breg}$
in the \emph{other} null variable. The square of this kernel is integrable in
the two null variables precisely for
$\breg>\frac12$, equivalently $\varepsilon=1-\breg<\frac12$. This is the content of Theorem~\ref{thm:null-form} and results in 
\begin{equation}\label{eq:NFintro}
    \|Q_0(\Ufield,\Ufield)\|_{L^2_tH^{s-1}}
 \lesssim \|\Ufield\|_{H_{\breg}^{s-1,1}}^2.
\end{equation}
The only multiplier in $H_{\breg}^{\sreg,\preg}$ that we haven't justified yet is the weight $\Lambda^\sreg$. Unlike the usual wave operator, $\Box$ contains no transverse spatial
derivatives, leaving the transverse frequencies as parameters.  In the differentiated equation, the null form first produces the Hodge datum for $dF$, which is estimated in $L^2_tH^{s-1}$. The requirement $s>(n+1)/2$ then follows as in the toy model computation giving \eqref{eq:simplequadraticintro}. 

The key technical point which makes the whole construction work is the fact that, despite controlling only $\sigma=s-1$ isotropic derivatives, the anisotropic norm $H_{\breg}^{\sigma,1}$, which we use in the estimates above, is strong enough to provide an embedding in $C_tL^\infty_y$ via \eqref{lem:Linftyembed} and \eqref{eq:trace-class}. This supplies the a.e.~boundedness of the deformation matrix $\Ufield$ and ultimately leads to the flow map being bi-Lipschitz.\\ 

\textbf{Wave estimates meet Hodge estimates.}  The above explained the wave part of \eqref{eq:whintro} and the adapted spaces $H_{\breg}^{\sreg,\preg}$. What is left to be addressed in the auxiliary system is how the Hodge part interacts with the null-form to give the forcing $F$.  From a trial field
$\Ufield$ (recall that we work at the level $\Ufield=d w$ not on \eqref{eq:whintro} directly), the null form first produces the scalar datum
$\mathcal Q[\Ufield]$,  the metric Hodge system recovers $F[\Ufield]$, and
$dF[\Ufield]$ is then the source in the differentiated \eqref{eq:whintro}, namely
\begin{equation}\label{eq:diffwave}
    \Box\Ufield=-dF[\Ufield].
\end{equation}
This concatenation in fact produces a map $\Ufield\mapsto dF[\Ufield]$ from
$H^{s-1,1}_{\breg}$ to $L^2_tH^{s-1}$ which costs \textit{no derivatives} beyond those
measured by the wave-adapted norm. In addition to the null-form gain explained above, the technical content of this statement is the $-1$ order nature of the metric Hodge system, which compensates for $d$. Expressing it in terms of the flat one, the inverse deformation matrix $\Kmat[\Ufield]=(\mathrm{Id}+\Ufield)^{-1}$ appears \emph{only} as a coefficient of $d$, and gives
\[
 (\Div_{\gmetric},\curl_{\gmetric})
 = (\Div_y,\curl_y)
   +\Mmat[\Ufield]d, \qquad
 \Mmat[\Ufield]:=\Kmat[\Ufield]-\mathrm{Id}.
\]
The second term is a principal-order perturbation, but the fixed-point class we run the argument in, namely $\mathcal{U}_{R,\delta}$ from \eqref{eq:fixedpoint-class}, ensures $\Mmat[\Ufield]$ is $\delta$ small in the right norm, and thus the operator has small coefficients. Standard perturbation theory, namely expansion in a Neumann series, can then be applied. 

Concretely, we first rewrite the Hodge system as the resolvent identity
\begin{equation}\label{eq:resolventintro}
    dF[\Ufield]
 =
 \mathscr T_0\mathcal Q[\Ufield]
 +\mathscr T_{\Mmat[\Ufield]}dF[\Ufield] \quad\leadsto \quad dF[\Ufield]=(\mathrm{Id}-\mathscr T_{\Mmat[\Ufield]})^{-1}\mathscr T_0\mathcal Q[\Ufield]
\end{equation}
where $\mathscr T_0$ is an order-zero Calderón-Zygmund type operator and
$\mathscr T_{\Mmat[\Ufield]}$ consists of multiplication by
$\Mmat[\Ufield]$ followed by an order-zero CZ type operator. This is then a small perturbation of the identity, ensuring the displayed invertibility.  The rewriting together with the null--form estimate \eqref{eq:NFintro} then yields
\begin{equation}\label{eq:hessianintro}
    \|dF[\Ufield]\|_{L^2_tH^{s-1}}
 \lesssim_\delta
 \|\mathcal Q[\Ufield]\|_{L^2_tH^{s-1}}
 \lesssim_\delta R^2
\end{equation}
for $\Ufield\in \mathcal{U}_{R,\delta}$. No Christoffel symbols or derivatives of the metric coefficients (which would be of second order in $X$) enter this argument.

The identity \eqref{eq:resolventintro} also makes the dependence of the pressure Hessian $d F[\Ufield]$ quantitative. Indeed, bilinearity of $Q_0$ and of \eqref{eq:resolventintro} in both the argument and the coefficients/geometry leads to
\begin{equation}\label{eq:pressdiffintro}
    \|dF[\Ufield]-dF[\Vfield]\|_{L^2_tH^{s-1}}
 \lesssim_\delta
 (R+R^2)\|\Ufield-\Vfield\|_{H_{\breg}^{s-1,1}}.
\end{equation}
The term proportional to $R$ comes from replacing one factor of the null form,
whereas the term proportional to $R^2$ comes from varying the geometry. This final ingredient will make the
fixed-point map into a contraction.

We now explain the small parameter $\delta$ absent in \cite{Zhang2026} and a key contribution of this work.  We begin by observing that it is not a smallness assumption
on the initial data as $R\gtrsim\|v_0\|_{H^s}$ may be
arbitrarily large. Instead, the fixed-point class requires
\[
 \sup_{t\in[0,T]}\|\Ufield(t)\|_{Y^{s-1,1}}\le\delta.
\]
The anisotropic embedding then makes $dX=\mathrm{Id}+\Ufield$ uniformly
invertible, so every trial map $X(t,\cdot)$ in the fixed-point class is globally bi-Lipschitz, making the pullback geometry well-defined.  This is a small-deformation
condition and the short time compensates for the large data: the homogeneous part of the solution for $\Box$ is controlled by the modulus
$\omega_{v_0}(T)$ from \eqref{eq:datum-modulus} and the Duhamel part gains a positive power of
$T$.\\

\textbf{The fixed-point argument.} Starting from an exact matrix valued $1$-form $\Ufield\in \mathcal U_{R,\delta}$ we define the contraction map $\mathcal M$ on
$[0,T]$ by
\[
 \Box\mathcal M\Ufield=-dF[\Ufield],
 \qquad
 \begin{cases}
  \curl_{\gmetric[\Ufield]}F[\Ufield]=0,\\
  \Div_{\gmetric[\Ufield]}F[\Ufield]=\mathcal Q[\Ufield],
 \end{cases}
\]
with Cauchy data $(0,\nabla v_0)$. Collecting the bounds and observations above, we can then estimate the norms defining the fixed point class $\mathcal U_{R,\delta}$, see \eqref{eq:fixedpoint-class}, and obtain
\[
 \|\mathcal M\Ufield\|_{H_{\breg}^{s-1,1}}
 \le \frac R2+C_\delta T^{\varepsilon}R^2, \qquad
 \sup_{t\in[0,T]}
 \|\mathcal M\Ufield(t)\|_{Y^{s-1,1}}
 \le
 C\big(
   \omega_{v_0}(T)+T^{1/2}C_\delta R^2\big).
\]
Small $T$ therefore makes $\mathcal M$ well defined on
$\mathcal U_{R,\delta}$.

For two elements of this class, their free-wave evolution parts cancel and combining the bounds \eqref{eq:NFintro},\eqref{eq:pressdiffintro} above gives
\[
 \|\mathcal M\Ufield-\mathcal M\Vfield\|_{H_{\breg}^{s-1,1}}
 \lesssim_\delta
 T^{\varepsilon}(R+R^2)
 \|\Ufield-\Vfield\|_{H_{\breg}^{s-1,1}}.
\]
Eventually decreasing $T$, Banach's fixed-point
theorem yields a fixed point  $\Ufield$ in $\mathcal U_{R,\delta}$.  We remark that the (available) uniqueness of $\Ufield$ does not enter the final Eulerian uniqueness argument.

\textit{Off-constraint class.} The fixed-point map is deliberately defined without the volume
constraint. The point is that
the nonlinear condition
\[
 \det(\mathrm{Id}+\Ufield)=1
\]
is not known to be preserved by the iteration map $\mathcal M$, see Remark \ref{rem:generalcaseJacobi} and compare to \eqref{eq:jacobiintro} below. Another argument motivating this as a good `relaxed space' to set the contraction is discussed in Remark \ref{rem:weak-limit-heuristic}.\\

\textbf{Back to $w$, the constraint $J=1$ and the Eulerian side.}
Once the fixed point is obtained, $w$ is reconstructed via the Duhamel formula and $dw=\Ufield$ on $[0,T_*]$. It remains
to prove that the auxiliary system gives back ideal MHD. For $X=\mathrm{id}+w$ set
\[
 J:=\det dX=\det(\mathrm{Id}+\Ufield),
 \qquad
 \varrho:=\log J.
\]
The weak Jacobi identity gives
\begin{equation}\label{eq:jacobiintro}
    \Box\varrho
 =\mathcal Q[dw]
  -\Div_{\gmetric[dw]}F[dw]
 =0.
\end{equation}
Since $\varrho$ has zero data, uniqueness gives
$\varrho\equiv0$ ensuring the incompressibility constraint $J\equiv1$ is propagated.
This, together with the bi-Lipschitz bound supplied by $\delta$, allows us to
test in the weak formulation of \eqref{eq:mhd} with functions of the form $\varphi\circ X$ for $\varphi\in C_c^\infty([0,T_*)\times\Rn;\Rn)$ and return to Eulerian variables. 

A Liouville argument then identifies the candidate pressure gradient
$F\circ X^{-1}$ with $\nabla p_{\mathrm R}$ obtained by solving \eqref{eq:rieszmain}. This also supplies the missing $L^2$ control on $F$ which is needed
because a priori Lemma~\ref{lem:pressure} gives only $F\in L^\infty_tL^4$ and
$dF\in L^2_tH^{s-1}$, with no $L^2$ bound on $F$ itself. See Remark~\ref{rem:pressure-L4} for the precise role of $L^4$. A final step then shows $v,b,p\in C_tH^s$ and closes the Eulerian existence proof.\\

\textbf{The two regimes.}
The two regimes in Theorem~\ref{thm:main} use longitudinal regularity only to
determine the lifespan and the propagated norm. In \textup{(R1)} no additional assumption is made:
$\omega_{v_0}(T)\to0$ for each fixed $H^s$ datum, but its rate is not uniform
on bounded $H^s$ sets. This is because the $H^s$ mass may live at arbitrarily large (parallel) frequencies by simply becoming more concentrated. See the discussion after Lemma \ref{lem:contraction}.
In~(R2), any additional longitudinal regularity gives
\[
 \omega_{v_0}(T)
 \le
 T^{\rpar}\|\nabla v_0\|_{Y^{s-1,\rpar}},
\]
so the lifespan can be chosen in terms of norms alone. The Picard iteration propagates this regularity, ensuring
\[
 \Ufield\in H_{\breg}^{s-1,1+\rpar}
\]
as well and without having to modify $\mathcal U_{R,\delta}$.\\

\textbf{Uniqueness.} This is the content of
Proposition~\ref{prop:unconditional-uniqueness} and can be seen as a sort of
weak--strong uniqueness statement. The normalised parts of both solutions
enjoy the same $C_tH^s$ regularity, while the constructed one carries the additional structure of being Lagrangian, meaning that it is induced by a bi-Lipschitz flow map $X$ as above.  Let
$(v,B,p)$ denote this solution and let
$(\widetilde v,\widetilde B,\widetilde p)$ be any other distributional
solution with the same initial data.  Set
\[
 Z^\pm:=v\pm B,
 \qquad
 \widetilde Z^\pm:=\widetilde v\pm\widetilde B,
 \qquad
 \delta Z^\pm:=\widetilde Z^\pm-Z^\pm,
 \qquad
 z^\pm:=\frac12\bigl(Z^\pm+\widetilde Z^\pm\bigr)\mp\Binf,
\]
and, along the reference map,
\[
 \mathsf Z^\pm:=Z^\pm\circ X=\mathcal A^\pm X,
 \qquad
 \delta\mathsf Z^\pm:=\delta Z^\pm\circ X,
 \qquad
 \delta p:=\widetilde p-p,
 \qquad
 \delta\mathsf p:=\delta p\circ X.
\]
Reading the difference equation in Lagrangian coordinates gives
\begin{equation}\label{eq:overview-conservative-difference}
 \mathcal A^\mp\delta\mathsf Z^\pm
 +\Div_{\gmetric}\bigl(
   (\mathsf Z^\pm+\delta\mathsf Z^\pm)
   \otimes\delta\mathsf Z^\mp\bigr)
 +\grad_{\gmetric}\delta\mathsf p=0.
\end{equation}
First write $E_i=(dX)^{-1}\eframe_i$, for the pulled-back frame seen as a differential operator and after a fixed rotation, assume
$\Binf=v_{\rm A}e_1$. We then note that the usual energy estimate for the difference of two solutions does not close.  Indeed, for the full
deformation tensor we have
\[
 \bigl(E_i\mathsf Z^{\pm,j}\bigr)_{i,j}
 \in C\bigl(I;H^{s-1}_y\bigr)
 \hookrightarrow
 L^2\bigl(I;L^2_{y_1}H^{s-1}_{y'}\bigr)
 \hookrightarrow
 L^2\bigl(I;L^2_{y_1}L^\infty_{y'}\bigr),
 \qquad
 s-1>\frac{n-1}{2},
\]
but
\[
 L^2_{y_1}L^\infty_{y'}\not\hookrightarrow L^\infty_y.
\]
Thus the transverse $L^\infty$ control is available, but the longitudinal
$L^\infty$ control is missing.  Now writing
\[
 \mathcal E(t):=\frac12\sum_\pm
 \int_{\Real^n}|\delta\mathsf Z^\pm(t,y)|^2dy,
\]
the ordinary energy identity gives
\[
\begin{aligned}
 \partial_t\mathcal E(t)
 &=-\sum_\pm\int_{\Rn}
 E_i\mathsf Z^{\pm,j}
 \delta\mathsf Z^{\mp,i}
 \delta\mathsf Z^{\pm,j}dy\lesssim
 \mathcal E(t)\sum_\pm
 \left\|
 \bigl(E_i\mathsf Z^{\pm,j}\bigr)_{i,j}(t)
 \right\|_{L^\infty_y}.
\end{aligned}
\]
This is the same obstruction as in
\eqref{eq:standardHsintro}, which motivated the whole construction. 

The solution is similar. One can use interaction functionals in the style of Golse-Perthame~\cite[Section~2]{GolsePerthame2013} (take A,B,D,E there to be $A:=\rho_-, \, B:=v_{\rm A}\rho_-,\, D:=\rho_+,\, E:=-v_{\rm A}\rho_+$ defined below). This is a sort of transport version of
the wave null-form pairing. See also Ifrim-Tataru~\cite[Sections~6.3.2 and~6.6.3]{IfrimTataru2025} for related interaction
Morawetz estimates, which was our original starting point. 

First, define the \textit{energy densities} and
\textit{line energies}
\[
 e_\pm:=\frac12|\delta\mathsf Z^\pm|^2,
 \qquad
 \rho_\pm(t,r):=\int_{\Real^{n-1}}e_\pm(t,r,y')dy'.
\]
For $I=[t_0,t_1]$, introduce the \textit{time-localised energy} and the
\textit{interaction flux} norm
\[
 \mathcal E_I:=\sup_{t\in I}\int_{\Rn}(e_++e_-)(t,y)dy,
 \qquad
 \Phi_I:=\left(
 \int_I\int_{\Real}\rho_+(t,r)\rho_-(t,r)dr dt
 \right)^{1/2}.
\]
Letting $t_0\in[0,T^*]$ be a time where the solutions agree, they satisfy
\begin{equation}\label{eq:fluxenergydominationintro}
 \mathcal E_I\lesssim c_I\Phi_I,
\end{equation}
where
\[
 c_I:=
 \sup_{t\in I}\bigl(
   \|\delta Z^+(t)\|_{H^s}+\|\delta Z^-(t)\|_{H^s}\bigr)+\left(\int_I
   \bigl(\|z^+(t)\|_{H^s}^2+\|z^-(t)\|_{H^s}^2\bigr)dt
   \right)^{1/2}.
\]
The product in $\Phi_I$ pairs energies propagating along opposite
Alfv\'en directions. As in \eqref{eq:keycomputation}, this compensates
for the missing longitudinal $L^\infty$ bound.

Indeed, taking the scalar product of
\eqref{eq:overview-conservative-difference} with
$\delta\mathsf Z^\pm$ gives
\begin{equation}\label{eq:densitytransportintro}
 \mathcal A^\mp e_\pm
 =-\delta\mathsf Z^{\pm,j}\left[
   \delta\mathsf Z^{\mp,i}
   E_i(\mathsf Z^{\pm,j}+\delta\mathsf Z^{\pm,j})
   +E_j\delta\mathsf p\right].
\end{equation}
The line energies get translated in opposite
directions with speed $v_{\rm A}$.
To estimate $\Phi_I$, following \cite{GolsePerthame2013} we define the \textit{interaction functional}
\[
 \mathscr I(t):=\int_{r_+>r_-}
 \rho_+(t,r_+)\rho_-(t,r_-)dr_+dr_-,
\]
and reason geometrically. This quantity counts pairs still awaiting an encounter: the $+$ energy lies to the
right and moves left, while the $-$ energy lies to the left and moves
right. For the \emph{free transports}, the pairs cross the diagonal
$r_+=r_-$ with relative speed $2v_{\rm A}$ it follows that
\[
 \mathscr I'(t)
 +2v_{\rm A}\int_{\Real}\rho_+(t,r)\rho_-(t,r)dr=0
\]
and $2v_{\rm A}\Phi_I^2$ can be seen as the flux through this diagonal.
For $\rho_\pm$ freely transported, computing as in \eqref{eq:keycomputation} would give the
product of the line energies, justifying the claim above. For the full equation, several error terms appear in the now zero right-hand side. These once integrated on $I$ can all be bounded by $\mathcal E_Ic_I\Phi_I$ and integrating the flux above since $\mathscr I(t_0)=0$ and $\mathscr I(t_1)\ge0$ gives
\begin{equation}\label{eq:fluxbalanceintro}
 2v_{\rm A}\Phi_I^2
 \le C\mathcal E_Ic_I\Phi_I
 \lesssim c_I^2\Phi_I^2.
\end{equation}
where we used \eqref{eq:fluxenergydominationintro} in the second inequality. The fact that $c_I\to0$ as $t_1\downarrow t_0$ forces $\Phi_I=0$ and hence
$\mathcal E_I=0$ by \eqref{eq:fluxenergydominationintro}. A continuation argument then ensures equality on the whole lifespan and uniqueness follows.

\subsection{What is left open} Some issues remain, and a research direction is natural. 

\emph{A variable initial field and maximal-lifespan}: the Lagrangian reduction itself does not use
constancy.  What one loses are the properties $\Box \mathrm{id}=0$ and $[\Lambda^{\sreg},\mathcal B]=0$.  The displacement equation acquires the
source term 
\[
    \Box w=-(\nabla p)\circ X+\nabla_{B_0}B_0.
\]
A joint functional
calculus for the commuting Alfv\'en operators survives and the spectral theorem for the associated $\Real^{2}$ unitary actions can be used as a replacement for the current wave Fourier symbol.  This however can no longer be simultaneously
diagonalised with $\Lambda^{\sreg}$, so the weights
$\langle D_{\parallel}\rangle^{\preg}m^{\breg}\Lambda^{\sreg}$ of \eqref{eq:Y}
cannot be written as a Fourier multiplier.  Determining the minimal
regularity of $B_0$ for which the various commutator estimates close and the present or a similar construction can be carried out is an open problem.

\emph{GWP in three dimensions}: the main takeaway of this work is that ideal incompressible MHD, at least for small times and with a uniformly non-vanishing magnetic field, can be seen as incompressible Euler in $n-1$ dimensions plus Alfv\'en wave geometry. Pushing this heuristic for $n=3$ given the global well-posedness of the 2D Euler system, it is natural to conjecture that one could continue the solution as long as the magnetic field remains uniformly bounded away from zero. The main difficulty is to control the wave–Hodge formulation beyond the regime in which the metric is close to the flat one. A possible strategy would be that of finding some kind of gauge still allowing for a semilinear treatment of the wave equation despite the large time and possibly large deformation of the geometry.\\

\section{Ideal MHD as a Wave--Hodge System}
\label{sec:wave-hodge}

This section introduces the first-order geometric calculus used below
and derives the wave--Hodge system from smooth ideal MHD.  The basic Hodge
solvability statement is proved here while the first-order derivative estimates needed for the fixed-point argument are deferred to Section~\ref{sec:elliptic}.\\

\textbf{Standing geometric definitions.} Let $X=\mathrm{id}+w:\Real_y^n\to\Real_x^n$ be a possibly time-dependent,
global bi-Lipschitz map, and suppress the time
dependence at each fixed time.  Indices of the label coordinate $y$ are denoted by $a,b$ while indices for the Cartesian coordinates along $X$ by $i,j$. Repeated indices are summed.  We write
\[
 \eframe_i(y):=\left.\partial_{x^i}\right|_{X(y)}
\]
for the canonical Cartesian vector fields along $X$. The indices $i,j$ are raised and lowered with the Euclidean metric,
which is the identity in the frame $\{\eframe_i\}$, so their position is (mostly) cosmetic, while raising or lowering a label index $a,b$ applies $\gmetric$ and does not
leave the components unchanged. Here and below $d=d_y$ is the exterior differential in the labels and acts
componentwise on maps and on vector fields along $X$ written in the frame
$\{\eframe_i\}$. Thus, for any map $F=F^i\eframe_i$, we have
\[
 dF=\partial_aF^i dy^a\otimes\eframe_i.
\]
The map $X$ seen as a change of coordinates induces the metric
\[
    \gmetric:=X^*\langle\cdot,\cdot\rangle=(dX)^\top dX, \qquad \gmetric_{ab}=\partial_aX^i\partial_bX^i,\\
\]
with its associated volume measure
\[
    J:=\det dX, \qquad d\mu:=|J|dy, \qquad \varrho:=\log|J|.
\]
To shorten the notation, we set
\[
 \Kmat^a_i:=((dX)^{-1})^a_i,
\]
to be the matrix representing the inverse differential of $X$. This plays a crucial role as it enters the definition of the canonical frame associated with this change of variables. Indeed, let $\mathcal C_Xf:=f\circ X$ be the precomposition operator, acting componentwise in the Cartesian frame on Eulerian vector fields and producing vector fields along $X$. The geometry in the $\gmetric$ metric is characterised by the pullback of the Cartesian frame,
\begin{equation}\label{eq:component-derivative}
 E_i:=dX^{-1}\eframe_i
 =\mathcal C_X\partial_{x^i}\mathcal C_X^{-1}
 =\Kmat^a_i\partial_a
\end{equation}
Regarding each element of the frame as a differential operator acting on functions, we can write all the standard differential operators in the $\gmetric$-geometry. For a vector field $F=F^i\eframe_i$ along $X$, sufficiently regular (we will address this precisely later), set
\begin{align}\label{eq:divcurlg}
 \Div_{\gmetric}F
 &:=E_iF^i
   =\Kmat^a_i\partial_aF^i,\\
 (\curl_{\gmetric}F)^{ij}
 &:=E_iF^j-E_jF^i
 =\Kmat^a_i\partial_aF^j-\Kmat^a_j\partial_aF^i.
\end{align}
For a scalar function $q$ along $X$ also write
\begin{equation}\label{eq:metricgradient}
    (\mathrm{grad}_\gmetric q)^i=E_i q=\Kmat^a_i\partial_aq.
\end{equation}
For a tensor $T=T^{ji}\eframe_j\otimes\eframe_i$ along $X$, we use the
same second-index convention
\begin{equation}\label{eq:metric-tensor-divergence}
 (\Div_{\gmetric}T)^j:=E_iT^{ji}.
\end{equation}
In particular,
\begin{equation}\label{eq:metric-transport-divergence}
 \Div_{\gmetric}(G\otimes F)
 =F^iE_iG+G\Div_{\gmetric}F.
\end{equation}

\begin{remark}[Geometric interpretation]
The almost-everywhere differential of $X$ is intrinsically a
pullback-bundle-valued $1$-form,
\[
 dX\in L^\infty_{\mathrm{loc}}\big(
 T^*\Real_y^n\otimes X^*T\Real_x^n\big),
 \qquad
 dX_y:T_y\Real_y^n\longrightarrow T_{X(y)}\Real_x^n.
\]
When $X$ is smooth, the pullback bundle $X^*T\Real_x^n$ carries the
Euclidean metric in each fibre and the pullback connection
$\nabla^{X^*}:=X^*\nabla$. The Cartesian frame $\{\eframe_i\}$ along $X$, is parallel, so
\[
 \nabla^{X^*}_V(F^i\eframe_i)=V(F^i)\eframe_i.
\]
Moreover, $dX$ identifies the genuinely pulled-back $\gmetric$-orthonormal frame $\{E_i\}$ of
the tangent space in the label coordinates with the orthonormal Cartesian frame
$\{\eframe_i\}$ along $X$.  Consequently, for smooth $X$,
\[
 \Div_{\gmetric}F
 =\tr_{\gmetric} \left(
   \nabla^{\mathrm{LC},\gmetric}(dX^{-1}F)\right)
 =\sum_i\left\langle
    \nabla^{X^*}_{E_i}F,\eframe_i
   \right\rangle,
\]
where $\nabla^{\mathrm{LC},\gmetric}$ is the Levi-Civita connection associated to $\gmetric$ and thus the divergence in the middle is the Riemannian divergence of the vector field $dX^{-1}F:\Real_y^n\to T\Real_y^n$.  Similarly, the functions $(\curl_{\gmetric}F)^{ij}$ are the components of
$d((dX^{-1}F)^\flat)$ in the orthonormal frame $\{E_i\}$. We remark that this metric curl is distinct from the rowwise curl
$\Curl_y$ used for $\Rn$-valued $1$-forms later on, see \eqref{eq:Curl}.  Smoothness of $X$ is used in the calculation above only to justify the Levi--Civita terms involving
$d\gmetric$. The definitions
\eqref{eq:component-derivative}--\eqref{eq:divcurlg} ask no more than
$\Kmat\in L^\infty$, so no weak Levi--Civita connection or Christoffel symbols need to be introduced.  
\end{remark}

\textbf{The metric Hodge system.} Given a metric $\gmetric$ as above, we now study the Hodge system for a vector field $F$ along $X$,
\begin{equation}\label{eq:invdiv}
    \begin{cases}
        \curl_{\gmetric}F=0,\\
     \Div_{\gmetric}F=f.
\end{cases}
\end{equation}
The following shows the existence and uniqueness of its solutions under sufficient integrability assumptions on the right-hand side $f$. Note that, the assumed integrability removes possible harmonic components. The equation would otherwise determine $F$ only up to $(\nabla H)\circ X$ with $H$ harmonic.

\begin{lemma}[canonical solvability of the metric Hodge system]
\label{lem:hodge-solvability}
Let $X:\Real_y^n\to\Real_x^n$ be as above.  Fix $1<p<n$ and put
\[
 p^*:=\frac{np}{n-p}.
\]
For every $f\in L^p(\mu)$ there is a unique vector field $F=F^i\eframe_i$ along $X$, satisfying \eqref{eq:invdiv} in the conjugated
distributional sense, whose Cartesian components belong to
$L^{p^*}(\mu;\Rn)$. By this we mean that $u:=F\circ X^{-1}$ satisfies
\[
 \partial_{x^i}u^j-\partial_{x^j}u^i=0,
 \qquad
 \partial_{x^i}u^i=f\circ X^{-1}
 \quad\text{in }\mathcal D'(\Real_x^n).
\]
The solution, which we call the canonical Hodge representative, is
\[
    F:=\mathcal P_{\gmetric}f
\]
where
\begin{equation}\label{eq:pressure-gradient}
 \mathcal P_{\gmetric}f
 :=\mathcal C_X\nabla\Delta_x^{-1}\mathcal C_X^{-1}f, \qquad \mathcal{C}_Xg=g\circ X,
\end{equation}
and satisfies
\begin{equation}\label{eq:canonical-hls}
 \|\mathcal P_{\gmetric}f\|_{L^{p^*}(\mu)}
 \le C_{n,p}\|f\|_{L^p(\mu)}.
\end{equation}
In particular given $n\ge2$, the result holds for
\begin{equation}\label{eq:rHLS}
    p=r_{\mathrm{HLS}}:=\frac{4n}{n+4} \qquad \text{and} \qquad p^*=r_{\mathrm{HLS}}^*=4
\end{equation}
as one has $1<r_{\mathrm{HLS}}<n$. 
\end{lemma}

\begin{proof}
The change of variables $x=X(y)$ makes
$\mathcal C_X:L^r(dx)\to L^r(\mu)$ an isometry.  Define
$\tilde f:=\mathcal C_X^{-1}f$ and $u:=\nabla\Delta_x^{-1}\tilde f$, the
Hardy--Littlewood--Sobolev inequality \cite{stein} gives
\begin{equation}\label{eq:flat-hls}
 \|u\|_{L^{p^*}(dx)}\le C_{n,p}\|\tilde f\|_{L^p(dx)},
\end{equation}
and the identities in the flat geometry
\[
 \partial_{x^i}u^j-\partial_{x^j}u^i=0,
 \qquad
 \partial_{x^i}u^i=\tilde f
\]
hold in distributions.  Thus 
\[
    F:=\mathcal C_Xu=\mathcal C_X\nabla\Delta_x^{-1}\tilde f=\mathcal C_X\nabla\Delta_x^{-1}\mathcal C_X^{-1}f=\mathcal P_{\gmetric}f
\]
solves
\eqref{eq:invdiv} in the stated sense, and the isometry property of $\mathcal C_X$ gives
\eqref{eq:canonical-hls}.

For uniqueness, let $H$ be the difference of two $L^{p^*}(\mu)$ solutions
and put $\tilde u:=H\circ X^{-1}$.  Then
$\tilde u\in L^{p^*}(\Rn)$ and since $\Div_x\tilde u=\curl_x\tilde u=0$, $\tilde u$ is harmonic. By Weyl's lemma, $\tilde u$ is smooth, and the mean-value property together
with $p^*<\infty$ gives
\[
    |\tilde u(x)|\le|B_R|^{-1}\|\tilde u\|_{L^1(B_R(x))}
\le c_nR^{-n/p^*}\|\tilde u\|_{L^{p^*}}\to0
\]
as $R\to\infty$, hence
$\tilde u=0$. Here, integrability kills the Hodge kernel and ensures uniqueness. 
\end{proof}

\begin{remark}[Numerology]\label{rem:pressure-L4}
Our aim is to reconstruct a solution with $v,b,p_{\mathrm R}\in C_tH^s$.
In particular, the pressure gradient must belong to
$C_tH^{s-1}\subset C_tL^2$. However, the estimates coming from the
fixed-point argument do not control the $L^2$ norm of the velocity
or of the force $F$ obtained through the solvability statement above.

The choice of exponents in \eqref{eq:rHLS} is convenient because it solves this issue independently of the dimensions. We use the notation from the previous lemma and motivate it as follows.
The exponents must satisfy the following requirements:
\begin{enumerate}
\item[(i)] The Hodge estimate \eqref{eq:canonical-hls} gives
$F\in L^{p^*}$ from $f\in L^p$, where
\[
 \frac1{p^*}=\frac1p-\frac1n,
 \qquad 1<p<n.
\]
Since the r.h.s.~will be quadratic in quantities controlled in $H^{s-1}$,
we are forced to require $H^{s-1}\hookrightarrow L^{2p}$.

\item[(ii)] The Duhamel formula
\eqref{eq:potential-reconstruction} then gives
$v=(\partial_t w)\circ X^{-1}\in L^{p^*}$,
since $v_0\in H^s\subset L^{p^*}$. 

\item[(iii)] To obtain $\nabla p_{\mathrm R}\in L^2$ by
H\"older's inequality as in \eqref{eq:return-pressure-L2},
we need the quadratic terms $v\cdot\nabla v$ and
$b\cdot\nabla b$ in $L^2$. It follows that the derivatives need to be controlled in $L^a$, where the value of $a$ is dictated by~(ii) via 
\[
 \frac1{p^*}+\frac1a=\frac12,
 \qquad a=\frac{2p^*}{p^*-2}.
\]
This will amount to
$H^{s-1}\hookrightarrow L^a$.
\end{enumerate}
Since $s-1>(n-1)/2$ the Sobolev embedding gives 
\[
 H^{s-1}\hookrightarrow L^a
 \qquad\text{for every }2\le a\le2n.
\]
Combining the three constraints, we obtain the lower bound
\[
 p^*\ge\frac{2n}{n-1}.
\]
Indeed, this ensures $a\le2n$, while $1<p<n$
gives $2<2p<2n$. Since
\[
 \frac{2n}{n-1}\le4
 \qquad\text{for every }n\ge2,
\]
with equality when $n=2$. The choice $p^*=4$ then works in any dimension and regularity, gives $a=4$ and
\[
 p=r_{\mathrm{HLS}}=\frac{4n}{n+4},
\]
as in \eqref{eq:rHLS}.

The pressure estimate \eqref{eq:return-pressure-L2} and the
identification \eqref{eq:return-force} then give $F\in L^\infty_tL^2$.
The velocity equation supplies the missing $C_tL^2$ control,
which combines with the spatial derivative estimates to complete
the Eulerian reconstruction in $C_tH^s$. The Riesz formula finally yields
$p_{\mathrm R}\in C_tH^s$, as required.
\end{remark}

\textbf{Definition of the off-constraint wave--Hodge auxiliary system.} We set some more notation and define the auxiliary wave--Hodge reformulation of \eqref{eq:mhd} without imposing $\det[d X]=1$.  This is the generalised form used later in the fixed-point argument. For scalar functions, set
\begin{equation}\label{eq:Q0}
 Q_0(f,g)
 :=\tfrac12\big(\mathcal A^+f \mathcal A^-g
                 +\mathcal A^-f \mathcal A^+g\big)
 =\partial_tf \partial_tg-\mathcal Bf \mathcal Bg.
\end{equation}
For an $\Rn$-valued $1$-form $\Ufield=(\Ufield_a^j)$ such that
$\mathrm{Id}+\Ufield$ is invertible, the full Hodge r.h.s. in the construction will be
\begin{equation}\label{eq:datum}
 \Kmat[\Ufield]:=(\mathrm{Id}+\Ufield)^{-1},
 \qquad
 \mathcal Q[\Ufield]
 :=-\Kmat[\Ufield]^a_i\Kmat[\Ufield]^b_j
 Q_0(\Ufield_a^j,\Ufield_b^i).
\end{equation}
If $\Ufield$ is exact and a primitive $w$ with $dw=\Ufield$ exists such that $X=\mathrm{id}+w$ is a
global bi-Lipschitz map, write
\begin{equation*}
 \gmetric[\Ufield]:= (\mathrm{Id}+\Ufield)^\top(\mathrm{Id}+\Ufield), \qquad
 J[\Ufield]:=\det(\mathrm{Id}+\Ufield), \qquad d\mu[\Ufield]:=|J[\Ufield]|dy.
\end{equation*}
for the associated metric Jacobian and volume form. Translations in the choice of $w$ do not affect these quantities.  Whenever
$\mathcal Q[\Ufield]\in L^{r_{\mathrm{HLS}}}(\mu[\Ufield])$, let
\begin{equation}\label{eq:FU}
    F[\Ufield]
 :=\mathcal P_{\gmetric[\Ufield]}\mathcal Q[\Ufield]
\end{equation}
denote the unique $L^4(\mu[\Ufield])$ solution of the Hodge system
\eqref{eq:invdiv} with $f=\mathcal Q[\Ufield]$, where the operators in that
system, namely \eqref{eq:divcurlg}, are formed from
$E_i[\Ufield]=\Kmat[\Ufield]^a_i\partial_a$.  Lemma~\ref{lem:hodge-solvability} together with Lemma \ref{lem:weak-pullback} (ii), to handle this more generalised definition, 
give its existence, uniqueness and Hardy--Littlewood--Sobolev bound. Its
first-order derivative estimates are proved in Subsection~\ref{ssec:pressure-hessian}.

In the notation and conventions above, we say that a map $w:\Real_y^n\to \Real_x^n$ is a solution of the wave--Hodge system if
\begin{subequations}\label{eq:auxiliary-system}
\begin{align}
&
\begin{cases}
    \Box w=-F[dw],\\
    (w,\partial_tw)|_{t=0}=(0,v_0),
\end{cases}\\
 &
 \begin{cases}
  \curl_{\gmetric[dw]}F[dw]=0,\\
  \Div_{\gmetric[dw]}F[dw]=\mathcal Q[dw].
 \end{cases}
 &\label{eq:force-hodge}
\end{align}
\end{subequations}
The sense in which this is intended will be specified depending on the context and we sometimes refer to \eqref{eq:auxiliary-system} as the auxiliary Lagrangian problem.\\ 

\textbf{From Ideal MHD to the wave--Hodge system.} In the following two lemmas, we first record the properties of the weak geometric calculus induced by a bi-Lipschitz map $X$. We then show that a smooth solution $(v,B)$ of \eqref{eq:mhd} solves the auxiliary wave--Hodge system \eqref{eq:auxiliary-system} with $w=X-\textrm{id}$ and $X$ the flow map of the velocity $v$. Note that the information $\det d X=1$ seems to be missing from \eqref{eq:auxiliary-system}, but it will be recovered dynamically from divergence-free initial data (see \S \ref{ssec:constraint}). Subsection~\ref{ssec:eulerian} later proves that an auxiliary solution for
which the volume constraint has been recovered determines a solution of
ideal MHD, giving the opposite implication.

\begin{lemma}[weak pullback calculus]\label{lem:weak-pullback}
Let $X$ be a global bi-Lipschitz map as above.
\begin{enumerate}
\item[(i)] Each $E_i$ is skew-adjoint on $L^2(\mu)$. In particular
  \[
   \int_{\Real_y^n}(E_if)\phi d\mu=-\int_{\Real_y^n}f(E_i\phi)d\mu
  \]
  whenever $f,\phi,E_if,E_i\phi\in L^2(\mu)$.
\item[(ii)] Let $F=F^i\eframe_i$ be a vector field along $X$ whose Cartesian components belong to $L^1_{\rm loc}(\Real^n_y)$. If
  $f,c^{ij}\in L^1_{\rm loc}(\mu)$, then
  \[
   \Div_{\gmetric}F=f,
   \qquad
   (\curl_{\gmetric}F)^{ij}=c^{ij}
  \]
  in the conjugated distributional sense of Lemma \ref{lem:hodge-solvability} if and only if
  \[
   \partial_a(J\Kmat_i^aF^i)=Jf,
   \qquad
   \partial_a\!\left(J\Kmat_i^aF^j-J\Kmat_j^aF^i\right)=Jc^{ij}
   \quad\text{in }\mathcal D'(\Real^n_y).
  \]
  Equivalently, for every $\varphi\in C_c^\infty(\Real^n_y)$,
  \begin{equation}\label{eq:conservativeform}
   \int f\varphi d\mu=-\int F^i\Kmat_i^a\partial_a\varphi d\mu,\qquad
   \int c^{ij}\varphi d\mu=-\int\bigl(F^j\Kmat_i^a-F^i\Kmat_j^a\bigr)
          \partial_a\varphi d\mu.
  \end{equation}
  Similarly, with the convention in \eqref{eq:metric-tensor-divergence}, whenever
  $H:=\Div_{\gmetric}T\in L^1_{\rm loc}(\mu)$, we have
  \begin{equation}\label{eq:conservative-tensor-divergence}
   JH^j
   =\partial_a(J\Kmat_i^aT^{ji})
   \quad\text{in }\mathcal D'(\Real_y^n)
  \end{equation}
  for every locally integrable tensor $T$.
\item[(iii)] If $\mathcal C_X^{-1}F\in W^{1,p}_{\rm loc}$ for some
  $1\le p\le\infty$, then 
  \begin{equation}\label{eq:biLipchainrule}
      \partial_aF^i=\partial_aX^jE_jF^i,
  \end{equation}
  and the divergence and curl in
  (ii) agree almost everywhere with the pointwise formulas 
  \[
    \Kmat^a_i\partial_aF^i \qquad \text{and} \qquad \Kmat^a_i\partial_aF^j-\Kmat^a_j\partial_aF^i.
  \]
\item[(iv)] The operators $E_i$ satisfy:
  \[
   E_iE_j=E_jE_i.
  \]
  whenever both compositions are well defined.
\end{enumerate}
\end{lemma}

\begin{proof} We prove each item separately. \emph{(i),(iv)} Follow immediately from the definitions.

\emph{(ii)} Put $u:=F\circ X^{-1}$. Assume first that
$\partial_{x^i}u^i=f\circ X^{-1}$ in $\mathcal D'(\Real_x^n)$.
For
$\varphi\in C_c^\infty(\Real_y^n)$ one can use
$\psi:=\varphi\circ X^{-1}$ as a test function. This is admissible as the identity above can be extended to Lipschitz test functions. The chain rule and change of variables give
\[
 \int f\varphi d\mu
 =\int (f\circ X^{-1})\psi dx =-\int u^i\partial_{x^i}\psi dx
 =-\int F^i\Kmat_i^a\partial_a\varphi d\mu.
\]
Since $d\mu=|J|dy$ and $J$ has constant sign, this is equivalent to $\partial_a(J\Kmat_i^aF^i)=Jf$.
Reversing the argument proves the opposite implication. The identities for the curl are analogous, and the proof of \eqref{eq:conservativeform} is complete. Note that, here no derivative of
$\Kmat$ is needed and we never multiply a distribution by $J^{-1}$.

\emph{(iii)} Put $u:=\mathcal C_X^{-1}F$. The Sobolev chain rule gives
\[
 \partial_aF^i
 =\partial_aX^j\bigl((\partial_{x^j}u^i)\circ X\bigr)
 =\partial_aX^jE_jF^i.
\]
Multiplying by $\Kmat$ gives
\[
 \Kmat_i^a\partial_aF^j=E_iF^j.
\]
Taking the trace and the antisymmetric part and using~(i) this gives the pointwise
representatives of the divergence and curl appearing in~(ii) and the last claim follows.
\end{proof}

We now turn to the reduction of smooth ideal MHD to the wave--Hodge system \eqref{eq:auxiliary-system}.

\begin{lemma}[Wave--Hodge reduction]
\label{lem:wave-hodge-reduction}
Let $X=X(t,y)$ be smooth in spacetime and $w$, $\Kmat$, $\mathcal{Q}$ as in \eqref{eq:datum} above.  
 
\emph{(i) Algebraic pressure identity.}
\begin{equation}\label{eq:datum-geometric}
 \mathcal Q[dw]
 =-\Kmat^a_i\Kmat^b_j
   Q_0(\partial_aw^j,\partial_bw^i)=-\tr \left(
   \Kmat\mathcal A^+(dw)\Kmat\mathcal A^-(dw)\right).
\end{equation}
 
\emph{(ii) Smooth MHD reduction.}
Assume in addition that $X$ is the Lagrangian flow of the velocity of a smooth solution of
\eqref{eq:mhd}, with $X(0)=\mathrm{id}$ and initial magnetic field $\Binf$.
If
$\mathcal Q[dw]\in L^{r_{\mathrm{HLS}}}(\mu)$ and $(\nabla p)\circ X\in L^4(\mu)$ then
\[
 (\nabla p)\circ X=F[dw]=\mathcal{P}_{\gmetric[dw]}\mathcal Q[dw],
\]
and $(w,F[dw])$ solves \eqref{eq:auxiliary-system}. In particular, this holds when
$v,b\in C([0,T];H^\infty)$ and $p$ is
as in \eqref{eq:return-pressure}.
\end{lemma}
 
\begin{proof}
For the algebraic identity let $\Ufield=dw$.  The two terms in the definition
of $Q_0$ after exchanging $(a,i)$ with $(b,j)$ are equal. Therefore
\[
 \Kmat^a_i\Kmat^b_j
   (\mathcal A^-\Ufield)_a^j
   (\mathcal A^+\Ufield)_b^i=
   \Kmat^b_j\Kmat^a_i
   (\mathcal A^-\Ufield)_b^i
   (\mathcal A^+\Ufield)_a^j=
   \Kmat^a_i\Kmat^b_j
   (\mathcal A^+\Ufield)_a^j
   (\mathcal A^-\Ufield)_b^i
\]
and \eqref{eq:datum-geometric} follows directly from the definition \eqref{eq:datum}.
 
It remains to derive \eqref{eq:auxiliary-system} from ideal MHD.  Let $(v,B,p)$ be a
smooth solution and let $X$ solve $\partial_tX=v\circ X$ and $X(0,\cdot)=\mathrm{id}$.
 
\emph{The frozen-in law.} Since $\Div_xv=\Div_xB=0$, the induction
equation reads as Lie
transport of $B$ along the flow of $v$. In particular, $B=X_*\Binf$, that is
\begin{equation}\label{eq:frozen-in}
 B\circ X=dX\Binf=\mathcal BX.
\end{equation}

\emph{The wave equation.} The chain rule and \eqref{eq:frozen-in} give
\begin{align*}
 (\partial_tv+\nabla_vv)\circ X
 &=\partial_t(v\circ X)=\partial_t^2X,\\
 (\nabla_BB)\circ X
 &=\big((\nabla_xB)\circ X\big)dX\Binf
 =\mathcal B(B\circ X)=\mathcal B^2X
\end{align*}
and from the momentum equation we obtain
$\Box X=-(\nabla p)\circ X$. Since $\Binf$ is constant
$\Box\mathrm{id}=0$ and hence
\begin{equation}\label{eq:waverev}
    \Box w=-(\nabla p)\circ X,
 \qquad
 w(0,\cdot)=0,
 \qquad
 \partial_tw(0,\cdot)=v_0.
\end{equation}

\emph{The Hodge equations.} By \eqref{eq:component-derivative},
\[
 E_i\bigl((\partial_{x^j}p)\circ X\bigr)
 =(\partial_{x^i}\partial_{x^j}p)\circ X.
\]
The symmetry of the Hessian therefore gives
\begin{equation}\label{eq:divcurlrev}
    \curl_{\gmetric}\bigl((\nabla p)\circ X\bigr)=0, \qquad \Div_{\gmetric}\bigl((\nabla p)\circ X\bigr)=(\Delta_xp)\circ X.
\end{equation}
Taking $\Div_x$ of the momentum equation and using
$\Div_xv=\Div_xB=0$ we obtain
\[
 \Delta_xp=(\partial_{x^i}B^j)(\partial_{x^j}B^i)
           -(\partial_{x^i}v^j)(\partial_{x^j}v^i).
\]
Since $v\circ X=\partial_tw$ and $B\circ X=\Binf+\mathcal Bw$,
the chain rule gives
\[
 (\partial_{x^i}v^j)\circ X=\Kmat^a_i\partial_t\partial_aw^j,
 \qquad
 (\partial_{x^i}B^j)\circ X=\Kmat^a_i\mathcal B\partial_aw^j
\]
and from \eqref{eq:divcurlrev},\eqref{eq:datum} we can finally compute
\begin{align*}
 \Div_{\gmetric}\bigl((\nabla p)\circ X\bigr)&=(\Delta_xp)\circ X\\
 &=\Kmat^a_i\Kmat^b_j\Big[
     \big(\mathcal B\partial_aw^j\big)\big(\mathcal B\partial_bw^i\big)
    -\big(\partial_t\partial_aw^j\big)\big(\partial_t\partial_bw^i\big)\Big]\\
 &=-\Kmat^a_i\Kmat^b_jQ_0\bigl(\partial_aw^j,\partial_bw^i\bigr)
  =\mathcal Q[dw]
\end{align*}
The assumed $L^4(\mu)$ integrability of
$(\nabla p)\circ X$ and uniqueness in Lemma~\ref{lem:hodge-solvability}
then shows $(\nabla p)\circ X=F[dw]$. Collecting \eqref{eq:divcurlrev}, \eqref{eq:waverev} and this, we deduce \eqref{eq:auxiliary-system}.
\end{proof}

\section{Wave-Adapted Anisotropic Spaces}
\label{sec:anisotropic-spaces}

This section proves the three fixed-time estimates used later: the
anisotropic embedding, the trace theorem and the product estimates. These are anisotropic adaptations of standard properties of Sobolev spaces and classical
Bourgain spaces (see for instance
\cite{Tao2006Dispersive}). Unsurprisingly, the arguments will be of the form: bound a convolution kernel, change of variables and apply Tonelli's theorem.  Zhang established the first version of these in the present MHD
setting in \cite{Zhang2026}. The formulations and
proofs given here are slight generalisations needed to deal with the different regimes in Theorem \ref{thm:main}. Without further mention, the statements below and in the following sections are first proved for Schwartz functions and eventually extended by continuity.

\subsection{Embeddings and traces}

\begin{lemma}[anisotropic Sobolev embedding]\label{lem:Linftyembed}
Let
\begin{equation}\label{eq:Y-embed-range}
 \sreg>\frac{n-1}{2},\qquad \preg\ge0,\qquad \sreg+\preg>\frac n2 .
\end{equation}
Then every $f\in Y^{\sreg,\preg}(\Rn)$ satisfies
\begin{equation}\label{eq:embed-frac}
 \|f\|_{L^\infty(\Rn)}
 \le
 \frac{C}
 {(\sreg+\preg-\frac n2)^{1/2}}
 \|f\|_{Y^{\sreg,\preg}},
\end{equation}
where $C=C(n,\sigma,\breg,v_{\rm A})$ depends on the standing parameters.
\end{lemma}

\begin{proof}
By Fourier inversion and Cauchy--Schwarz,
\[
 \|f\|_{L^\infty}
 \le C_n\mathcal J(\sreg,\preg)^{1/2}\|f\|_{Y^{\sreg,\preg}},
 \qquad
 \mathcal J(\sreg,\preg)
 :=\int_{\Rn}
 \langle\xi\rangle^{-2\sreg}
 \langle\lambda(\xi)\rangle^{-2\preg} d\xi .
\]
Write $\xi=\mathbf e_1\xi_1+\xi'$, after rotating $\Binf$ to $v_{\rm A}\mathbf e_1$ and using
$\langle v_{\rm A}\xi_1\rangle\simeq_{v_{\rm A}}\langle\xi_1\rangle$, the change
$\xi'=\langle\xi_1\rangle\zeta$ gives
\[
 \mathcal J(\sreg,\preg)
 \simeq_{v_{\rm A},\preg}
 c_\perp(n,\sreg)
 \int_{\Real}\langle\xi_1\rangle^{n-1-2\sreg-2\preg} d\xi_1,
 \qquad
 c_\perp(n,\sreg)
 :=\int_{\Real^{n-1}}\langle\zeta\rangle^{-2\sreg} d\zeta .
\]
The two factors are finite precisely under \eqref{eq:Y-embed-range} and the remaining integral in $\mathcal J(\sreg,\preg)$ is
$O((\sreg+\preg-n/2)^{-1})$.
This proves \eqref{eq:embed-frac}.  
\end{proof}
\begin{corollary}[traces in the $Y$ scale]\label{cor:trace-Y}
Let 
\[
    \breg>\frac12, \qquad \sreg\ge0, \qquad \preg\in\Real.
\]
Every $f\in Y^{\sreg,\preg}_{\breg}$ has a representative in
$C(\Real;Y^{\sreg,\preg})$.
If $f\in H^{\sreg,\preg}_{\breg}$, then in addition
$f\in C^1(\Real;Y^{\sreg,\preg-1})$ and
\begin{equation}\label{eq:trace-class}
 \sup_{t\in\Real}\bigl(
   \|f(t)\|_{Y^{\sreg,\preg}}
   +\|\partial_tf(t)\|_{Y^{\sreg,\preg-1}}\bigr)
 \le C_{\breg}\|f\|_{H^{\sreg,\preg}_{\breg}}.
\end{equation}
\end{corollary}

\begin{proof}
At fixed $\xi$, Fourier inversion in time and Cauchy--Schwarz give
\[
 |\widehat f(t,\xi)|^2
 \lesssim
 \left(\int_{\Real}m^{-2\breg}(\tau,\lambda) d\tau\right)
 \left(\int_{\Real}m^{2\breg}(\tau,\lambda)
       |\widetilde f(\tau,\xi)|^2 d\tau\right).
\]
Since $\breg>\frac12$,
\[
 \sup_{\lambda\in\Real}
 \int_{\Real}m^{-2\breg}(\tau,\lambda) d\tau
 \le 2\int_{\Real}\langle r\rangle^{-2\breg} dr<\infty.
\]
Multiplying by
$\langle\xi\rangle^{2\sreg}\langle\lambda\rangle^{2\preg}$
and integrating in $\xi$ therefore yields
\begin{equation}\label{eq:traceinY}
    \sup_{t\in\Real}\|f(t)\|_{Y^{\sreg,\preg}}
 \le C_{\breg}\|f\|_{Y^{\sreg,\preg}_{\breg}}.
\end{equation}
To obtain \eqref{eq:trace-class}, use
$\Omega/m^{\breg}=\langle|\tau|+|\lambda|\rangle$ to write
\[
 \langle\lambda\rangle^{\preg}m^{\breg}
 \le\langle\lambda\rangle^{\preg-1}\Omega, \qquad|\tau|\langle\lambda\rangle^{\preg-1}m^{\breg}
 \le\langle\lambda\rangle^{\preg-1}\Omega.
\]
These inequalities hold for every real $\preg$, since the common factor
$\langle\lambda\rangle^{\preg-1}$ is positive and finite.
Multiplying by $\langle\xi\rangle^{\sreg}|\widetilde f|$ and taking
the $L^2_{\tau,\xi}$ norm gives, by \eqref{eq:Hsymbol},
\[
 \|f\|_{Y^{\sreg,\preg}_{\breg}}
 \le\|f\|_{H^{\sreg,\preg}_{\breg}},
 \qquad
 \|\partial_tf\|_{Y^{\sreg,\preg-1}_{\breg}}
 \le\|f\|_{H^{\sreg,\preg}_{\breg}}.
\]
Applying the preceding trace estimate to $f$ and $\partial_tf$ proves
the bound in \eqref{eq:trace-class}. The continuity property follows directly from \eqref{eq:traceinY} applied to $f(\cdot)-f(\cdot+h)$, with dominated convergence on the Fourier side proving continuity of translations in $\|\cdot\|_{Y^{\sreg,\preg}_{\breg}}$. One can argue similarly for differentiability.
\end{proof}

\subsection{Products}

\begin{proposition}[products in the $Y$ scale]\label{prop:tame-product}
Let $\preg\ge0$ and
\begin{equation}\label{eq:Y-product-range}
 \sreg>\frac{n-1}{2},\qquad \sreg+\preg>\frac n2.
\end{equation}
Then $Y^{\sreg,\preg}$ is a Banach algebra and is a multiplier on $H^{\sreg}$:
\begin{equation}\label{eq:Y-product}
 \|fg\|_{Y^{\sreg,\preg}}\le C\|f\|_{Y^{\sreg,\preg}}\|g\|_{Y^{\sreg,\preg}},
 \qquad
 \|fg\|_{H^{\sreg}}\le C\|f\|_{H^{\sreg}}\|g\|_{Y^{\sreg,\preg}}.
\end{equation}
The second estimate also holds with $H^\varsigma$ in place of $H^{\sreg}$ for
$0\le\varsigma\le\sreg$.  In particular, at $(\sreg,\preg)=(s-1,1)$,
\begin{equation}\label{eq:tame-sp}
 \|fg\|_{H^\varsigma}\le C\|f\|_{H^\varsigma}\|g\|_{Y^{s-1,1}}
 \quad(0\le\varsigma\le s-1).
\end{equation}
\end{proposition}

\begin{proof}
Use the coordinates from the proof of Lemma~\ref{lem:Linftyembed},
writing $\xi=(\xi_1,\xi')$ and $\eta=(\eta_1,\eta')$, so that
$\langle\lambda(\xi)\rangle\simeq_{v_{\rm A}}\langle\xi_1\rangle$.

We begin with some preliminary computations. By the weighted arithmetic–geometric mean inequality,
\[
 \frac{1}{\langle\eta_1\rangle^{2\sreg-n+1}
              \langle\xi_1-\eta_1\rangle^{2\preg}}
 \le
 \langle\eta_1\rangle^{n-1-2\sreg-2\preg}
 +\langle\xi_1-\eta_1\rangle^{n-1-2\sreg-2\preg}.
\]
Using the transverse integration from the proof of
Lemma~\ref{lem:Linftyembed}, integrating this bound in $\eta_1$
and translating the second term, we obtain
\begin{equation*}\label{eq:convineq}
 \sup_{\xi_1\in\Real}
 \int_{\Rn}
 \frac{d\eta}{\langle\eta\rangle^{2\sreg}
                 \langle\xi_1-\eta_1\rangle^{2\preg}}
 \lesssim
 \int_{\Rn}
 \frac{d\eta}{\langle\eta\rangle^{2\sreg}
                 \langle\eta_1\rangle^{2\preg}}<\infty.
\tag{*}
\end{equation*}
We claim the following two bounds:
\begin{equation}\label{eq:weight-convolution}
 \int_{\Rn}
 \frac{d\eta}{\langle\eta\rangle^{2\sreg}
                 \langle\eta_1\rangle^{2\preg}
                 \langle\xi-\eta\rangle^{2\sreg}
                 \langle\xi_1-\eta_1\rangle^{2\preg}}
 \lesssim
 \frac{1}{\langle\xi\rangle^{2\sreg}
             \langle\xi_1\rangle^{2\preg}},
\end{equation}
\begin{equation}\label{eq:mixed-weight-convolution}
 \int_{\Rn}
 \frac{d\eta}{\langle\eta\rangle^{2\sreg}
                 \langle\xi-\eta\rangle^{2\sreg}
                 \langle\xi_1-\eta_1\rangle^{2\preg}}
 \lesssim\frac{1}{\langle\xi\rangle^{2\sreg}}.
\end{equation}
To prove \eqref{eq:mixed-weight-convolution}, multiply and divide the integrand by
$\langle\xi\rangle^{2\sreg}$ and use
\[
 \langle\xi\rangle^{2\sreg}
 \lesssim
 \langle\eta\rangle^{2\sreg}
 +\langle\xi-\eta\rangle^{2\sreg}.
\]
The resulting two integrals are bounded by \eqref{eq:convineq}, one
directly and the other at $\xi_1=0$ after the change of variables
$\eta\mapsto\xi-\eta$. For \eqref{eq:weight-convolution}, multiply the integrand by
$\langle\xi_1\rangle^{2\preg}$ and use
\[
 \langle\xi_1\rangle^{2\preg}
 \lesssim
 \langle\eta_1\rangle^{2\preg}
 +\langle\xi_1-\eta_1\rangle^{2\preg}.
\]
Estimate~\eqref{eq:mixed-weight-convolution}, together with
the same change of variables, bounds both terms.

By the Fourier convolution formula, Cauchy--Schwarz and
\eqref{eq:weight-convolution},
\[
\begin{aligned}
 &\langle\xi\rangle^{2\sreg}\langle\xi_1\rangle^{2\preg}
   |\widehat{fg}(\xi)|^2\\
 &\quad\leq
   \langle\xi\rangle^{2\sreg}\langle\xi_1\rangle^{2\preg}
   \left(\int_{\Rn}
     |\widehat f(\eta)|\,|\widehat g(\xi-\eta)|\,d\eta
   \right)^2\\
 &\quad\le
   \langle\xi\rangle^{2\sreg}\langle\xi_1\rangle^{2\preg}
   \left(\int_{\Rn}
     \frac{d\eta}{
       \langle\eta\rangle^{2\sreg}
       \langle\eta_1\rangle^{2\preg}
       \langle\xi-\eta\rangle^{2\sreg}
       \langle\xi_1-\eta_1\rangle^{2\preg}}
   \right)\\
 &\qquad\quad\times
   \left(\int_{\Rn}
     \langle\eta\rangle^{2\sreg}\langle\eta_1\rangle^{2\preg}
     |\widehat f(\eta)|^2
     \langle\xi-\eta\rangle^{2\sreg}
     \langle\xi_1-\eta_1\rangle^{2\preg}
     |\widehat g(\xi-\eta)|^2\,d\eta
   \right)\\
 &\quad\lesssim
   \int_{\Rn}
     \langle\eta\rangle^{2\sreg}\langle\eta_1\rangle^{2\preg}
     |\widehat f(\eta)|^2
     \langle\xi-\eta\rangle^{2\sreg}
     \langle\xi_1-\eta_1\rangle^{2\preg}
     |\widehat g(\xi-\eta)|^2\,d\eta.
\end{aligned}
\]

Integrating in $\xi$ and applying Tonelli's theorem proves the algebra
estimate in \eqref{eq:Y-product}.
Using \eqref{eq:mixed-weight-convolution} gives
the same calculation with the factors
$\langle\xi_1\rangle^{2\preg}$ and
$\langle\eta_1\rangle^{2\preg}$ omitted. This proves the multiplier
estimate in \eqref{eq:Y-product}.

Finally, the hypotheses \eqref{eq:Y-product-range} coincide with the
embedding range \eqref{eq:Y-embed-range}, so \eqref{eq:embed-frac} gives
\[
 \|fg\|_2\le\|f\|_2\|g\|_\infty
 \lesssim\|f\|_2\|g\|_{Y^{\sreg,\preg}}.
\]
For fixed $g$, interpolate the preceding $L^2$ bound with the
$H^{\sreg}$ multiplier estimate in \eqref{eq:Y-product}.
Since $[L^2,H^{\sreg}]_\theta=H^{\theta\sreg}$, multiplication by $g$
satisfies the same bound on $H^\varsigma$ for $0\le\varsigma\le\sreg$.
Taking $(\sreg,\preg)=(s-1,1)$ gives \eqref{eq:tame-sp}.
\end{proof}

\begin{remark}[interpolation in the parallel index]\label{rem:Y-interpolation}
At fixed $\sreg$, the scale $Y^{\sreg,\preg}$ is, on the Fourier side, a
weighted $L^2$ scale whose weight
$\langle\xi\rangle^{2\sreg}\langle\lambda\rangle^{2\preg}$ depends
multiplicatively on $\preg$.  The Stein--Weiss interpolation theorem for
weighted $L^2$ spaces, see \cite[Theorem~5.5.3]{bergh}, therefore gives, with
equal norms,
\[
 \big[Y^{\sreg,\preg_0},Y^{\sreg,\preg_1}\big]_{\theta}
 =Y^{\sreg,(1-\theta)\preg_0+\theta\preg_1},
 \qquad 0\le\theta\le1 .
\]
In particular, a linear operator that is bounded on $Y^{\sreg,0}=H^{\sreg}$
and on $Y^{\sreg,1}$ is bounded on every $Y^{\sreg,\preg}$ with
$0\le\preg\le1$, with the interpolated operator norm. Now assume
that $\sigma>(n-1)/2$ and $\sigma+1>n/2$. Multiplication by an element of
$Y^{\sigma,1}$ is bounded on
$Y^{\sigma,0}=H^{\sigma}$ by \eqref{eq:tame-sp} and on $Y^{\sigma,1}$ by
\eqref{eq:Y-product}, hence
\begin{equation}\label{eq:Y-multiplier-interp}
 \|fg\|_{Y^{\sigma,\preg}}
 \le C\|f\|_{Y^{\sigma,\preg}}\|g\|_{Y^{\sigma,1}},
 \qquad 0\le\preg\le1 .
\end{equation}
This is how we will deal with the intermediate parameters in the elliptic estimates of
Section~\ref{sec:elliptic}.
\end{remark}

\section{Linear and Null-Form Estimates}\label{sec:wave-estimates}
We conclude in this section the adaptation of the classical semilinear--wave--tools to the anisotropic setting started in \S \ref{sec:anisotropic-spaces}. We establish the
linear and bilinear estimates used by the elliptic \S\ref{sec:elliptic} and fixed-point \S \ref{sec:fixedpoint} arguments. \\

We set some notation first. With $(\tau,\xi)$ the spacetime Fourier variables, $\lambda$ as in
\eqref{eq:Bmult}, and $\tau_\pm=\tau\pm\lambda$, one has
\begin{equation}\label{eq:joint-res}
 \widetilde{\mathcal Bf}=\mathrm i\lambda\widetilde f,
 \qquad
 \widetilde{\mathcal A^+f}=\mathrm i\tau_+\widetilde f,
 \qquad
 \widetilde{\mathcal A^-f}=\mathrm i\tau_-\widetilde f,
 \qquad
 \widetilde{\Box f}=-\tau_+\tau_-\widetilde f .
\end{equation}
For every Borel function $\Phi$ and sufficiently regular $f$, we use the Fourier-multiplier notation
\begin{equation}\label{eq:calculus}
 \Phi(D_+,D_-)f
 :=\big(\Phi(\tau_+,\tau_-)\widetilde f\big)^\vee
\end{equation}
Since $\mathcal B$ has constant coefficients can be diagonalized together with $\Lambda$. In particular
\begin{equation}\label{eq:comm-zero}
 [\Lambda^{\sreg},\mathcal B]=0
\end{equation}
for $\sreg\in\Real$.

\subsection{The linear estimate}
\label{ssec:linear}

Fix a cutoff $\chi\in C^\infty_c(\Real)$, with
\[
    0\le\chi\le1, \qquad \chi\equiv1 \quad \text{on} \quad [-1,1], \qquad \supp\chi\subset[-2,2],
\]
and recall
$\varepsilon=1-\breg$, $\breg\in(\frac12,1)$ and $0<T\le1$.

The free evolution propagators are
\begin{equation}\label{eq:propagator}
\Cprop(t):=\cos\big(t|\mathcal B|\big)
=\tfrac12\big(\mathrm e^{t\mathcal B}+\mathrm e^{-t\mathcal B}\big),
\qquad
\Sprop(t):=\frac{\sin\big(t|\mathcal B|\big)}{|\mathcal B|}
=\int_0^{t}\tfrac12\big(\mathrm e^{r\mathcal B}+\mathrm e^{-r\mathcal B}\big)\,dr,
\end{equation}
where the value of the symbol of $\Sprop(t)$ at $\lambda=0$ is $t$.
Since $\mathrm e^{r\mathcal B}f(y)=f(y+r\Binf)$, their
$L^p$-operator norms are at most $1$ and $|t|$, respectively.

For initial data $f_0,f_1$ and forcing $G$, define the \textit{time-localised
Duhamel solution} by
\begin{equation}\label{eq:localized-duhamel}
 u(t):=\chi(t)\Cprop(t)f_0+\chi(t)\Sprop(t)f_1+\chi(t/T)\int_0^t\Sprop(t-t')G(t')\,dt'.
\end{equation}
The following records the estimates of this object in the wave-adapted norm.

\begin{theorem}[the linear estimate]\label{thm:linear}
Let 
\[
    \breg\in(\frac12,1), \qquad \varepsilon:=1-\breg, \qquad \sreg\ge0, \qquad \preg\in\Real, \qquad 0<T\le1.
\]
For every
\[
    (f_0,f_1)\in Y^{\sreg,\preg}\times Y^{\sreg,\preg-1}, \qquad G\in L^{2}_tY^{\sreg,\preg-1}(\Rst),
\]
the function $u$ defined by
\eqref{eq:localized-duhamel} satisfies
\[
 \Box u=G\quad\text{on }[0,T],\qquad
 (u,\partial_tu)|_{t=0}=(f_0,f_1),
\]
and the estimate
\begin{equation}\label{eq:linear}
\big\|u\big\|_{H^{\sreg,\preg}_{\breg}}
 \le C_{\rm lin}\Big(
 \big\|f_0\big\|_{Y^{\sreg,\preg}}
 +\big\|f_1\big\|_{Y^{\sreg,\preg-1}}
 +T^{\varepsilon}\big\|G\big\|_{L^{2}_tY^{\sreg,\preg-1}}\Big),
\end{equation}
where $C_{\rm lin}=C(\chi,\breg)$ is independent of $T\le1$, of the
field, of the data and of the indices $\sreg,\preg$. Moreover,
\[
 u\in C([0,T],Y^{\sreg,\preg})
 \cap C^1([0,T],Y^{\sreg,\preg-1}).
\]
\end{theorem}

\begin{proof}
Set
\[
 u_D(t):=\int_0^t\Sprop(t-t')G(t')\,dt'.
\]
for the forced zero-data non-localised part of the solution. Formula~\eqref{eq:localized-duhamel} gives the asserted equation and
data on $[0,T]$.
By \eqref{eq:Omega-regularized-wave} and commutation with the spatial
weights,
\begin{equation}\label{eq:regularized-oscillator-weight}
 \|u\|_{H^{\sreg,\preg}_{\breg}}
 =\|m^{-\varepsilon}\Box^{\mathrm{reg}}u\|_
 {L^2_tY^{\sreg,\preg-1}}.
\end{equation}

\textit{Time-localisation gain.} We first establish the time-localisation estimate
\begin{equation}\label{eq:negative-modulation-local}
 \|m^{-\varepsilon}F\|_{L^2_tY^{\sreg,\preg-1}}
 \lesssim_\varepsilon T^\varepsilon
 \|F\|_{L^2_tY^{\sreg,\preg-1}},
 \qquad \supp_tF\subset[-2T,2T].
\end{equation}
Using the dual form of the one-dimensional Sobolev embedding
$H^\varepsilon(\Real)\hookrightarrow L^{2/(1-2\varepsilon)}(\Real)$, namely $L^{2/(1+2\varepsilon)}(\Real)\hookrightarrow H^{-\varepsilon}(\Real)$, we obtain for fixed $\xi$,
\[
\begin{aligned}
 \int_{\Real}m(\tau,\xi)^{-2\varepsilon}
       |\widetilde F(\tau,\xi)|^2\,d\tau
 &\le\sum_{\pm}\int_{\Real}
       \langle\tau\pm\lambda(\xi)\rangle^{-2\varepsilon}
       |\widetilde F(\tau,\xi)|^2\,d\tau\\
 &=\sum_{\pm}\int_{\Real}
       \langle\rho\rangle^{-2\varepsilon}
       \big|\widetilde F(\rho\mp\lambda(\xi),\xi)\big|^2\,d\rho\\
 &\lesssim_\varepsilon\sum_{\pm}
       \big\|\mathrm e^{\pm it\lambda(\xi)}\widehat F(t,\xi)
       \big\|_{L^{2/(1+2\varepsilon)}_t}^2\\
 &=2\|\widehat F(\cdot,\xi)\|_{L^{2/(1+2\varepsilon)}_t}^2\\
 &\lesssim T^{2\varepsilon}
       \int_{\Real}|\widehat F(t,\xi)|^2\,dt.
\end{aligned}
\]
where the last inequality uses the time support of $\widehat F(\cdot,\xi)$ and H\"older's inequality.

\textit{Free evolution terms.} We see that if $\Box V=0$ then by
\eqref{eq:regularized-wave},
\begin{equation}\label{eq:waveoncut}
    \Box^{\mathrm{reg}}(\chi V)
 =(\chi''+2\chi'+\chi)V+2(\chi'+\chi)\partial_tV.
\end{equation}
Since $\chi$ and its derivatives are supported in $[-2,2]$,
\eqref{eq:regularized-oscillator-weight} and $m^{-\varepsilon}\le1$ give
\[
\begin{aligned}
 \|\chi V\|_{H^{\sreg,\preg}_{\breg}}
 &\le \|\Box^{\mathrm{reg}}(\chi V)\|_{L^2_tY^{\sreg,\preg-1}}\\
 &\lesssim_\chi
 \|V\|_{L^2([-2,2];Y^{\sreg,\preg-1})}
 +\|\partial_tV\|_{L^2([-2,2];Y^{\sreg,\preg-1})}.
\end{aligned}
\]

For $V(t)=\Cprop(t)f_0$, the symbols of $\Cprop(t)$ and
$\partial_t\Cprop(t)$ are $\cos(t|\lambda|)$ and
$-|\lambda|\sin(t|\lambda|)$. Hence, for every $t$,
\[
 \|\Cprop(t)f_0\|_{Y^{\sreg,\preg-1}}
 \le \|f_0\|_{Y^{\sreg,\preg-1}}
 \le \|f_0\|_{Y^{\sreg,\preg}}, \qquad
 \|\partial_t\Cprop(t)f_0\|_{Y^{\sreg,\preg-1}}
 \le \|f_0\|_{Y^{\sreg,\preg}},
\]
where the second line uses
$\langle\lambda\rangle^{\preg-1}|\lambda|
\le\langle\lambda\rangle^\preg$.

For $V(t)=\Sprop(t)f_1$, from
\begin{equation}\label{eq:Sbounds}
    \frac{|\sin(t|\lambda|)|}{|\lambda|}\le|t| \qquad \text{and} \qquad \partial_t\Sprop(t)=\Cprop(t)
\end{equation}
we obtain
\[
 \|\Sprop(t)f_1\|_{Y^{\sreg,\preg-1}}
 \le |t|\|f_1\|_{Y^{\sreg,\preg-1}},
 \qquad
 \|\partial_t\Sprop(t)f_1\|_{Y^{\sreg,\preg-1}}
 \le \|f_1\|_{Y^{\sreg,\preg-1}}.
\]
Both choices of $V$ solve $\Box V=0$. Integrating these bounds over
$[-2,2]$ and applying the preceding cutoff estimate therefore yields
\begin{equation}\label{eq:dataestimate}
    \|\chi\Cprop f_0\|_{H^{\sreg,\preg}_{\breg}}
 +\|\chi\Sprop f_1\|_{H^{\sreg,\preg}_{\breg}}
 \lesssim_\chi
 \|f_0\|_{Y^{\sreg,\preg}}
 +\|f_1\|_{Y^{\sreg,\preg-1}}.
\end{equation}

\textit{Forced term and conclusion.} By construction we have $\Box u_D=G$ with zero data.
From \eqref{eq:Sbounds},
Cauchy--Schwarz in $t'$ and integration on $[-2T,2T]$, we deduce
\begin{equation}\label{eq:oscillator-local}
\begin{aligned}
 \|u_D\|_{L^2([-2T,2T];Y^{\sreg,\preg-1})}
 &\lesssim T^2\|G\|_{L^2_tY^{\sreg,\preg-1}},\\
 \|\partial_tu_D\|_{L^2([-2T,2T];Y^{\sreg,\preg-1})}
 &\lesssim T\|G\|_{L^2_tY^{\sreg,\preg-1}}.
\end{aligned}
\end{equation}
Computing as in \eqref{eq:waveoncut}, we obtain
\[
\begin{aligned}
 \Box^{\mathrm{reg}}\big(\chi(t/T)u_D\big)
 &=\chi(t/T)G
   +2\big(T^{-1}\chi'(t/T)+\chi(t/T)\big)\partial_tu_D\\
 &\quad+\big(T^{-2}\chi''(t/T)
       +2T^{-1}\chi'(t/T)+\chi(t/T)\big)u_D.
\end{aligned}
\]
Thus \eqref{eq:oscillator-local} and $T\le1$ imply
\begin{equation}\label{eq:oscillator-cutoff-source}
 \big\|\Box^{\mathrm{reg}}\big(\chi(t/T)u_D\big)\big\|_
 {L^2_tY^{\sreg,\preg-1}}
 \lesssim_\chi\|G\|_{L^2_tY^{\sreg,\preg-1}}.
\end{equation}
Since $\Box^{\mathrm{reg}}$ preserves time support, 
$\Box^{\mathrm{reg}}\big(\chi(t/T)u_D\big)$ is supported in $[-2T,2T]$.
Applying \eqref{eq:negative-modulation-local} in
\eqref{eq:regularized-oscillator-weight} gives
\[
\begin{aligned}
 \|\chi(t/T)u_D\|_{H^{\sreg,\preg}_{\breg}}
 &\lesssim_\varepsilon
 T^\varepsilon
 \big\|\Box^{\mathrm{reg}}\big(\chi(t/T)u_D\big)\big\|_
 {L^2_tY^{\sreg,\preg-1}}\\
 &\lesssim_{\chi,\breg}
 T^\varepsilon\|G\|_{L^2_tY^{\sreg,\preg-1}}.
\end{aligned}
\]
This together with \eqref{eq:dataestimate} proves \eqref{eq:linear}.
The continuity statement follows from Corollary~\ref{cor:trace-Y}.
\end{proof}

The following uniqueness statement for the free evolution is used twice
below: to identify auxiliary solutions with the fixed point in
Theorem~\ref{thm:existence}, and to propagate the volume constraint in
Proposition~\ref{prop:constraint}.

\begin{lemma}[uniqueness for the homogeneous wave equation]\label{lem:wave-uniqueness}
Let $0<T<\infty$ and let $u\in C^1([0,T];L^2(\Rn))$ satisfy
$\mathcal Bu\in C([0,T];L^2(\Rn))$,
\[
 \Box u=0\quad\text{in }\mathcal D'((0,T)\times\Rn),
 \qquad
 u(0)=\partial_tu(0)=0 .
\]
Then $u\equiv0$ on $[0,T]$. The same conclusion holds componentwise for
vector- and matrix-valued $u$.
\end{lemma}

\begin{proof}
Apply the uniqueness assertion of
\cite[Theorem~10.14, p.~345]{Brezis2011Functional} on $L^2(\Rn)$
to the symmetric form
$(v,w)\mapsto\langle\mathcal Bv,\mathcal Bw\rangle_{L^2}$
with domain $V=\{v\in L^2(\Rn):\mathcal Bv\in L^2(\Rn)\}$,
equipped with its graph norm. The required lower bound follows from
$\|\mathcal Bv\|_2^2=\|v\|_V^2-\|v\|_2^2$, while the hypotheses
and the equation give $u\in C([0,T];V)$,
$\partial_tu\in C([0,T];L^2)$ and
$\partial_t^2u\in C([0,T];V')$.
The zero initial data therefore imply $u\equiv0$.
\end{proof}

\subsection{The null-form estimate}
\label{ssec:null-form}

The nonlinearity in \eqref{eq:auxiliary-system} is the bilinear Alfv\'en null form $Q_0$ from \eqref{eq:Q0}. The goal of the following is to show an estimate of the form \eqref{eq:simplequadraticintro} where the heuristic assumption $\mathcal A^\mp V^\pm=0$ is replaced by a control of those derivatives through the wave-adapted norm. 

\begin{theorem}[fractional null-form estimate]\label{thm:null-form}
Let 
\[
    n\ge2, \qquad \breg\in(\frac12,1], \qquad \sigma>\frac{n-1}2, \qquad \rpar\in[0,1].
\]
Then
\begin{equation}\label{eq:null-form}
 \|Q_0(f,g)\|_{L^2_tY^{\sigma,\rpar}}
 \le C\Big(
   \|f\|_{H^{\sigma,1+\rpar}_{\breg}}
   \|g\|_{H^{\sigma,1}_{\breg}}
  +\|f\|_{H^{\sigma,1}_{\breg}}
   \|g\|_{H^{\sigma,1+\rpar}_{\breg}}
 \Big),
\end{equation}
where $C=C(n,\sigma,\breg,v_{\rm A})$ is uniform for
$\rpar\in[0,1]$. In particular,
\begin{equation}\label{eq:diagonalNF}
     \|Q_0(f,f)\|_{L^2_tY^{\sigma,\rpar}}
 \le 2C\|f\|_{H^{\sigma,1}_{\breg}}
          \|f\|_{H^{\sigma,1+\rpar}_{\breg}}.
\end{equation}
\end{theorem}

Following the computations in \eqref{eq:keycomputation}, we use the null coordinates of \eqref{eq:taupm}. As in the proof of Lemma~\ref{lem:Linftyembed}, we first rotate the coordinates so that 
$\xi=(\xi_1,\xi')$ and $\lambda=v_{\rm A}\xi_1$, then change of variables
\[
    (\tau,\xi_1)\mapsto z
=(\tau+v_{\rm A}\xi_1,\tau-v_{\rm A}\xi_1)
\]
so that we can write by Plancherel
\begin{equation}\label{eq:null-coordinate-plancherel}
 d\nu(z):=\frac{d\tau_+\,d\tau_-}{2v_{\rm A}},
 \qquad
 \|h\|_{L^2(\Rst)}^2
 =\iint|\widetilde h(z,\xi')|^2\,d\nu(z)\,d\xi'.
\end{equation}
When $h=fg$ is a product, that is a convolution on the Fourier side, at fixed output frequency $(z,\xi')$, we let $(z_1,\xi'-\eta')$ be the frequency
of $\widetilde f$ and $(z_2,\eta')$ the one of $\widetilde g$, with $z_1:=z-z_2$. Thus \eqref{eq:Q0} gives
\begin{equation}\label{eq:convolformula}
    \widetilde{Q_0(f,g)}(z,\xi')=-(2\pi)^{-(n+1)/2}
 \iint q_0(z_1,z_2)\,\widetilde f(z_1,\xi'-\eta')\,
       \widetilde g(z_2,\eta')\,d\eta'\,d\nu(z_2),
\end{equation}
where the minus sign comes from $\mathrm i^2=-1$ and
\begin{equation}\label{eq:resonance}
 q_0(z_1,z_2):=\tfrac12\big(
 \tau_+(z_1)\tau_-(z_2)+\tau_-(z_1)\tau_+(z_2)\big).
\end{equation}
The whole bilinear machinery then runs on the kernel
\begin{equation}\label{eq:eq:kernellnull}
    K(z_1,z_2;\eta)
 :=\frac{|q_0(z_1,z_2)|}
 {\Omega(z_1)\Omega(z_2)\langle\eta\rangle^{\sigma}}, \qquad\eta=(\eta_1,\eta'),\qquad
 \eta_1=\frac{\tau_+(z_2)-\tau_-(z_2)}{2v_{\rm A}}
\end{equation}
which appears when dividing by the symbols of the wave norms entering the rhs of \eqref{eq:null-form}. For later use, we also note that 
$\Omega=\langle\tau_+\rangle\langle\tau_-\rangle/m^{1-\breg}$ gives
\begin{equation}\label{eq:null-weight-pointwise}
 \frac{|\tau_\pm|}{\Omega}
 \le\frac{m^{1-\breg}}{\langle\tau_\mp\rangle}
 \le\langle\tau_\mp\rangle^{-\breg}.
\end{equation}
Here, the Alfv\'en derivatives in \eqref{eq:resonance} are controlled by the wave-norm symbol $\Omega$ bounding \eqref{eq:eq:kernellnull} by oppositely decaying square-integrable kernels ($2\breg>1$ by assumption). 

\begin{proposition}[the core estimate]\label{prop:core}
For $\breg\in(\frac12,1]$ and $\sigma>\frac{n-1}2$,
\begin{equation}\label{eq:core}
 \left\|
 \iint K(z-z_2,z_2;\eta)
 F(z-z_2,\xi'-\eta')G(z_2,\eta')
  d\eta' d\nu(z_2)
 \right\|_{L^2_{z,\xi'}}
 \le C\|F\|_2\|G\|_2 .
\end{equation}
One may take
\begin{equation}\label{eq:core-const}
 C^2\le
 \frac{c_\perp(n,\sigma)c_\parallel(\breg)^2}{2v_{\rm A}},
 \quad
 c_\perp(n,\sigma):=\int_{\Real^{n-1}}\langle\eta'\rangle^{-2\sigma} d\eta',
 \quad
 c_\parallel(\breg):=\int_{\Real}\langle r\rangle^{-2\breg} dr .
\end{equation}
\end{proposition}

\begin{proof}
Using \eqref{eq:resonance}, \eqref{eq:null-weight-pointwise},
$|a+b|^2\le2|a|^2+2|b|^2$ and
$\langle\eta\rangle\ge\langle\eta'\rangle$, we obtain
\[
 K(z_1,z_2;\eta)^2\le\tfrac12\langle\eta'\rangle^{-2\sigma}\Big(
   \langle\tau_-(z_1)\rangle^{-2\breg}
   \langle\tau_+(z_2)\rangle^{-2\breg}+\langle\tau_+(z_1)\rangle^{-2\breg}
   \langle\tau_-(z_2)\rangle^{-2\breg}\Big).
\]
At fixed $z$, put $r_\pm=\tau_\pm(z_2)$. Integrating the preceding
bound with $z_1=z-z_2$, the two terms give equal contributions, so
\[
\begin{aligned}
 \iint K(z-z_2,z_2;\eta)^2\,d\eta'\,d\nu(z_2)&\quad\le
 \frac{c_\perp(n,\sigma)}{2v_{\rm A}}
 \int_{\Real^2}
 \langle\tau_-(z)-r_-\rangle^{-2\breg}
 \langle r_+\rangle^{-2\breg}\,dr_+\,dr_-\\
 &\quad=
 \frac{c_\perp(n,\sigma)c_\parallel(\breg)^2}{2v_{\rm A}}.
\end{aligned}
\]
Cauchy--Schwarz in $(z_2,\eta')$ therefore yields
\[
\begin{aligned}
 &\left|\iint K(z-z_2,z_2;\eta)
 F(z-z_2,\xi'-\eta')G(z_2,\eta')\,d\eta'\,d\nu(z_2)\right|^2\\
 &\quad\le
 \frac{c_\perp(n,\sigma)c_\parallel(\breg)^2}{2v_{\rm A}}
 \iint |F(z-z_2,\xi'-\eta')|^2
 |G(z_2,\eta')|^2\,d\eta'\,d\nu(z_2).
\end{aligned}
\]
We can now integrating in $(z,\xi')$ use Tonelli's theorem and the change of variables
$(z_1,\tilde\xi')=(z-z_2,\xi'-\eta')$ to obtain \eqref{eq:core} with the constant in \eqref{eq:core-const}.
\end{proof}

\begin{proof}[Proof of Theorem~\ref{thm:null-form}]
By Plancherel and \eqref{eq:convolformula}, it suffices
to bound the $L^2_{z,\xi'}$ norm of the convolution with symbol
$\langle\xi\rangle^\sigma\langle\lambda(\xi)\rangle^{\rpar}
|q_0(z_1,z_2)|$ and inputs $|\widetilde f|$, $|\widetilde g|$.
Since $\xi=(\xi-\eta)+\eta$ and $\lambda$ is linear, we have
\[
\begin{aligned}
 \langle\xi\rangle^\sigma\langle\lambda(\xi)\rangle^{\rpar}
 &\lesssim_\sigma
 \big(\langle\xi-\eta\rangle^\sigma+\langle\eta\rangle^\sigma\big)
 \big(\langle\lambda(\xi-\eta)\rangle^{\rpar}
      +\langle\lambda(\eta)\rangle^{\rpar}\big),
\end{aligned}
\]
where we used that $\langle a+b\rangle\le\langle a\rangle+\langle b\rangle$ and
$(u+v)^{\rpar}\le u^{\rpar}+v^{\rpar}$ for $u,v>0$ and $0\le\rpar\le1$,
so the constant is uniform in $\rpar$.

For the two terms with isotropic weight $\langle\xi-\eta\rangle^\sigma$,
the definition of $K$ gives
\[
 \langle\xi-\eta\rangle^\sigma |q_0(z_1,z_2)|
 =K(z_1,z_2;\eta)
  \big(\Omega(z_1)\langle\xi-\eta\rangle^\sigma\big)
  \big(\Omega(z_2)\langle\eta\rangle^\sigma\big).
\]
The two factors in parentheses multiply $|\widetilde f|$ and
$|\widetilde g|$, respectively. By \eqref{eq:Hsymbol}, the resulting
weighted Fourier transforms have norms $\|f\|_{H^{\sigma,1}_{\breg}}$
and $\|g\|_{H^{\sigma,1}_{\breg}}$. Distributing the additional longitudinal
weight on either of the inputs increases the parallel index to $1+\rpar$.
Thus \eqref{eq:core} bounds the two
convolution norms, up to a constant, by
\[
 \|f\|_{H^{\sigma,1+\rpar}_{\breg}}\|g\|_{H^{\sigma,1}_{\breg}}
 \qquad \text{and} \qquad
 \|f\|_{H^{\sigma,1}_{\breg}}\|g\|_{H^{\sigma,1+\rpar}_{\breg}}.
\]

For the terms with isotropic weight $\langle\eta\rangle^\sigma$,
the change of variables $(z_2,\eta')\mapsto(z-z_2,\xi'-\eta')$ swaps
the two input frequencies, preserves $d\eta'\,d\nu(z_2)$ and given the symmetry $q_0(z_1,z_2)=q_0(z_2,z_1)$ (see \eqref{eq:resonance}), we can then reason as in the previous case. Summing the two proves \eqref{eq:null-form}, uniformly in $\rpar$.
Taking $g=f$ gives \eqref{eq:diagonalNF}.
\end{proof}

\section{Elliptic Estimates}\label{sec:elliptic}

With the pullback geometry and canonical Hodge representative constructed in
Section~\ref{sec:wave-hodge} at hand, this section develops the first-order derivative
estimates for the Hodge system.  Two preliminary results on top of the previous sections are needed: stability of Sobolev regularity under
the flow, and a coefficient calculus with base space $Y^{s-1,1}$.  

\subsection{Sobolev stability under the flow}

\begin{lemma}[properties of precomposition operators]\label{lem:composition}
Let $X=\mathrm{id}+w:\Rn_y\to\Rn_x$ be a
bi-Lipschitz homeomorphism such that
\begin{equation}\label{eq:asschange}
     \|dX\|_{L^\infty}+\|(dX)^{-1}\|_{L^\infty}\le\mathfrak m,
 \qquad 0<J_-\le|\det dX|\le J_+.
\end{equation}
Put $\Ufield:=dw$, $\Kmat:=(dX)^{-1}$ and assume
$\Ufield,\Kmat-\mathrm{Id}\in Y^{s-1,1}$.  Recall $\mathcal{C}_{X}^{\pm1}f=f\circ X^{\pm1}$. The following holds.
\begin{enumerate}
    \item[(i)]\textup{Stability.} For every $0\le r\le s$,
        \begin{equation}\label{eq:sobolev-composition}
        \|\mathcal C_Xf\|_{H^r_y}\le C_{\mathrm{comp}}(X,r)\|f\|_{H^r_x},
        \qquad
        \|\mathcal C_X^{-1}u\|_{H^r_x}\le C_{\mathrm{comp}}(X,r)\|u\|_{H^r_y},
        \end{equation}
        where the constant can be chosen so that
        \[
        C_{\mathrm{comp}}(X,r)\le C(n,s,r,v_{\rm A},\mathfrak m,J_\pm)
        \big(1+\|\Ufield\|_{Y^{s-1,1}}
        +\|\Kmat-\mathrm{Id}\|_{Y^{s-1,1}}\big)^{k(r)},
        \]
        here $k(r):=\max\{\lceil r\rceil-1,0\}$.
    \item[(ii)]\textup{Uniform in time stability.} If \eqref{eq:asschange} hold uniformly for $X=X(t)$ on a compact time interval $I$ and
    \[
        w\in C(I;H^s),\qquad
        \Ufield,\Kmat-\mathrm{Id}\in C(I;Y^{s-1,1}),
    \]
        then $u\in C(I;H^r_y)$ implies
        $\mathcal C_{X(t)}^{-1}u(t)\in C(I;H^r_x)$.
    \item[(iii)]\textup{Lebesgue spaces continuity.} Independently of~(ii), suppose \eqref{eq:asschange} holds uniformly on $I$ and $X^{\pm1}\in C(I\times\Rn;\Rn)$. Then for every $1\le q<\infty$,
        \begin{equation}\label{eq:composition-Lr-continuity}
         \sup_{t\in I}\|\mathcal C_{X(t)}^{\pm1}\|_{L^q\to L^q}<\infty,\qquad
         \lim_{t\to t_0}\|\mathcal C_{X(t)}^{\pm1}g
         -\mathcal C_{X(t_0)}^{\pm1}g\|_{L^q}=0 \qquad g\in L^q, \,t_0\in I.
        \end{equation}
\end{enumerate}
\end{lemma}

\begin{proof}
For $r=0$, \eqref{eq:sobolev-composition} is the change-of-variables formula;
for $r=1$ it follows from the weak chain rules
\begin{equation}\label{eq:pullback-chain-rule}
 \nabla^y(\mathcal C_Xf)=\mathcal C_X(\nabla f)dX,
 \qquad
 \nabla(\mathcal C_X^{-1}u)=\mathcal C_X^{-1}\big((\nabla^yu)\Kmat\big),
\end{equation}
together with the bounds on $\|dX\|_{L^\infty}$, $\|\Kmat\|_{L^\infty}$ and
$|\det dX|$; and interpolation gives $0<r<1$.  For $1<r\le s$, use the elementary
characterisation
\[
 \|\phi\|_{H^r}\simeq\|\phi\|_{H^{r-1}}+\|\nabla \phi\|_{H^{r-1}}
\]
and induct on \(k(r)\), applying the product estimate \eqref{eq:tame-sp} at \(\varsigma=r-1\) in both the chain rules from \eqref{eq:pullback-chain-rule}.

For the time-dependent statement, the uniform operator bounds reduce the claim by
density to a fixed smooth compactly supported function.  The hypotheses imply
$X(t)^{\pm1}\to X(t_0)^{\pm1}$ locally uniformly and
$\Kmat(t)-\Kmat(t_0)\to0$ in $L^\infty\cap Y^{s-1,1}$.  Dominated
convergence proves the claim for $r=0$, the chain rules give $r=1$, and interpolation
gives $0<r<1$; by the product estimate \eqref{eq:tame-sp} and induction on $k(r)$ the case $1<r\le s$ follows.

For \eqref{eq:composition-Lr-continuity}, the operators
$\mathcal C_{X(t)}^{\pm1}$ are uniformly bounded on every $L^q$,
$1\le q<\infty$, because $|J|$ is bounded above and below. For fixed
$g\in C_c(\Rn)$, the continuity statement follows from dominated convergence and the
assumed continuity of $X(t)^{\pm1}$, arguing by density gives the full claim.
\end{proof}

\subsection{Coefficient calculus}\label{ssec:coeffcalculus}

By Proposition~\ref{prop:tame-product}, $Y^{\sigma,1}$ is a Banach
algebra. The following Lemma applies to the determinant and cofactors
of $\mathrm{Id}+M$, to $\log\det(\mathrm{Id}+M)$ and
$(\mathrm{Id}+M)^{-1}$, and to their products. In particular,
it controls $J-1$, $\Kmat-\mathrm{Id}$, $\varrho$ and the
coefficient products in $\mathcal Q$ out of a single statement.

\begin{lemma}[coefficient calculus]\label{lem:coeff-calculus}
Recall $\sigma=s-1>(n-1)/2$ and let $\Psi$ be real analytic near the origin and possibly matrix valued. There is
$\delta_\Psi=\delta_\Psi(\Psi,n,\sigma,v_{\rm A})>0$ such that, if
\begin{equation}\label{eq:deltapsi}
    \|M_1\|_{Y^{\sigma,1}}+\|M_2\|_{Y^{\sigma,1}}\le\delta_\Psi
\end{equation}
then
\begin{equation*}\tag{P}\label{eq:coeff-lipschitz}
\|\Psi(M_1)-\Psi(M_2)\|_{Y^{\sigma,1}}
\ \le\ C_\Psi\|M_1-M_2\|_{Y^{\sigma,1}}.
\end{equation*}
For time-continuous $M_i$, these estimates hold in the
$C_tY^{\sigma,1}$ norm, provided the bound \eqref{eq:deltapsi} is uniform in time.
\end{lemma}

\begin{proof}
Expand $\Psi$ at the origin. Carefully rewriting each polynomial to highlight differences $M_1^j-M_2^j$ in
$\Psi(M_1)-\Psi(M_2)$ and using the $Y^{\sigma,1}$ algebra estimate \eqref{eq:Y-product} gives majorant in the form of a summable geometric series as soon as
$\|M_1\|_{Y^{\sigma,1}}+\|M_2\|_{Y^{\sigma,1}}\le\delta_\Psi$, proving \eqref{eq:coeff-lipschitz}.
\end{proof}

\subsection{The pressure Hessian}
\label{ssec:pressure-hessian} The canonical Hodge representative
$F=\mathcal P_{\gmetric}f$ and its HLS bound are given in
Lemma~\ref{lem:hodge-solvability}.  We now estimate its
derivative, namely the Lagrangian pressure Hessian (recall the correspondence $\nabla p \circ X= F$ in the derivation from smooth solutions). 

The idea is that $\mathcal P_{\gmetric}$ has order $-1$, so $d\mathcal P_{\gmetric}$
should behave like an operator of order zero.  To make this rigorous one needs
only observe that, in label coordinates, the metric Hodge operator is a
principal but small-coefficient perturbation of the flat Hodge operator. The smallness will be guaranteed by the parameter $\delta$ of the
fixed-point class $\mathcal U_{R,\delta}$ of \S\ref{sec:fixedpoint}.

\begin{lemma}[Hodge a priori estimates: general rhs]\label{lem:label-hodge}
Let $\Ufield=dw$ be exact, $X=\mathrm{id}+w$, with associated $\gmetric$ and $\mu$ as in \S\ref{sec:wave-hodge}, and recall $\sigma=s-1>(n-1)/2$.
There is $\delta_{\rm H}>0$ such that if
\begin{equation}\label{eq:hodge-smallness}
 \|\Ufield\|_{Y^{\sigma,1}}\le\delta_{\rm H},
\end{equation}
then, for $0\le\preg\le1$ and $f\in Y^{\sigma,\preg}\cap L^p(\mu)$ with $1<p<n$,
the canonical Hodge representative $F=\mathcal P_{\gmetric}f$ from Lemma~\ref{lem:hodge-solvability} satisfies
\begin{equation}\label{eq:elliptic-Y}
 \|d F\|_{Y^{\sigma,\preg}}
 \le C_{\rm H}\|f\|_{Y^{\sigma,\preg}}.
\end{equation}
In particular, 
\begin{equation}\label{eq:elliptic-Hsigma}
 \|d F\|_{H^\sigma}
 \le C_{\rm H}\|f\|_{H^\sigma}.
\end{equation}
Here the constant $C_{\rm H}$ is uniform on the ball defined by \eqref{eq:hodge-smallness} and $\delta_{\rm H}$ depends only on the standing parameters.
\end{lemma}
We remark that the condition $f\in L^p(\mu)$ with $1<p<n$ coming from the HLS bound becomes relevant only for $n=2$.
\begin{proof}
\emph{Geometry and coefficients.}
By \eqref{eq:embed-frac}, we may choose $\delta_{\rm H}$ so that
$\|dw\|_{L^\infty}\le\tfrac12$. Then $w$ is Lipschitz with constant at most
$\tfrac12$, and hence
\[
 \tfrac12|y-z|\le |X(y)-X(z)|\le\tfrac32|y-z|.
\]
For each $x$, the contraction $y\mapsto x-w(y)$ has a unique fixed point,
so $X$ is a global bi-Lipschitz homeomorphism. Set
\begin{equation}\label{eq:KMdef}
    \Kmat=\Kmat[\Ufield]:=(\mathrm{Id}+\Ufield)^{-1},\qquad
    \Mmat=\Mmat[\Ufield]:=\Kmat[\Ufield]-\mathrm{Id}.
\end{equation}
From Lemma \ref{lem:coeff-calculus} we deduce
\begin{equation}\label{eq:smallcoeff}
 \|\Mmat\|_{Y^{\sigma,1}}\le C \|\Ufield\|_{Y^{\sigma,1}},
 \qquad
 \|\Kmat\|_{L^\infty}\le C
\end{equation}
for some uniform constant $C$ independent of $\delta_{\rm H}$ (for $\delta_{\rm H}$ sufficiently small).

\emph{Splitting the Hodge system.}
Substituting $\Kmat=\mathrm{Id}+\Mmat$ in \eqref{eq:divcurlg} separates the
flat divergence and curl from the terms involving $\Mmat$. For a solution
of \eqref{eq:invdiv} with $dF\in L^2$, Lemma~\ref{lem:weak-pullback}(iii) gives
\begin{equation}\label{eq:label-hodge-perturbation}
 \begin{cases}
  \Div_y F=f-\Mmat^a_i\partial_aF^i,\\
  (\curl_y F)^{ij}=-\Mmat^a_i\partial_aF^j+\Mmat^a_j\partial_aF^i.
 \end{cases}
\end{equation}
Here
$\Div_y F=\partial_iF^i$ and $(\curl_y F)^{ij}=\partial_iF^j-\partial_jF^i$. It is now apparent that the perturbation from this flat operator is of first
order, but has small coefficients thanks to \eqref{eq:smallcoeff}.

\emph{Flat Hodge reconstruction.}
To recover $dF$ from the flat divergence and curl, we use
\begin{equation}\label{eq:motivation}
 \Delta_yF^j
 =\partial_j\partial_iF^i+\partial_i(\partial_iF^j-\partial_jF^i)
 =\partial_j\Div_y F+\partial_i(\curl_y F)^{ij}.
\end{equation}
For $dF\in L^2$ (which ensures that no harmonic remainder gets created), we obtain
\begin{equation}\label{eq:Hodgereconstruction}
 \partial_bF^k
 =\partial_b\Delta_y^{-1}
   \bigl(\partial_k\Div_y F+\partial_i(\curl_y F)^{ik}\bigr),
\end{equation}
now substituting \eqref{eq:label-hodge-perturbation} gives
\begin{equation}\label{eq:recoinproof}
    dF=\mathscr T_0f+\mathscr T_{\Mmat}dF.
\end{equation}
Here for any tensor $G$, we defined
\[
 (\mathscr T_0f)_b^k:=\partial_b\Delta_y^{-1}\partial_kf,
 \qquad
 (\mathscr T_{\Mmat}G)_b^k:=\partial_b\Delta_y^{-1}\Bigl[
   -\partial_k(\Mmat_\ell^aG_a^\ell)
   +\partial_i\bigl(-\Mmat_i^aG_a^k+\Mmat_k^aG_a^i\bigr)
 \Bigr].
\]
Thus $\mathscr T_0f$ is the flat Hessian, and $\mathscr T_{\Mmat}$ is the
perturbation. Each multiplier $\partial_b\Delta_y^{-1}\partial_i$ has
symbol $\xi_b\xi_i/|\xi|^2$, bounded by one. From Proposition \ref{prop:tame-product} and Remark \ref{rem:Y-interpolation} it follows that
\[
 \|\mathscr T_{\Mmat}\|_{L^2\to L^2}\lesssim\|\Mmat\|_{L^\infty},
 \qquad
 \|\mathscr T_{\Mmat}\|_{Y^{\sigma,\preg}\to Y^{\sigma,\preg}}
 \lesssim\|\Mmat\|_{Y^{\sigma,1}},
 \quad 0\le\preg\le1.
\]
By \eqref{eq:embed-frac} and \eqref{eq:smallcoeff}, eventually shrinking
$\delta_{\rm H}$ makes
$\mathrm{Id}-\mathscr T_{\Mmat}$ invertible on both spaces.

Arguing as in the proof of Lemma~\ref{lem:hodge-solvability}, we deduce $d\mathcal P_{\gmetric}f\in L^2$. Indeed, $\mathcal C_X^{-1}f\in L^2$, we can then apply the chain rule
\eqref{eq:pullback-chain-rule} to \eqref{eq:pressure-gradient} and using
the $L^2$ bound for $\nabla_x^2\Delta_x^{-1}$ the claim follows. The reconstruction \eqref{eq:recoinproof} and observations above therefore yield the identities
\begin{equation}\label{eq:pressure-Hessian}
    (\mathrm{Id}-\mathscr T_{\Mmat})d\mathcal P_{\gmetric}f=\mathscr T_0f
\end{equation}
and 
\begin{equation}\label{eq:neumann-hodge-derivative}
 d\mathcal P_{\gmetric}f=(\mathrm{Id}-\mathscr T_{\Mmat})^{-1}\mathscr T_0f,
\end{equation}
together with the bound
\[
 \|d\mathcal P_{\gmetric}f\|_{Y^{\sigma,\preg}}
 \lesssim\|\mathscr T_0f\|_{Y^{\sigma,\preg}}
 \lesssim\|f\|_{Y^{\sigma,\preg}},
\]
which proves \eqref{eq:elliptic-Y}.
\end{proof}

Let $I=[0,T]$, and let $\Ufield,\Vfield$ be spatially exact
$\Rn$-valued $1$-forms in $H^{\sigma,1}_{\breg}(\Rst)$ satisfying
\begin{equation}\label{eq:admissible}
 \max\left\{\|\Ufield\|_{H^{\sigma,1}_{\breg}},\, \|\Vfield\|_{H^{\sigma,1}_{\breg}}\right\}\le R,
 \qquad
 \max\left\{\sup_{t\in I}\|\Ufield(t)\|_{Y^{\sigma,1}},\, \sup_{t\in I}\|\Vfield(t)\|_{Y^{\sigma,1}}\right\}\le\delta,
\end{equation}
Choose $\delta\le\delta_{\rm H}$ small enough
that Lemma~\ref{lem:coeff-calculus} applies to pairs of such fields, as in
\eqref{eq:delta-choice}. 

We set the following conventions. All $L^p_t$ norms below are taken over $I$.
The constant $C_\delta$ depends only on $\delta$ and the standing parameters, independently of $R,T$ and the fields.

\begin{lemma}[Hodge a priori estimates: null-form rhs]\label{lem:pressure}
Let $\Ufield,\Vfield$ be as above, recall $\sigma=s-1>(n-1)/2$, and let
$F[\Ufield], F[\Vfield]$ be the canonical Hodge representatives \eqref{eq:FU}.
\begin{enumerate}
    \item[(i)]\textup{Higher-order integrability.} $F[\Ufield]\in C(I;L^4)$, with
            \begin{equation}\label{eq:timecontL4}
            \|F[\Ufield]\|_{L^\infty(I;L^4)}\le C_\delta R^2.
            \end{equation}
    \item[(ii)]\textup{Derivative null-form estimate.} For every $\rpar\in[0,1]$ such that
        $\Ufield\in H^{\sigma,1+\rpar}_{\breg}$,
        \begin{equation}\label{eq:pressure-anisotropic}
        \|dF[\Ufield]\|_{L^2(I;Y^{\sigma,\rpar})}
        \le C_\delta
        \|\Ufield\|_{H^{\sigma,1}_{\breg}}
        \|\Ufield\|_{H^{\sigma,1+\rpar}_{\breg}}.
        \end{equation}
        In particular,
        \begin{equation}\label{eq:pressure-isotropic}
            \|dF[\Ufield]\|_{L^2(I;H^\sigma)}\le C_\delta R^2.
        \end{equation}
    \item[(iii)]\textup{Local Lipschitz continuity.} 
        \begin{equation}\label{eq:pressure-lipschitz}
         \|dF[\Ufield]-dF[\Vfield]\|_{L^2(I;H^\sigma)}
        \le C_\delta R(1+R)
            \|\Ufield-\Vfield\|_{H^{\sigma,1}_{\breg}}.
\end{equation}
\end{enumerate}
\end{lemma}

\begin{proof}
\emph{Continuity and the $L^4$ bound.}
Corollary~\ref{cor:trace-Y} gives
\[
 \|\Ufield\|_{C(I;Y^{\sigma,1})}
 +\|\partial_t\Ufield\|_{C(I;H^\sigma)}
 +\|\mathcal B\Ufield\|_{C(I;H^\sigma)}
 \le CR.
\]
By Lemma~\ref{lem:coeff-calculus}, \eqref{eq:admissible} and
\eqref{eq:embed-frac},
\[
 \|\Mmat[\Ufield]\|_{C(I;Y^{\sigma,1})}\le C\delta,
 \qquad
 \|\Kmat[\Ufield]\|_{C(I;L^\infty)}\le C_\delta.
\]
Sobolev embedding $H^\sigma\hookrightarrow L^{2r_{\mathrm{HLS}}}$
and H\"older's inequality applied to \eqref{eq:datum} give
$\mathcal Q[\Ufield]\in C(I;L^{r_{\mathrm{HLS}}})$ and
\begin{equation}\label{eq:datum-fixed-time}
\begin{aligned}
 \|\mathcal Q[\Ufield]\|_{L^\infty_tL^{r_{\mathrm{HLS}}}}
 &\le C\|\Kmat[\Ufield]\|_{L^\infty_{t,y}}^2
 \left(
 \|\partial_t\Ufield\|_{L^\infty_tL^{2r_{\mathrm{HLS}}}}^2
 +\|\mathcal B\Ufield\|_{L^\infty_tL^{2r_{\mathrm{HLS}}}}^2
 \right)\\
 &\le C_\delta R^2.
\end{aligned}
\end{equation}
Since $r_{\mathrm{HLS}}^*=4$, the uniform Jacobian bounds and
\eqref{eq:canonical-hls} imply
\[
 \|F[\Ufield](t)\|_{L^4}
 \le C_\delta\|\mathcal Q[\Ufield](t)\|_{L^{r_{\mathrm{HLS}}}}
 \le C_\delta R^2,
\]
which proves \eqref{eq:timecontL4}.

The reconstruction argument in the proof of
Lemma~\ref{lem:label-hodge} also applies in $L^{r_{\mathrm{HLS}}}$,
in particular \eqref{eq:neumann-hodge-derivative} holds uniformly in time as $\Mmat[\Ufield]\in C(I;L^\infty)$ is uniformly small and 
$\mathcal Q[\Ufield]\in C(I;L^{r_{\mathrm{HLS}}})$. We deduce that
\[
 dF[\Ufield]\in C(I;L^{r_{\mathrm{HLS}}}).
\]
The Sobolev inequality then yields for $t,t_0\in I$ 
\[
 \|F[\Ufield](t)-F[\Ufield](t_0)\|_{L^4}
 \le C\|dF[\Ufield](t)-dF[\Ufield](t_0)\|_{L^{r_{\mathrm{HLS}}}}
 \longrightarrow0
\]
as $t\to t_0$ proving $F[\Ufield]\in C(I;L^4)$.

\emph{The anisotropic estimate.}
By Lemma~\ref{lem:coeff-calculus}, each
$\Kmat[\Ufield]^a_i\Kmat[\Ufield]^b_j-\delta_i^a\delta_j^b$
is bounded in $C(I;Y^{\sigma,1})$ by some constant depending on $\delta$. From the definition \eqref{eq:datum}, the product and null-form estimates \eqref{eq:Y-multiplier-interp},\eqref{eq:null-form}, we then obtain
\[
 \|\mathcal Q[\Ufield]\|_{L^2_tY^{\sigma,\rpar}}
 \le C_\delta
 \|\Ufield\|_{H^{\sigma,1}_{\breg}}
 \|\Ufield\|_{H^{\sigma,1+\rpar}_{\breg}}.
\]
Now take the $L^2$ norm in time of \eqref{eq:elliptic-Y} to deduce
\[
 \|dF[\Ufield]\|_{L^2_tY^{\sigma,\rpar}}
 \le C_\delta
      \|\mathcal Q[\Ufield]\|_{L^2_tY^{\sigma,\rpar}}
\]
and \eqref{eq:pressure-anisotropic} follows.

\emph{The difference estimate.}
By \eqref{eq:coeff-lipschitz} and Corollary~\ref{cor:trace-Y}, we have
\begin{equation}\label{eq:coeff-difference}
 \|\Mmat[\Ufield]-\Mmat[\Vfield]\|_{L^\infty_tY^{\sigma,1}}
 \le C_\delta\|\Ufield-\Vfield\|_{H^{\sigma,1}_{\breg}},
\end{equation}
with a similar bound for the difference in the coefficients $\Kmat[\cdot]$ of $Q_0$ entering the definition of $\mathcal{Q}$, see \eqref{eq:datum}. Highlighting these errors in \eqref{eq:datum} and using the same product
and null-form estimates above therefore gives
\begin{equation}\label{eq:datum-lipschitz}
 \|\mathcal Q[\Ufield]-\mathcal Q[\Vfield]\|_{L^2_tH^\sigma}
 \le C_\delta(R^2+R)
       \|\Ufield-\Vfield\|_{H^{\sigma,1}_{\breg}}.
\end{equation}
The factor $R^2$ comes from the coefficients difference while $R$ from
the error in the argument of $Q_0$.
Subtracting \eqref{eq:pressure-Hessian} for the two fields gives
\[
 (\mathrm{Id}-\mathscr T_{\Mmat[\Ufield]})
   \bigl(dF[\Ufield]-dF[\Vfield]\bigr)
   =\mathscr T_0\bigl(\mathcal Q[\Ufield]-\mathcal Q[\Vfield]\bigr)
   +\mathscr T_{\Mmat[\Ufield]-\Mmat[\Vfield]}dF[\Vfield].
\]
The first term is bounded by \eqref{eq:datum-lipschitz} as $\mathscr T_0$ is a bounded operator.
Then \eqref{eq:coeff-difference} and \eqref{eq:pressure-isotropic}
bound the second by
$C_\delta R^2\|\Ufield-\Vfield\|_{H^{\sigma,1}_{\breg}}$.
As in the proof of Lemma~\ref{lem:label-hodge}, the operator $\mathrm{Id}-\mathscr T_{\Mmat[\Ufield]}$ is invertible (uniformly in time), and
\eqref{eq:pressure-lipschitz} follows.
\end{proof}

\begin{remark}
    Time continuity of $F[\Ufield]$ is available but strictly speaking not needed. Indeed, the subsequent displacement reconstruction $\Ufield \mapsto w$ and $C_tH^s$ pressure recovery $F\circ X^{-1}=\nabla p_{\mathrm R}$ require only the
    bound \eqref{eq:timecontL4}.
\end{remark}

\section{The Differentiated Fixed Point}
\label{sec:fixedpoint}

We combine the linear, null-form and elliptic estimates to produce a
spatially exact spacetime field $\Ufield\in H^{s-1,1}_{\breg}$ solving the differentiated wave-Hodge system. The fixed-time smallness $\delta$ built into the class
supplies the coefficient and Hodge estimates, while short time and the datum
modulus $\omega_{v_0}$ yield this smallness for arbitrary-size data $R$.  The displacement $w$ is
reconstructed by the Duhamel formula only after the differentiated fixed point $\Ufield$ is found, and it will be such that $dw=\Ufield$ on the interval $[0,T_*]$ of definition.

\subsection{Preliminaries} We first introduce all the needed objects.\\

\textbf{The class.} Fix $0<T\le1$ and write $I:=[0,T]$.
For $v_0\in H^{s}(\Rn)$ divergence-free, of arbitrary size, we let
$C_{\rm lin}$ be the constant of \eqref{eq:linear} at $(\sreg,\preg)=(s-1,1)$, and put
\begin{equation}\label{eq:Rdefn}
    R:=2C_{\rm lin}\|v_0\|_{H^s}.
\end{equation}
For an $\Rn$-valued $1$-form $\Ufield$, write
\begin{equation}\label{eq:Curl}
    (\Curl_{y}\Ufield)^j_{ab}
 :=\partial_a\Ufield_b^j-\partial_b\Ufield_a^j, 
\end{equation}
we then let the fixed-point class be:
\begin{equation}\label{eq:fixedpoint-class}
  \mathcal U_{R,\delta} := \left\{ \Ufield \in H^{s-1,1}_{\breg}(\Rst) : 
  \begin{aligned}
    &\Curl_{y}\Ufield = 0, \quad (\Ufield,\partial_t\Ufield)(0) = (0,\nabla v_0), \\
    &\|\Ufield\|_{H^{s-1,1}_{\breg}} \le R, \quad \sup_{t\in I}\|\Ufield(t)\|_{Y^{s-1,1}} \le \delta
  \end{aligned}
  \right\}.
\end{equation}
Note that the exactness (intended as a spacetime distribution) and the trace conditions are closed. An element of $\mathcal U_{R,\delta}$ should be thought of as a `trial field' for the differentiated wave-Hodge system before a fixed point is found.  Moreover, at $t=0$ the identity $X[\Ufield](0,y)=y$ for any trial flow map (see \eqref{eq:trial-bilipschitz} below) identifies $x=y$, and we can write $(\nabla v_0)_a^i=\partial_av_0^i$ unequivocally.\\

\textbf{Reconstruction of the displacement $w[\Ufield]$ and map $X[\Ufield]$ for a trial field $\Ufield$.} We now explain how from the exactness condition of a trial field $\Ufield$ we can go back to a (primitive) displacement $w[\Ufield]$ and map $X[\Ufield]$; this is important because these objects will enter the fixed-point map $\mathcal{M}$, which will be defined next.  By the distributional Poincar\'e
lemma, see for example
\cite[Theorem~2.1]{Mardare2008}, an exact $\Ufield$ admits a primitive
$w[\Ufield]$ with $dw[\Ufield]=\Ufield$.  Here the primitive can
be written down explicitly via the homotopy
formula
\[
 w[\Ufield]^{ j}(t,y):=\int_0^1\Ufield_a^j(t,\theta y) y^a d\theta.
\]
This defines a continuous space-time function and is indeed a primitive. To see continuity, from Corollary~\ref{cor:trace-Y},
$\Ufield\in C(\Real;Y^{s-1,1})$ and the proof of
Lemma~\ref{lem:Linftyembed} shows that elements of $Y^{s-1,1}$ have integrable
Fourier transforms, so that $Y^{s-1,1}\hookrightarrow C_b(\Rn)$. To see that $d w[\Ufield] = \Ufield$ in distributions, one proceeds by mollification. Observe first that $\Curl_y\Ufield(t)=0$ in $\mathcal D'(\Rn)$ for
every $t$ by continuity. The formula above, after mollifying $\Ufield$ in space to $\Ufield_\varepsilon$, then defines primitives $w[\Ufield_\varepsilon]$ solving classically
$dw[\Ufield_\varepsilon]=\Ufield_\varepsilon$. One can then pass to the limit in distribution by (local) uniform convergence.   

In fact, since $\Ufield(t)\in L^\infty_y$,
$w[\Ufield](t)$ is Lipschitz in $y$ for every $t$.
We finally fix the remaining spatial translation gauge by subtracting the mean over the
unit ball from the homotopy above (without renaming), so that
\[
 |B_1|^{-1}\int_{B_1}w[\Ufield](t,y) dy=0.
\]
For $t\in I$ the anisotropic embedding \eqref{eq:embed-frac} and the definition
of the class give
\[
 L_t:=\|dw[\Ufield](t)\|_{L^\infty}
 \le C_{\rm e}\|\Ufield(t)\|_{Y^{s-1,1}}\le C_{\rm e}\delta\le\tfrac12,
\]
where the last inequality follows from the definition of $\delta$, which for convenience will be given later in \eqref{eq:delta-choice}.  Off $I$ the class imposes no
fixed-time smallness, and the discussion below is not claimed there.
Consequently, for all $t\in I$ and all $y,z\in\Rn$,
\begin{equation}\label{eq:trial-bilipschitz}
 (1-L_t)|y-z|
 \le|X[\Ufield](t,y)-X[\Ufield](t,z)|
 \le(1+L_t)|y-z|,
 \qquad X[\Ufield]:=\mathrm{id}+w[\Ufield].
\end{equation}
This proves injectivity.  Surjectivity follows from $y\mapsto x-w[\Ufield](t,y)$ being a contraction.  Thus $X[\Ufield](t,\cdot)$ is a
bi-Lipschitz and orientation-preserving homeomorphism. Any other primitive
differs by a spatial translation which leaves $\Ufield$,
$\Kmat[\Ufield]$, $\gmetric[\Ufield]$, $\mu[\Ufield]$, $\mathcal Q[\Ufield]$ and the operators $E_i$
from \S\ref{sec:wave-hodge} unchanged. The $L^4(\mu)$ uniqueness in Lemma~\ref{lem:hodge-solvability}
therefore shows that the associated canonical Hodge representative \[
 F[\Ufield]
 =\mathcal C_{X[\Ufield]}\nabla\Delta_{x}^{-1}\mathcal C_{X[\Ufield]}^{-1}\mathcal Q[\Ufield],
\] 
is unaffected as well. In particular, the construction below is independent of the fixed normalisation.\\

\textbf{The map.} We define the time-localised contraction map by
\eqref{eq:localized-duhamel} with $(f_0,f_1)=(0,\nabla v_0)$ and
$G=-\mathbf 1_I dF[\Ufield]$, namely:
\begin{equation}\label{eq:differential-map}
 (\mathcal M\Ufield)(t)
 :=\chi(t)\Sprop(t)\nabla v_0
 -\chi(t/T)\int_0^t\Sprop(t-t')
       \mathbf 1_I(t')dF[\Ufield](t') dt'.
\end{equation}
Since $\chi$ fixed in \S \ref{ssec:linear} satisfies $\chi\equiv1$ on $[-1,1]$ and $T\le1$, both cutoffs equal one on $I$, and
the field defined by $\Vfield(t):=(\mathcal M\Ufield)(t)$ solves
\[
    \Box\Vfield=-dF[\Ufield] \quad \text{on} \quad I \quad \text{with data} \quad (0,\nabla v_0).
\]
In particular, a fixed point of $\mathcal{M}$ solves the differentiated wave
equation associated with \eqref{eq:auxiliary-system}. The reconstruction of $w$ and
the identity $dw=\Ufield$ are proved afterwards.\\

\textbf{Preliminary bounds.} Firstly, since both terms in \eqref{eq:differential-map} are spatial differentials, we obtain
\[
    \Curl_y(\mathcal M\Ufield)=0,
\]
and thus $\mathcal{M}$ preserves the exactness condition of $\mathcal U_{R,\delta}$. Moreover, for $G\in L^2(I;H^{s-1})$, the time-localised Duhamel term satisfies
\begin{equation}\label{eq:cutoff-duhamel}
\begin{aligned}
 \left\|\chi(t/T)\int_0^t\Sprop(t-t')
       \mathbf 1_I(t')G(t') dt'\right\|_{H^{s-1,1}_{\breg}}
 &\le C_{\rm lin}T^\varepsilon\|G\|_{L^2_tH^{s-1}},\\
 \sup_{t\in I}\left\|\int_0^t\Sprop(t-t')G(t') dt'\right\|_{Y^{s-1,1}}
 &\le CT^{1/2}\|G\|_{L^2_tH^{s-1}}.
\end{aligned}
\end{equation}
The first estimate follows immediately from Theorem~\ref{thm:linear} with zero data. For the second, we first note that the free propagator \eqref{eq:propagator} satisfies the elementary scalar bound
\begin{equation}\label{eq:prop-gain}
\langle\lambda\rangle \Big|\frac{\sin(t\lambda)}{\lambda}\Big|\ \le\
\sqrt2 \min\big\{1,\ |t| \langle\lambda\rangle\big\},\qquad |t|\le1 ,
\end{equation}
obtained from $|\sin(t\lambda)|\le\min\big\{1,\ |t\lambda| \big\}$ in the regions
$|\lambda|\ge1$ and $|\lambda|\le1$. It follows that
\[
 \|\Sprop(r)f\|_{Y^{s-1,1}}
 \le C\|f\|_{H^{s-1}},\qquad |r|\le1.
\]
Now, Cauchy--Schwarz
uniformly for $0\le t\le T$, gives
\[
 \left\|\int_0^t\Sprop(t-t')G(t') dt'\right\|_{Y^{s-1,1}}
 \le C\int_0^t\|G(t')\|_{H^{s-1}} dt'
 \le CT^{1/2}\|G\|_{L^2(I;H^{s-1})}.
\]
and the second bound in \eqref{eq:cutoff-duhamel} follows.\\

\textbf{The smallness condition $\delta$.} We now define $\delta$ and explain why the smallness condition in \eqref{eq:fixedpoint-class} comes for free in our setup. This was used already to ensure that $X[\Ufield]=\mathrm{id}+w[\Ufield]$ is globally bi-Lipschitz with two-sided Jacobian
bounds at every $t\in I$. 

Let $\Psi_1,\dots,\Psi_N$ be the finitely many real-analytic functions used
below, namely $Z\mapsto\det(\mathrm{Id}+Z)$, its cofactors,
$\log\det(\mathrm{Id}+Z)$, $(\mathrm{Id}+Z)^{-1}-\mathrm{Id}$ and the pairwise
products of the entries of the last, and let $C_{\rm e}$ be the constant of
\eqref{eq:embed-frac} at $(\sreg,\preg)=(s-1,1)$.  Fix
\begin{equation}\label{eq:delta-choice}
 \delta:=\min\Big\{\tfrac1{2C_{\rm e}},\
 \tfrac12\min_{1\le k\le N}\delta_{\Psi_k},\
 \delta_{\rm H}\Big\}>0,
\end{equation}
which depends only on the standing parameters.  The first entry through
\eqref{eq:embed-frac} gives $\|\Ufield\|_{L^\infty(I\times\Rn)}\le\frac12$ whenever
$\sup_{t\in I}\|\Ufield(t)\|_{Y^{s-1,1}}\le\delta$. The factor $\frac12$ in the second entry is needed to apply \eqref{eq:coeff-lipschitz} to
$\Ufield(t),\Vfield(t)$ and use \eqref{eq:pressure-lipschitz}. Finally, the third entry ensures that the geometric smallness condition
\eqref{eq:hodge-smallness} in Lemma~\ref{lem:label-hodge} holds at every $t\in I$.\\
We emphasise that $\delta$ is fixed once and for all,
independently of $R$ and $T$. This is because the constants $C_\delta$ of
Lemma~\ref{lem:pressure} depend on neither, and we can ensure that the norm of the free evolution of the initial data in the definition \eqref{eq:differential-map} of $\mathcal{M}$, given $R$, does not exceed $\delta$ simply by controlling $T$.  Indeed, Plancherel, the definition of the space \eqref{eq:Y} and
estimate \eqref{eq:prop-gain} give, for $0\le t\le T$,
\begin{equation}\label{eq:free-fixed-time-modulus}
    \begin{split}
        \|\Sprop(t)\nabla v_0\|_{Y^{s-1,1}}^2
        &\le C\int_{\Rn}
        \min(1,T\langle\lambda(\xi)\rangle)^2
        \langle\xi\rangle^{2(s-1)}
        |\widehat{\nabla v_0}(\xi)|^2 d\xi \\
        &=C\omega_{v_0}(T)^2.
    \end{split}
\end{equation}
Thus, the free wave part in the map $\mathcal{M}$ is small in the fixed-time geometric norm even though its
spacetime norm is of size $\|v_0\|_{H^s}\simeq R$.  This is the only point at which
regime~\textup{(R1)} needs the datum-dependent modulus.  We use the two
properties recorded after \eqref{eq:datum-modulus}:
$\omega_{v_0}(T)\downarrow0$, and
$\omega_{v_0}(T)\le T^{\rpar}
\|\nabla v_0\|_{Y^{s-1,\rpar}}$ when that norm is finite.\\

We are now ready to state the existence result.
\begin{lemma}[the map is a self-map of the class]\label{lem:selfmap}
Let $v_0\in H^{s}(\Rn)$ be divergence-free. Then there exists $T_*>0$, such that for every $0< T\le T_*$, $\mathcal M$ is a well-defined self-map of the
nonempty class $\mathcal U_{R,\delta}$. Here, $T_*$ depends on
$n,s,\breg,v_{\rm A}$, on $\|v_0\|_{H^{s}}$ and on the modulus $\omega_{v_0}$.
\end{lemma}

\begin{proof}
For $\Ufield\in\mathcal U_{R,\delta}$, Lemma~\ref{lem:pressure}, the definition of $R$ in \eqref{eq:Rdefn} and the bounds \eqref{eq:free-fixed-time-modulus},
\eqref{eq:cutoff-duhamel} give
\[
 \|\mathcal M\Ufield\|_{H^{s-1,1}_{\breg}}
 \le \frac R2+C_\delta T^\varepsilon R^2,
 \qquad
 \sup_{t\in I}\|\mathcal M\Ufield(t)\|_{Y^{s-1,1}}
 \le C\big(\omega_{v_0}(T)+T^{1/2}C_\delta R^2\big).
\]
We now fix the minimal existence time $T_*$. After enlarging the common constants in the preceding estimates, choose
$T_*\in(0,1]$ so that
\begin{equation}\label{eq:Tstar}
\begin{aligned}
 C_\delta T_*^\varepsilon R^2&\le\tfrac R2,
 &&\text{spacetime radius of the self-map},\\
 C\omega_{v_0}(T_*)&\le\tfrac\delta2,
 &&\text{fixed-time size of the free wave},\\
 CC_\delta T_*^{1/2}R^2&\le\tfrac\delta2,
 &&\text{fixed-time size of the Duhamel term},\\
 C_\delta T_*^\varepsilon R(1+R)&\le\tfrac12,
 &&\text{contraction constant}.
\end{aligned}
\end{equation}
The fourth constraint is fixed here only out of convenience.  Such a positive time exists in
regime~\textup{(R1)} because $\omega_{v_0}(T)\to0$.  For every $T\le T_*$,
the first three lines give respectively
\[
 \|\mathcal M\Ufield\|_{H^{s-1,1}_{\breg}}\le R,
 \qquad
 \sup_{t\in I}\|\mathcal M\Ufield(t)\|_{Y^{s-1,1}}\le\delta.
\]
Together with exactness and the traces (checked in the preparation part), these are precisely the defining
conditions of $\mathcal U_{R,\delta}$.

For non-emptiness, take the free evolution
$\Ufield^{(0)}(t):=\chi(t)\Sprop(t)\nabla v_0$ it is clearly exact and has the correct initial data.  The linear estimate and the
definition of $R$ give
$\|\Ufield^{(0)}\|_{H^{s-1,1}_{\breg}}\le R/2$, while
\eqref{eq:free-fixed-time-modulus} and the second line of
\eqref{eq:Tstar} give the remaining  $C_t Y^{s-1,1}$ bound.  
\end{proof}

\begin{lemma}[the map is a contraction]\label{lem:contraction}
With $T_*$ chosen by \eqref{eq:Tstar}, the map $\mathcal M$ satisfies
\begin{equation}\label{eq:map-contraction}
 \|\mathcal M\Ufield-\mathcal M\Vfield\|_{H^{s-1,1}_{\breg}}
 \le C_\delta T^\varepsilon R(1+R)
       \|\Ufield-\Vfield\|_{H^{s-1,1}_{\breg}}
 \le\tfrac12\|\Ufield-\Vfield\|_{H^{s-1,1}_{\breg}}
\end{equation}
for all $\Ufield,\Vfield\in\mathcal U_{R,\delta}$ and every $T\le T_*$.
\end{lemma}

\begin{proof}
The free terms in \eqref{eq:differential-map} cancel. Applying the first
estimate in \eqref{eq:cutoff-duhamel} to the difference and then
\eqref{eq:pressure-lipschitz} gives
\[
 \|\mathcal M\Ufield-\mathcal M\Vfield\|_{H^{s-1,1}_{\breg}}
 \le C_{\rm lin}T^\varepsilon
 \|dF[\Ufield]-dF[\Vfield]\|_{L^2(I;H^{s-1})}
 \le C_\delta T^\varepsilon R(1+R)
       \|\Ufield-\Vfield\|_{H^{s-1,1}_{\breg}}.
\]
The last line of \eqref{eq:Tstar} makes the coefficient at most $1/2$.
\end{proof}

Let us make the distinction between the two existence regimes quantitative.
In regime~\textup{(R1)}, the second condition in \eqref{eq:Tstar} depends on the
actual high-frequency tail of the initial data.  Indeed, for every $\lambda_0\ge1$,
\[
 \omega_{v_0}(T)
 \le T\lambda_0 \|v_0\|_{H^{s}}
 +\big\|\mathbf 1_{\{\langle D_{\parallel}\rangle>\lambda_0\}}
       \Lambda^{s-1}\nabla v_0\big\|_{L^2}.
\]
For any fixed $v_0$, we can first choose $\lambda_0$ so that the second term is small and
then choose $T$ to control the first at fixed $\lambda_0$.  The tail is not uniformly
controlled on bounded subsets of $H^s$, which is why the regime~\textup{(R1)}
lifespan is not a function of $\|v_0\|_{H^s}$ alone.

In regime~\textup{(R2)}, the estimate
$\omega_{v_0}(T)\le T^{\rpar}\|\nabla v_0\|_{Y^{s-1,\rpar}}$ from \eqref{eq:omegav0bound} shows that the second line of
\eqref{eq:Tstar} is ensured by the explicit norm condition
\begin{equation}\label{eq:Tstar-R2}
 CT_*^{\rpar}\|\nabla v_0\|_{Y^{s-1,\rpar}}\le\tfrac\delta2.
\end{equation}
All quantities in \eqref{eq:Tstar} and \eqref{eq:Tstar-R2} are then controlled
by $n,s,\breg,\rpar,v_{\rm A}$ and
$\|v_0\|_{H^s}+\|\nabla v_0\|_{Y^{s-1,\rpar}}$.  This is the asserted norm dependence of the
regime~\textup{(R2)} lifespan.  To propagate the stronger spacetime norm we
will separately impose
\begin{equation}\label{eq:Tstar-enhanced}
 C_{\rm lin}C_{\delta,\rpar}T_*^\varepsilon R\le\tfrac12.
\end{equation}
Condition \eqref{eq:Tstar-enhanced} is not needed to construct a fixed point by Picard iteration, but it ensures the improved regularity estimate propagates along the approximating sequence.

\subsection{Lagrangian existence} With the contraction map $\mathcal{M}$ at hand existence for the auxiliary problem in regime~(R1) follows quickly from Banach's fixed-point Theorem once $w$ is reconstructed from the fixed point $\Ufield$. Propagation of any positive amount of parallel regularity under Picard iterations addresses regime~(R2).

\begin{theorem}[existence for the auxiliary system]\label{thm:existence}
With the notation and assumptions above, take $T=T_*$ satisfying
\eqref{eq:Tstar}. Then $\mathcal M$ has a unique fixed point
$\Ufield\in\mathcal U_{R,\delta}$. The displacement $w$ reconstructed by
\eqref{eq:potential-reconstruction} solves
\eqref{eq:auxiliary-system} on $[0,T_*]$, satisfies $dw=\Ufield$ there,
and has the regularity
\begin{equation}\label{eq:existence-traces}
 dw\in C\big([0,T_*],Y^{s-1,1}\big)
       \cap C^1\big([0,T_*],H^{s-1}\big),
 \qquad
 E_i\partial_tX^j\in C\big([0,T_*],H^{s-1}\big).
\end{equation}
If $\|\nabla v_0\|_{Y^{s-1,\rpar}}<\infty$ for some $\rpar\in(0,1]$, then $T_*$
may additionally be chosen to satisfy
\eqref{eq:Tstar-R2} and \eqref{eq:Tstar-enhanced}, depending only on
$n,s,\breg,\rpar,v_{\rm A}$ and $\|v_0\|_{H^s}+\|\nabla v_0\|_{Y^{s-1,\rpar}}$.
For this choice,
\begin{equation}\label{eq:longitudinal-propagation}
 \|\Ufield\|_{H^{s-1,1+\rpar}_{\breg}}
 \le 2C_{\rm lin}\|\nabla v_0\|_{Y^{s-1,\rpar}}.
\end{equation}
\end{theorem}

\begin{proof}
Let $I=[0,T_*]$. Since $\mathcal U_{R,\delta}$ is a nonempty closed subset of the
Hilbert space $H^{s-1,1}_{\breg}$, Lemmas~\ref{lem:selfmap}
and \ref{lem:contraction} together with Banach's fixed-point theorem
give a unique fixed point $\Ufield\in\mathcal U_{R,\delta}$.

\textit{Back to the undifferentiated wave--Hodge system.} We first reconstruct the displacement using the Duhamel formula, namely
\begin{equation}\label{eq:potential-reconstruction}
 w(t):=\Sprop(t)v_0-\int_0^t\Sprop(t-t')F[\Ufield](t') dt'.
\end{equation}
It follows immediately from the formula that $w\in C^1(I;L^4)$. Indeed, $\partial_t\Sprop(r)=\Cprop(r)$, $v_0\in H^s\hookrightarrow L^4$, Lemma~\ref{lem:pressure} guarantees $F[\Ufield]\in C(I;L^4)$ $\partial_t\Sprop(r)=\Cprop(r)$, and
\[
 \|\Sprop(r)\|_{L^4\to L^4}\le |r|,
 \qquad
 \|\Cprop(r)\|_{L^4\to L^4}\le 1.
\]
Its derivative can be computed explicitly
\[
 \partial_tw(t)=\Cprop(t)v_0
 -\int_0^t\Cprop(t-t')F[\Ufield](t') dt'
\]
Moreover,
$(w,\partial_tw)(0)=(0,v_0)$ and by \eqref{eq:differential-map}
\[
 dw(t)=\Sprop(t)\nabla v_0
       -\int_0^t\Sprop(t-t')dF[\Ufield](t') dt'
       =(\mathcal M\Ufield)(t)=\Ufield(t)
\]
on $I$ in distributions, since the cutoffs in \eqref{eq:differential-map}
equal one there. From uniqueness of $\Ufield$ and of the canonical H\"odge representative it follows that on $I$ $\Ufield=dw$ and $F[\Ufield]=F[dw]$, proving $w$ solves \eqref{eq:auxiliary-system}.

Given the regularity of $\Ufield\in \mathcal U_{R,\delta}$ the expressions above and Corollary~\ref{cor:trace-Y} we deduce
$dw\in C_tY^{s-1,1}$ and $\partial_tdw\in C_tH^{s-1}$, which is the first
statement in \eqref{eq:existence-traces}.  Finally, by \eqref{eq:component-derivative} we have
\[
 E_i\partial_tX^j
 =\Kmat^a_i\partial_t\Ufield_a^j
 =\partial_t\Ufield_i^j
  +(\Kmat^a_i-\delta_i^a)\partial_t\Ufield_a^j,
\]
then Lemma~\ref{lem:coeff-calculus} gives
$\Kmat-\mathrm{Id}\in C_tY^{s-1,1}$ and the
second statement in \eqref{eq:existence-traces} follows from this and  \eqref{eq:tame-sp}.

\textit{Uniqueness in the adapted fixed-point class.}
Let $\widetilde w\in C^1(I;L^1_{\rm loc})$ solve
\eqref{eq:auxiliary-system} with the same data, and suppose
$d\widetilde w=\widetilde\Ufield$ on $I$ for some
$\widetilde\Ufield\in\mathcal U_{R,\delta}$.
The fields $\mathcal M\widetilde\Ufield$ and $d\widetilde w$
solve the same wave equation with the same data, so
Lemma~\ref{lem:wave-uniqueness} gives
\[
 \mathcal M\widetilde\Ufield=d\widetilde w=\widetilde\Ufield
 \quad\text{on }I.
\]
Since $\mathcal M$ depends only on the restriction to $I$,
$\mathcal M\widetilde\Ufield$ is a fixed point and by its uniqueness 
\[
    d w=\Ufield=d\widetilde w \qquad\text{and} \qquad \widetilde w-w=c(t)
\]
This difference solves the wave equation with zero initial data. In particular, $c''=0$, $c(0)=c'(0)=0$ and we conclude $\widetilde w=w$ as desired.

\textit{Additional parallel regularity \textup{(R2)}.} Suppose now that
$D_{\rpar}:=\|\nabla v_0\|_{Y^{s-1,\rpar}}<\infty$ for some
$0<\rpar\le1$. Start the Picard iteration from the free field
\[
 \Ufield^{(0)}(t):=\chi(t)\Sprop(t)\nabla v_0,
 \qquad
 \Ufield^{(j+1)}:=\mathcal M\Ufield^{(j)}.
\]
The proof of Lemma~\ref{lem:selfmap} shows that
$\Ufield^{(0)}\in \mathcal U_{R,\delta}$ and by Lemma \ref{lem:selfmap} the iterates stay in the class. Moreover, \eqref{lem:contraction} guarantees strong convergence $\Ufield^{(j)}\to\Ufield$ in $H^{s-1,1}_{\breg}$.
We now show that additional parallel regularity is in fact propagated along the sequence. Set
\[
 C_j:=\|\Ufield^{(j)}\|_{H^{s-1,1+\rpar}_{\breg}}.
\]
We can now apply iteratively Theorem \ref{thm:linear} with $(\sreg,\preg)=(s-1,1+\rpar)$, the Hodge bound \eqref{eq:pressure-anisotropic} and the uniform estimate the uniform bound
$\|\Ufield^{(j)}\|_{H^{s-1,1}_{\breg}}\le R$. We obtain
$C_0\le C_{\rm lin}D_{\rpar}$ first and then the recursive inequality
\[
 C_{j+1}
 \le C_{\rm lin}D_{\rpar}
   +C_{\rm lin}C_{\delta,\rpar}T^\varepsilon R C_j.
\]
Condition \eqref{eq:Tstar-enhanced} makes the coefficient of $C_j$ at most
$1/2$. Starting from $C_0\le C_{\rm lin}D_{\rpar}$, induction gives
$C_j\le2C_{\rm lin}D_{\rpar}$ for every $j$. A subsequence therefore converges
weakly in the Hilbert space $H^{s-1,1+\rpar}_{\breg}$. Its continuous embedding
into $H^{s-1,1}_{\breg}$ and the strong convergence there identify its weak limit with
$\Ufield$. By weak lower semicontinuity of the norm, we obtain
\eqref{eq:longitudinal-propagation}.
\end{proof}

\section{Determinant Constraint and Return to MHD}\label{sec:return-mhd}
\subsection{The volume constraint propagates}
\label{ssec:constraint}

The system \eqref{eq:auxiliary-system} does not contain $\det dX=1$, however, any solution
propagates it dynamically at the available regularity.

\begin{proposition}[weak Jacobi's identity]\label{prop:constraint}
Let $X=\mathrm{id}+w$ on $[0,T]\times\Rn$ and assume
\begin{equation}\label{eq:constraint-hyp}
\Ufield=dw\ \text{on }[0,T]\times\Rn
\ \text{for some }\Ufield\in H^{s-1,1}_{\breg},
\qquad \big\|dw\big\|_{L^\infty}\le\tfrac12,\qquad
d(\Box X)\in L^1_tL^2_{y} .
\end{equation}
Then $\varrho:=\log\det dX$ belongs to
$C^1([0,T];L^2)$, with
$\mathcal B\varrho\in C([0,T];L^2)$, and satisfies
\[
 \Box\varrho
 =\tr\bigl(\Kmat d(\Box X)\bigr)+\mathcal Q[dw]
\]
in distributions on $(0,T)\times\Rn$.
In particular, if $w$ solves \eqref{eq:auxiliary-system} with
$v_0$ divergence-free, then $\det dX\equiv1$ on $[0,T]$.
\end{proposition}

\begin{proof}
Set $\varrho:=\log\det dX$. We first compute as if all the objects were smooth. From Jacobi's formula we have
\begin{equation}\label{eq:constraint-first-order-jacobi}
 \partial_t\varrho=\tr(\Kmat\partial_t\Ufield),
 \qquad
 \mathcal B\varrho=\tr(\Kmat\mathcal B\Ufield)
\end{equation}
and using
$\mathcal A^+\Kmat=-\Kmat(\mathcal A^+\Ufield)\Kmat$
together with \eqref{eq:datum-geometric} we obtain
\begin{equation}\label{eq:weak-jacobi}
\begin{aligned}
 \Box\varrho
 &=\mathcal A^+\tr(\Kmat\mathcal A^-\Ufield)\\
 &=\tr\bigl(\Kmat d(\Box X)\bigr)
   -\tr\bigl(
      \Kmat(\mathcal A^+\Ufield)
      \Kmat(\mathcal A^-\Ufield)\bigr)\\
 &=\tr\bigl(\Kmat d(\Box X)\bigr)+\mathcal Q[dw].
\end{aligned}
\end{equation}
We now justify this calculation at the stated regularity.
Corollary~\ref{cor:trace-Y} gives
\[
 \Ufield\in C_tY^{s-1,1}\hookrightarrow C_tL^\infty,
 \qquad
 \partial_t\Ufield,\mathcal B\Ufield\in C_tH^{s-1}
 \hookrightarrow C_t(L^2\cap L^4),
\]
since $s-1>(n-1)/2\ge n/4$.
The bound $\|\Ufield\|_{L^\infty}\le\tfrac12$ gives
$\Kmat\in C_tL^\infty$, and the Sobolev chain rule proves
\eqref{eq:constraint-first-order-jacobi} and the formula for
$\mathcal A^+\Kmat$. Further differentiation and \eqref{eq:constraint-hyp} give
\[
 \mathcal A^+\Kmat
 =-\Kmat(\mathcal A^+\Ufield)\Kmat\in L^2_tL^4,
 \qquad
 \mathcal A^+(\mathcal A^-\Ufield)
 =d(\Box X)\in L^1_tL^2.
\]
Thus $\Kmat$ and $\mathcal A^-\Ufield$ are locally absolutely
continuous along almost every $\mathcal A^+$ characteristic.
It follows that the Lebnitz rule for $\mathcal{A}^+$ used above holds in distribution, with both terms in
$L^1_tL^2$, as $\Kmat$ is bounded and
$\mathcal A^\pm\Ufield\in L^2_tL^4$.
This proves \eqref{eq:weak-jacobi} in distributions.

Moreover, $|\varrho|\lesssim|\Ufield|$, \eqref{lem:Linftyembed} and
\eqref{eq:constraint-first-order-jacobi} give
\[
 \varrho\in C^1([0,T];L^2),
 \qquad
 \mathcal B\varrho\in C([0,T];L^2).
\]

Suppose now that $w$ solves \eqref{eq:auxiliary-system} with
$v_0$ divergence-free. The trace bounds and \eqref{eq:canonical-hls} give
$F[dw]\in L^\infty_tL^4$. Together with
$dF[dw]=-d(\Box X)\in L^1_tL^2$, this yields
$F[dw]\in L^1_tW^{1,2}_{\rm loc}$.
Lemma~\ref{lem:weak-pullback} therefore identifies
$\tr(\Kmat dF[dw])=\Div_{\gmetric}F[dw]$,
so \eqref{eq:auxiliary-system} gives
\begin{equation}\label{eq:cancellationjacobi}
    \Box\varrho
 =-\Div_{\gmetric}F[dw]+\mathcal Q[dw]=0.
\end{equation}
The Cauchy data imply
\[
 \varrho(0)=0,
 \qquad
 \partial_t\varrho(0)=\Div_yv_0=0.
\]
From the uniqueness in Lemma~\ref{lem:wave-uniqueness} we obtain $\varrho=0$ and thus $\det dX\equiv1$ on $[0,T]$.
\end{proof}

\begin{remark}[generalized Jacobi identity]\label{rem:generalcaseJacobi}
The proof of Proposition \ref{prop:constraint} in fact gives a more general statement. Assume $\Box X=-F$ and $\Div_{\gmetric}F=f$ then \eqref{eq:weak-jacobi} gives
\[
    \Box\varrho=\mathcal Q[dw]-f.
\]
We also note that the inverse-deformation
$\Kmat=(dX)^{-1}$ in \eqref{eq:datum}, rather than the cofactor
$J\Kmat$ used in \cite{Zhang2026} is exactly what allow us to prove $\det dX\equiv1$.
\end{remark}

\begin{remark}[weak limits and the relaxation of the volume constraint]
\label{rem:weak-limit-heuristic}
The following heuristics, stated in $n=3$ for concreteness, try to motivate the off-constraint class we use in the fixed-point argument, where the volume preservation is not required.
Let $(v_m,B_m)$ be smooth solutions of \eqref{eq:mhd} and let
$E_m:=B_m\times v_m$ be the associated electric fields, as prescribed by
the ideal Ohm law. Assume $v_m$, $B_m$ and $E_m$ are bounded in
$L^2_{\rm loc}(\Real^{1+3})$ and converge weakly:
\[
 v_m\rightharpoonup v,\qquad B_m\rightharpoonup B,\qquad
 E_m\rightharpoonup E,
\]
with $B\neq0$ almost everywhere.  Write $\mathrm{vol}_x=dx^1\wedge dx^2\wedge
dx^3$ and let
\[
 F_m:=\iota_{B_m}\mathrm{vol}_x+E_m^\flat\wedge dt
\]
be the Faraday $2$-form. Then $dF_m=0$ and
$F_m\wedge F_m=-2(B_m\cdot E_m)dt\wedge\mathrm{vol}_x=0$. This div--curl structure gives weak continuity à la Murat and
Tartar~\cite{Murat1978,Tartar1979} and passing to the limit we deduce
\[
 B\cdot E=0 .
\]
From compensated compactness alone, we cannot deduce $B_m\times v_m \rightharpoonup B\times v$. For some function $\alpha$ set
\[
 V_{\rm eff}:=\frac{E\times B}{|B|^2}+\alpha B,
 \qquad
 u:=V_{\rm eff}-v.
\]
Then $E=B\times V_{\rm eff}=B\times(v+u)$ and letting
$\beta:=\iota_B\mathrm{vol}_x$ be the
$2$-form dual to $B$, after passing by linearity to the limit in the Faraday--Ohm systems one formally obtains
\[
 \partial_t\beta+\mathcal L_{V_{\rm eff}}\beta=0.
\]
Note that $\Div V_{\rm eff}=0$ might not hold in general and this could be read as the Eulerian counterpart of running the fixed point on trial
1-forms with $\det(\mathrm{Id}+\Ufield)\neq1$.
\end{remark}

\subsection{Eulerian reconstruction and existence}
\label{ssec:eulerian}

We now show that an auxiliary solution determines a distributional solution of
\eqref{eq:mhd}.
Here $\nabla v$ and $\nabla B$ lie only in $C_tH^{s-1}$ and $\nabla^2p$ only in
$C_tH^{s-2}$, so the Eulerian derivatives entering the formal smooth change-of-variables
calculation that leads to \eqref{eq:auxiliary-system} need not have the classical regularity required for that calculation,
and that calculation cannot simply be reversed by composing with $X^{-1}$.  All the statements need to be verified in an integral sense against test functions. 

\begin{proof}[Proof of clauses~\textup{(1)}--\textup{(3)} and the existence
statements in Theorem~\ref{thm:main}]
Set $T:=T_*$. Let $\Ufield, w$ and $X$ as in
Theorem~\ref{thm:existence}. The corresponding vector fields along $X$ are
\[
 F=F[dw],\qquad
 \vLag=\partial_tX,\qquad \BLag=\mathcal BX,
\]
while the Eulerian ones 
\[
    v=\vLag\circ X^{-1},\qquad B=\BLag\circ X^{-1}.
\]
The bounds in
\eqref{eq:main-adapted-bound} are then already in the statement of Theorem~\ref{thm:existence} and \eqref{eq:embed-frac} gives
$\|dw\|_{L^\infty}\le\frac12$.

\emph{The constraint and preliminary bounds.} Proposition~\ref{prop:constraint} now ensures $\det dX=1$. Moreover, arguing as in
\eqref{eq:trial-bilipschitz} gives that $X(t,\cdot)$ is a
volume-preserving bi-Lipschitz homeomorphism, which proves the
geometric assertions in clause~(2).

Recall $\sigma=s-1$ and
$\bLag:=b\circ X=\BLag-\Binf=\mathcal Bw=(dw)\Binf$.
From the trace estimates \eqref{eq:existence-traces} and the bounds in Theorem~\ref{thm:existence} we have
\begin{equation}\label{eq:preliminaryeulerian}
 \partial_tdw,\ \mathcal Bdw\in C_tH^\sigma,\qquad
 \bLag\in C_tH^\sigma,
\end{equation}
In particular $v_0\in L^4$ as $H^\sigma\hookrightarrow L^4$ since
$\sigma>(n-1)/2\ge n/4$.
Similarly, $\mathcal B^2w=(\mathcal Bdw)\Binf\in C_tH^\sigma\hookrightarrow C_tL^4$,
while $F\in C_tL^4$ by Lemma~\ref{lem:pressure}.
Integrating $\partial_t\vLag=\mathcal B^2w-F$ from
$\vLag(0)=v_0$ then gives $\vLag\in C_tL^4$ ($\bLag$ is even more regular). Gathering the above we deduce
\begin{equation}\label{eq:return-preliminary}
 \vLag,\bLag\in C_tL^4,\qquad
 d\vLag=\partial_tdw,\quad d\bLag=\mathcal Bdw
 \quad\text{in }C_tH^\sigma.
\end{equation}

\emph{Return to Eulerian variables.}
For $\varphi\in C_c^\infty([0,T)\times\Rn;\Rn)$, put $\Phi:=\varphi\circ X$.
The Sobolev chain rule gives
\begin{equation}\label{eq:sobolevchainrule}
 \partial_t\Phi=(\partial_t\varphi)\circ X
 +((\nabla\varphi)\circ X)\vLag,
 \qquad
 \mathcal B\Phi=((\nabla\varphi)\circ X)\BLag
 \quad\text{in }\mathcal D'\text{ and a.e.}
\end{equation}
By the reconstruction \eqref{eq:potential-reconstruction}
$w\in C([0,T];L^4)$ and averaging over unit balls (or citing Gagliardo-Nirenberg) gives
\[
 |w(t,y)|
 \le \frac{1}{|B_1|}
       \int_{B_1(y)}|w(t,z)|\,dz
      +\|dw(t)\|_{L^\infty}
 \lesssim \|w(t)\|_{L^4}+\|dw(t)\|_{L^\infty}.
\]
In particular, $w$ is uniformly bounded and since $X=\mathrm{id}+w$,
the support of $\varphi\circ X$ is still compact.
Since $\vLag,\BLag,F\in L^2_{\rm loc}$ and $\Phi\in W^{1,2}$ has fixed compact support by the computations above, it can be used as a valid test function for the Lagrangian auxiliary problem \eqref{eq:auxiliary-system}.

Using volume preservation, $\Phi(0)=\varphi(0)$ and
$\partial_t\vLag-\mathcal B\BLag=-F$, we compute
\begin{equation}\label{eq:return-momentum}
\begin{aligned}
 &\iint v\cdot\partial_t\varphi\,dx\,dt
   +\iint (v\otimes v-B\otimes B):\nabla\varphi\,dx\,dt
   +\int v_0\cdot\varphi(0)\,dx\\
 &\qquad=\iint \vLag\cdot\partial_t\Phi\,dy\,dt
   -\iint \BLag\cdot\mathcal B\Phi\,dy\,dt
   +\int \vLag(0)\cdot\Phi(0)\,dy\\
 &\qquad=-\iint
   (\partial_t\vLag-\mathcal B\BLag)\cdot\Phi\,dy\,dt\\
 &\qquad=\iint F\cdot\Phi\,dy\,dt
   =\iint (F\circ X^{-1})\cdot\varphi\,dx\,dt.
\end{aligned}
\end{equation}
For test functions supported in $(0,T)\times\Rn$, this gives
\begin{equation}\label{eq:weakg}
    \partial_tv+\Div_x(v\otimes v-B\otimes B)=-F\circ X^{-1}
 \quad\text{in }\mathcal D'.
\end{equation}
For $\psi\in C_c^\infty([0,T)\times\Rn;\Rn)$,
put $\Psi:=\psi\circ X$, the same admissibility remarks above hold.
Since $\BLag(0)=\Binf$ and
$\partial_t\BLag=\mathcal B\vLag$, a similar calculation gives
\begin{equation}\label{eq:return-induction}
\begin{aligned}
 &\iint B\cdot\partial_t\psi\,dx\,dt
   +\iint (B\otimes v-v\otimes B):\nabla\psi\,dx\,dt
   +\int\Binf\cdot\psi(0)\,dx\\
 &\qquad=\iint \BLag\cdot\partial_t\Psi\,dy\,dt
   -\iint \vLag\cdot\mathcal B\Psi\,dy\,dt
   +\int\BLag(0)\cdot\Psi(0)\,dy\\
 &\qquad=-\iint
   (\partial_t\BLag-\mathcal B\vLag)\cdot\Psi\,dy\,dt=0.
\end{aligned}
\end{equation}
This is exactly the weak formulation of the induction equation of~\S\ref{sec:introduction}.

By \eqref{eq:constraint-first-order-jacobi} and
\eqref{eq:divcurlg},
\[
 \Div_{\gmetric}\vLag=\partial_t\log\det dX=0,\qquad
 \Div_{\gmetric}\BLag=\mathcal B\log\det dX=0.
\]
Since $X$ is bi-Lipschitz and
$\vLag,\BLag\in W^{1,2}_{\rm loc}$ by
\eqref{eq:return-preliminary},
Lemma~\ref{lem:weak-pullback}(ii)--(iii) gives
\begin{equation}\label{eq:diveulerian}
 \Div_xv=\Div_xB=0\qquad\text{in }\mathcal D'(\Rn).
\end{equation}

\emph{Pressure recovery $\nabla p_{\mathrm R}=F\circ X^{-1}$.}
Define
\begin{equation}\label{eq:return-pressure}
 p_{\mathrm R}:=\Riesz_i\Riesz_j
 \big(v^iv^j-\Binf^ib^j-b^i\Binf^j-b^ib^j\big),
 \qquad b:=B-\Binf.
\end{equation}
By \eqref{eq:return-preliminary},
\eqref{eq:pullback-chain-rule} and volume preservation,
$v,b\in L^\infty_tL^4$ and
$\nabla v,\nabla b\in L^\infty_t(L^2\cap L^4)$. Hence
\begin{equation}\label{eq:return-pressure-L2}
 \|\nabla p_{\mathrm R}\|_{L^2}
 \lesssim\|v\|_{L^4}\|\nabla v\|_{L^4}
 +\|b\|_{L^4}\|\nabla b\|_{L^4}
 +v_{\rm A}\|\nabla b\|_{L^2},
\end{equation}
uniformly in time.
Expressing $\curl_{\gmetric}F=0$ in Eulerian coordinates
by Lemma~\ref{lem:weak-pullback}, and taking the divergence
of \eqref{eq:weakg}, we obtain
\[
\begin{aligned}
 \curl_x(F\circ X^{-1})&=0,\\
 \Div_x(F\circ X^{-1})
 &=-\partial_i\partial_j(v^iv^j-B^iB^j)
  =\Delta_xp_{\mathrm R},
\end{aligned}
\]
in $\mathcal D'(\Rn)$ for almost every $t$.
Here we used $\Div_xv=0$ and \eqref{eq:return-pressure}.
It follows that
$F\circ X^{-1}-\nabla p_{\mathrm R}$ is harmonic for almost every $t$.
Since it belongs to $L^4+L^2$, the mean-value property implies
that it vanishes. Thus
\begin{equation}\label{eq:return-force}
 F\circ X^{-1}=\nabla p_{\mathrm R},\qquad
 F=(\nabla p_{\mathrm R})\circ X\in L^\infty_tL^2.
\end{equation}
Together with \eqref{eq:return-momentum}, \eqref{eq:return-induction}
and \eqref{eq:diveulerian}, this proves that $(v,B,p_{\mathrm R})$
solves \eqref{eq:mhd} in distributions with data $(v_0,\Binf)$.

\emph{Final regularity.}
By \eqref{eq:preliminaryeulerian} and \eqref{eq:return-force} we have
\[
 \partial_t\vLag=(\mathcal Bdw)\Binf-F\in L^\infty_tL^2
\]
and arguing as above integrating from $\vLag(0)=v_0\in L^2$ we deduce
$\vLag\in W^{1,\infty}_tL^2\subset C_tL^2$. From
\eqref{eq:preliminaryeulerian}, \eqref{eq:return-preliminary} and \eqref{eq:potential-reconstruction} we then obtain $\vLag,\bLag, w\in C_tH^s$. From the bounds in Theorem~\ref{thm:existence} and Lemma~\ref{lem:composition}(ii) we then conclude $v,b\in C_tH^s$.

Finally, \eqref{eq:return-pressure}, the algebra property of $H^s$, boundedness of the Riesz transforms and \eqref{eq:rieszmain} give
$p_{\mathrm R}\in C_tH^s$. This proves clause~(1).

\emph{Conservation.}
Since $\det dX=1$ change of variables gives
\[
\begin{aligned}
 E
 &=\tfrac12\int\bigl(|v|^2+|b|^2\bigr)\,dx
  =\tfrac12\int\bigl(|\vLag|^2+|\bLag|^2\bigr)\,dy,\\
 H_{\rm c}
 &=\int v\cdot b\,dx
  =\int\vLag\cdot\bLag\,dy.
\end{aligned}
\]
The preceding bounds ensure that
$\vLag,\bLag\in W^{1,\infty}_tL^2$ and differentiating these identities, using the evolution equations
and skew-adjointness of $\mathcal B$, we obtain
\[
\begin{aligned}
 \frac{dE}{dt}
 &=-\int F\cdot\vLag\,dy
  =-\int\nabla p_{\mathrm R}\cdot v\,dx=0,\\
 \frac{dH_{\rm c}}{dt}
 &=-\int F\cdot\bLag\,dy
  =-\int\nabla p_{\mathrm R}\cdot b\,dx=0.
\end{aligned}
\]
Here we used \eqref{eq:return-force}, volume preservation
and \eqref{eq:diveulerian}. This proves clause~(3).

\emph{The existence regimes.}
The lifespan of regime~\textup{(R1)} is that of
Theorem~\ref{thm:existence} and Lemma~\ref{lem:selfmap}.
If the datum has positive longitudinal regularity, then
$\omega_{v_0}(T)\le T^{\rpar}
\|\nabla v_0\|_{Y^{s-1,\rpar}}$, see \eqref{eq:omegav0bound}. The Picard argument in Theorem~\ref{thm:existence} also gives
\eqref{eq:longitudinal-propagation} and regime~\textup{(R2)} also follows.
\end{proof}

\subsection{Eulerian uniqueness}
\label{ssec:unconditional-uniqueness}

Banach's theorem gives uniqueness of the differentiated solution inside the
fixed-point class. We now prove the comparison statement needed for uniqueness
in the full Eulerian $C_tH^s$ class. The preceding construction has already
shown that the map $X$ is volume preserving and that its induced Eulerian
fields solve \eqref{eq:mhd}. No flow is constructed for the second solution.

We first record the following Lemma giving an explicit expression for the pressure and whose proof is immediate.

\begin{lemma}[the pressure of an $H^s$ distributional solution]
\label{lem:pressure-recovery}
Let $s>\frac n2$ and $(v,B)$ be a distributional solution of
\eqref{eq:mhd} on $[0,T]\times\Rn$ with data $(v_0,\Binf)$, in the sense of
Section~\ref{sec:introduction}. Assume also that $v,b=B-\Binf\in C([0,T],H^s)$ and define
\[
 p_{\mathrm R}:=\Riesz_i\Riesz_j
 \big(v^iv^j-\Binf^ib^j-b^i\Binf^j-b^ib^j\big).
\]
Then $p_{\mathrm R}\in C([0,T],H^s)$,
$\nabla p_{\mathrm R}\in C([0,T],H^{s-1})$, and
\[
 \partial_tv+\Div_x(v\otimes v-B\otimes B)+\nabla p_{\mathrm R}=0
 \quad\text{in }\mathcal D'((0,T)\times\Rn).
\]
\end{lemma}

In the following, we use italic letters for Eulerian fields and sans serif letters for fields
composed with $X$.

\begin{proposition}[weak-strong uniqueness]
\label{prop:unconditional-uniqueness}
Let $(v,B,p)$ and $X$ be the Lagrangian solution
constructed in \S\ref{ssec:eulerian} on $[0,T_*]$.
If $(\widetilde v,\widetilde B, \tilde p)$ is any other distributional solution with
the same initial data, pressure normalisation and
\[
 \widetilde v,\ \widetilde B-\Binf\in C([0,T_*],H^s),
\]
then
\[
 \widetilde v=v,\qquad \widetilde B=B
 \quad\text{on }[0,T_*].
\]
In particular, clause~\textup{(4)} of Theorem~\ref{thm:main} holds.
\end{proposition}

\begin{proof}
\emph{Step 1: the difference equation in the reference geometry.} We first apply Lemma~\ref{lem:pressure-recovery} to both solutions and still call the two obtained
normalised pressures $p$ and $\widetilde p$. Now let
\[
 Z^\pm=v\pm B,
 \qquad
 \widetilde Z^\pm=\widetilde v\pm\widetilde B,
 \qquad
 \delta Z^\pm=\widetilde Z^\pm-Z^\pm, \qquad \delta p:=\widetilde p-p
\]
denote the Eulerian objects and 
\[
 \mathsf Z^\pm:=Z^\pm\circ X=\mathcal A^\pm X,
 \qquad
 \delta\mathsf Z^\pm:=\delta Z^\pm\circ X,
 \qquad
 \delta\mathsf p:=\delta p\circ X.
\]
the corresponding quantities along $X$. After some algebraic computations, we obtain
\[
 \partial_t\delta Z^\pm
 +\Div_x\bigl(\delta Z^\pm\otimes Z^\mp
 +Z^\pm\otimes\delta Z^\mp+\delta Z^\pm\otimes\delta Z^\mp\bigr)
 +\nabla\delta p=0
\]
in distributions. The momentum equation and the weak induction equation in \eqref{eq:mhd} then give
$\partial_t\delta Z^\pm\in C_tH^{s-1}$ and thus
$\delta Z^\pm\in W^{1,2}_{\mathrm{loc}}$ in spacetime.  Moreover, the map
$(t,y)\mapsto(t,X(t,y))$ is bi-Lipschitz because $X(t,\cdot)$ is uniformly
bi-Lipschitz by construction and $\partial_tX=v\circ X\in L^\infty$. Passing to the vector fields along $X$ we then obtain 
\[
 \partial_t\delta\mathsf Z^\pm
 =(\partial_t\delta Z^\pm)\circ X
   +(\partial_tX)^jE_j\delta\mathsf Z^\pm,
 \qquad
 \mathcal B\delta\mathsf Z^\pm
 =(\mathcal BX)^jE_j\delta\mathsf Z^\pm
 \quad\text{a.e. and in }\mathcal D'.
\]
where we use the notation from \S\ref{sec:wave-hodge}. Now arguing as in the proof of Lemma \ref{lem:wave-hodge-reduction} by means of $\mathcal A^\mp X=\mathsf Z^\mp=Z^\mp\circ X$, the identities above and $(\nabla\delta p)\circ X=\grad_{\gmetric}\delta\mathsf p$, we deduce:
\begin{equation}\label{eq:reference-difference}
\begin{cases}
    \mathcal A^\mp\delta\mathsf Z^\pm
 +\Div_{\gmetric}\bigl(
   (\mathsf Z^\pm+\delta\mathsf Z^\pm)
   \otimes\delta\mathsf Z^\mp\bigr)
 +\grad_{\gmetric}\delta\mathsf p=0,\\
 \Div_{\gmetric}\delta\mathsf Z^\pm=0,\\
 \delta\mathsf Z^\pm(0)=0.
\end{cases}
\end{equation}
Lemma~\ref{lem:composition} and \eqref{eq:reference-difference} yield
\[
 \delta\mathsf Z^\pm
 \in C_tH^s\cap C_t^1H^{s-1}\subset C_t^1L^2.
\]
This regularity justifies the calculations below.

\emph{Step 2: transport of the energy density and the interaction flux.} We proceed on the lines of~\cite[Section~2]{GolsePerthame2013}.
After a fixed rotation, assume $\Binf=v_{\rm A}e_1$, and define the energy densities
\begin{equation}\label{eq:endesinties}
    e_\pm:=\frac12|\delta\mathsf Z^\pm|^2.
\end{equation}
Taking the scalar product of \eqref{eq:reference-difference} with
$\delta\mathsf Z^\pm$ and using
$\Div_{\gmetric}\delta\mathsf Z^\mp=0$ gives
\begin{equation}\label{eq:uniqueness-density-transport}
\begin{split}
 (\partial_t\mp v_{\rm A}\partial_1)e_\pm
 =-\delta\mathsf Z^{\pm,j}\Bigl[
 \delta\mathsf Z^{\mp,i}
 E_i(\mathsf Z^{\pm,j}+\delta\mathsf Z^{\pm,j})
 +E_j\delta\mathsf p\Bigr].
\end{split}
\end{equation}
Let $t_0$ be a time through which the two solutions agree and put
$I=[t_0,t_1]\subset[0,T_*]$. Writing $r=y_1$, we introduce the line energies
\[
 \rho_\pm(t,r):=\int_{\Real^{n-1}}e_\pm(t,r,y')dy',
 \qquad
 R_-(t,r):=\int_{-\infty}^{r}\rho_-(t,q)dq,
 \qquad
 R_+(t,r):=\int_{r}^{+\infty}\rho_+(t,q)dq,
\]
together with the interaction functional and the interaction flux
\[
 \mathscr I(t):=\int_{r_+>r_-}
 \rho_+(t,r_+)\rho_-(t,r_-)dr_+dr_-,
 \qquad
 \Phi_I:=\left(
 \int_I\int_{\Real}\rho_+(t,r)\rho_-(t,r)drdt\right)^{1/2}.
\]
These definitions were motivated geometrically in \S\ref{ssec:overview}.
The regularity established in Step~1 and H\"older's inequality give
\[
 \rho_\pm\in C^1(I;L^1(\Real))\cap C(I;W^{1,1}(\Real)),
 \qquad
 R_\pm\in C^1(I;L^\infty(\Real)).
\]
In particular, $\rho_\pm$ are bounded and $\Phi_I$ is finite.
Since $\mathscr I=\int_{\Real}\rho_+R_-dr$, the product rule and Tonelli's theorem give
\[
\begin{aligned}
 \mathscr I'(t)
 &=\int_{\Real}(\partial_t\rho_+)(t,r)R_-(t,r)dr
   +\int_{\Real}\rho_+(t,r)\partial_tR_-(t,r)dr\\
 &=\int_{\Real}(\partial_t\rho_+)(t,r)R_-(t,r)dr
   +\int_{\Real}\rho_+(t,r)
     \int_{q<r}\partial_t\rho_-(t,q)dqdr\\
 &=\int_{\Real}(\partial_t\rho_+)(t,r)R_-(t,r)dr
   +\int_{\Real}\left(\int_{r>q}\rho_+(t,r)dr\right)
     \partial_t\rho_-(t,q)dq\\
 &=\int_{\Real}
   \bigl((\partial_t\rho_+)R_-+R_+\partial_t\rho_-\bigr)dr.
\end{aligned}
\]
Integrating \eqref{eq:uniqueness-density-transport} in $y'$ and
substituting the two time derivatives \eqref{eq:uniqueness-density-transport} therefore gives
\begin{equation}\label{eq:flux1}
    \begin{split}
 \mathscr I'(t)
 &=v_{\rm A}\int_{\Real}
   \bigl((\partial_r\rho_+)R_--R_+\partial_r\rho_-\bigr)dr\\
 &\quad-\sum_\pm\int_{\Rn}R_\mp(t,y_1)\delta\mathsf Z^{\pm,j}
 \Bigl[\delta\mathsf Z^{\mp,i}
 E_i(\mathsf Z^{\pm,j}+\delta\mathsf Z^{\pm,j})
 +E_j\delta\mathsf p\Bigr]dy.
\end{split}
\end{equation}
By the definition of $R_\pm$, we have $\partial_rR_\pm=\mp\rho_\pm$ and integrating by parts we obtain
\[
\begin{aligned}
 &v_{\rm A}\int_{\Real}
 \bigl((\partial_r\rho_+)R_--R_+\partial_r\rho_-\bigr)dr\\
 &\qquad=-v_{\rm A}\int_{\Real}\rho_+\partial_rR_-dr
         +v_{\rm A}\int_{\Real}(\partial_rR_+)\rho_-dr\\
 &\qquad=-v_{\rm A}\int_{\Real}\rho_+\rho_-dr
         -v_{\rm A}\int_{\Real}\rho_+\rho_-dr
         =-2v_{\rm A}\int_{\Real}\rho_+\rho_-dr.
\end{aligned}
\]
Substituting this in \eqref{eq:flux1} above, we deduce
\begin{equation}\label{eq:interaction-flux-identity}
 \mathscr I'(t)+2v_{\rm A}\int_{\Real}\rho_+\rho_-dr=-\sum_\pm\int_{\Rn}R_\mp(t,y_1)\delta\mathsf Z^{\pm,j}\Bigl[
 \delta\mathsf Z^{\mp,i}
 E_i(\mathsf Z^{\pm,j}+\delta\mathsf Z^{\pm,j})
 +E_j\delta\mathsf p\Bigr]dy.
\end{equation}
All terms are integrable by the bounds in Step~1.
There are no boundary terms because $\rho_\pm\in W^{1,1}(\Real)$
vanish at infinity and $R_\pm$ are bounded.

\emph{Step 3: control of the energy by the flux.}
Define the averaged and normalized Elsasser fields
\[
 z^\pm:=\frac12\bigl(Z^\pm+\widetilde Z^\pm\bigr)\mp\Binf
\]
and for $|I|\le1$, set
\[
 c_I:=
 \sup_{t\in I}\bigl(
  \|\delta Z^+(t)\|_{H^s}+\|\delta Z^-(t)\|_{H^s}\bigr)+\left(\int_I
  \bigl(\|z^+(t)\|_{H^s}^2+\|z^-(t)\|_{H^s}^2\bigr)dt
  \right)^{1/2}.
\]
Since the difference vanishes at $t_0$ and belongs to $C_tH^s$, one has
$c_I\to0$ as $t_1\downarrow t_0$. The identities
\[
 \widetilde Z^\pm\mp\Binf=z^\pm+\frac12\delta Z^\pm,
 \qquad
 E_i(\mathsf Z^{\pm,j}+\delta\mathsf Z^{\pm,j})
 =(\partial_i\widetilde Z^{\pm,j})\circ X
\]
then give by the embedding $H^{s-1}_y\hookrightarrow
L^2_{y_1}H^{s-1}_{y'}$ and Lemma~\ref{lem:composition} the bound
\begin{equation}\label{eq:uniqueness-small-coefficients}
\begin{aligned}
 \sum_\pm
 \left\|\bigl(E_i(\mathsf Z^{\pm,j}
 +\delta\mathsf Z^{\pm,j})\bigr)_{i,j}\right\|_
 {L^2(I;L^2_{y_1}H^{s-1}_{y'})}
 &\le\sum_\pm
 \| (\nabla\widetilde Z^\pm)\circ X\|_{L^2(I;H^{s-1}_y)}\\
 &\lesssim\sum_\pm
 \|\nabla\widetilde Z^\pm\|_{L^2(I;H^{s-1}_x)}\\
 &\lesssim\sum_\pm\bigl(
 \|z^\pm\|_{L^2(I;H^s_x)}
 +\|\delta Z^\pm\|_{L^2(I;H^s_x)}\bigr)\\
 &\lesssim\left(\int_I\bigl(
 \|z^+(t)\|_{H^s}^2+\|z^-(t)\|_{H^s}^2\bigr)dt\right)^{1/2}\\
 &\quad+|I|^{1/2}\sup_{t\in I}\bigl(
 \|\delta Z^+(t)\|_{H^s}+\|\delta Z^-(t)\|_{H^s}\bigr)\\
 &\le c_I.
\end{aligned}
\end{equation}
The last inequality uses $|I|\le1$.
All the implicit constants in the remainder of the proof may depend on the
standing parameters and on the bounds for $X$ and $\Kmat$ on
$[0,T_*]$, but never on $I$.

We now define the time-localised energy
\[
 \mathcal E_I:=\sup_{t\in I}\int_{\Rn}(e_++e_-)(t,y)dy.
\]
For $t\in I$, integrating
\eqref{eq:uniqueness-density-transport} in space and from $t_0$ to $t$,
and using $e_\pm(t_0)=0$, gives
\[
 \int_{\Rn}(e_++e_-)(t,y)dy=\sum_\pm\int_{t_0}^t\int_{\Rn}
 \Bigl[
 \pm v_{\rm A}\partial_1e_\pm
 -\delta\mathsf Z^{\pm,j}\delta\mathsf Z^{\mp,i}
 E_i(\mathsf Z^{\pm,j}+\delta\mathsf Z^{\pm,j})
 -\delta\mathsf Z^{\pm,j}E_j\delta\mathsf p
 \Bigr]dy d\tau.
\]
We first get rid of the transport and pressure terms. Since $e_\pm\in W^{1,1}(\Rn)$,
\[
 \int_{\Rn}\partial_1e_\pm\,dy=0.
\]
Moreover, $\Div_{\gmetric}\delta\mathsf Z^\pm=0$ and $\det dX=1$
give
\[
\begin{aligned}
 \int_{\Rn}\delta\mathsf Z^{\pm,j}E_j\delta\mathsf p\,dy
 &=\int_{\Rn}
 \Div_{\gmetric}(\delta\mathsf p\,\delta\mathsf Z^\pm)dy\\
 &=\int_{\Rn}\partial_a
 \bigl(\Kmat_j^a\delta\mathsf p\,\delta\mathsf Z^{\pm,j}\bigr)dy
 =0,
\end{aligned}
\]
where the vector field inside the divergence belongs to $L^1$.
Consequently,
\[
 \int_{\Rn}(e_++e_-)(t,y)dy
 =-\sum_\pm\int_{t_0}^t\int_{\Rn}
 \delta\mathsf Z^{\pm,j}\delta\mathsf Z^{\mp,i}
 E_i(\mathsf Z^{\pm,j}+\delta\mathsf Z^{\pm,j})dy d\tau.
\]
Taking the time supremum, applying H\"older's inequality in $y'$,
and using
$\|\delta\mathsf Z^\pm(t,r,\cdot)\|_{L^2_{y'}}
=(2\rho_\pm(t,r))^{1/2}$, we obtain
\[
\begin{aligned}
 \mathcal E_I&=\sup_{t\in I}\int_{\Rn}(e_++e_-)(t,y)dy\\
 &\le\sup_{t\in I}\sum_\pm\left|
 \int_{t_0}^t\int_{\Rn}
 \delta\mathsf Z^{\pm,j}\delta\mathsf Z^{\mp,i}
 E_i(\mathsf Z^{\pm,j}+\delta\mathsf Z^{\pm,j})dy d\tau
 \right|\\
 &\le\sum_\pm\int_{I\times\Rn}
 |\delta\mathsf Z^+||\delta\mathsf Z^-|
 \left|\bigl(E_i(\mathsf Z^{\pm,j}
 +\delta\mathsf Z^{\pm,j})\bigr)_{i,j}\right|dydt\\
 &\le\sum_\pm\int_{I\times\Real}
 \|\delta\mathsf Z^+(t,r,\cdot)\|_{L^2_{y'}}
 \|\delta\mathsf Z^-(t,r,\cdot)\|_{L^2_{y'}}
 \left\|\bigl(E_i(\mathsf Z^{\pm,j}
 +\delta\mathsf Z^{\pm,j})\bigr)_{i,j}(t,r,\cdot)
 \right\|_{L^\infty_{y'}}drdt\\
 &=2\sum_\pm\int_{I\times\Real}
 (\rho_+\rho_-)^{1/2}
 \left\|\bigl(E_i(\mathsf Z^{\pm,j}
 +\delta\mathsf Z^{\pm,j})\bigr)_{i,j}
 \right\|_{L^\infty_{y'}}drdt.
\end{aligned}
\]
Cauchy--Schwarz in $(t,r)$, transverse Sobolev embedding
$H^{s-1}_{y'}\hookrightarrow L^\infty_{y'}$, and
\eqref{eq:uniqueness-small-coefficients} now yield
\[
\begin{aligned}
 \mathcal E_I
 &\le 2\left(\int_{I\times\Real}\rho_+\rho_-drdt\right)^{1/2}
 \sum_\pm\left(\int_{I\times\Real}
 \left\|\bigl(E_i(\mathsf Z^{\pm,j}
 +\delta\mathsf Z^{\pm,j})\bigr)_{i,j}\right\|_{L^\infty_{y'}}^2
 drdt\right)^{1/2}\\
 &\lesssim\Phi_I\sum_\pm
 \left\|\bigl(E_i(\mathsf Z^{\pm,j}
 +\delta\mathsf Z^{\pm,j})\bigr)_{i,j}
 \right\|_{L^2(I;L^2_{y_1}H^{s-1}_{y'})}
 \lesssim c_I\Phi_I.
\end{aligned}
\]
We conclude the desired control of the energy by the interaction flux
\begin{equation}\label{eq:energy-by-interaction}
 \mathcal E_I\lesssim c_I\Phi_I.
\end{equation}

\emph{Step 4: pressure bound.} To bound the pressure term in \eqref{eq:interaction-flux-identity} we will need a
bound on $\delta\mathsf p$ in $L^2(I;L^\infty_rL^2_{y'})$. This will keep $\delta Z^\pm$ in $L^2$ and highlight the energy terms \eqref{eq:endesinties}. To obtain this one-dimensional $L^\infty_r$ bound, by Sobolev embedding, we need to gain more than half a derivative from the pressure equation. Fix
\begin{equation}\label{eq:alphaenergy}
    \frac12<\alpha<s-\frac n2,
\end{equation} 
so that
\begin{equation}\label{eq:bochner-embedding}
H^\alpha(\mathbb{R}^n)
=H^\alpha\bigl(\mathbb{R}_{y_1};L^2(\mathbb{R}^{n-1})\bigr)
 \cap L^2\bigl(\mathbb{R}_{y_1};H^\alpha(\mathbb{R}^{n-1})\bigr)
\hookrightarrow L^\infty_{y_1}L^2_{y'}.
\end{equation}
To do so, we use the standard trick of using different representations for $\delta p$ and $\nabla \delta p$.

After noticing that $\Binf$ simplifies in the expressions, we can write
\[
\begin{aligned}
 \delta p
 &=\Riesz_i\Riesz_j\bigl(
   \delta Z^{+,i}z^{-,j}+z^{+,i}\delta Z^{-,j}\bigr),\\
 \nabla\delta p
 &=\nabla(-\Delta_x)^{-1}\Div_x\bigl(
   (\delta Z^+\cdot\nabla)z^-
   +(\delta Z^-\cdot\nabla)z^+\bigr).
\end{aligned}
\]
The first formula controls low frequencies. In the second, incompressibility
is used to place the derivative on $z^\pm$, so that we can place $\delta Z^\pm \in L^2$, highlighting the energy terms as desired.

We now recall two classical inequalities. Since $s-\alpha>n/2$, standard Sobolev product estimate
give
\[
 \|ag\|_{H^{\alpha-1}}
 \le C\|a\|_{H^{s-1}}\|g\|_2.
\]
Moreover,
\[
 \|q\|_{H^\alpha}
 \le C\bigl(\|q\|_2+\|\nabla q\|_{H^{\alpha-1}}\bigr)
\]
see \cite[Chapter~V]{stein} (or use Plancherel and
$\langle\xi\rangle^{2\alpha}\lesssim
1+|\xi|^2\langle\xi\rangle^{2(\alpha-1)}$).  From the two representations above, we first obtain
\[
\begin{aligned}
 \|\delta p\|_2
 &\lesssim
 \|\delta Z^+\otimes z^-+z^+\otimes\delta Z^-\|_2\\
 &\le
 \|\delta Z^+\|_2\|z^-\|_\infty
 +\|z^+\|_\infty\|\delta Z^-\|_2\\
 &\lesssim
 \|\delta Z^+\|_2\|z^-\|_{H^s}
 +\|z^+\|_{H^s}\|\delta Z^-\|_2,
\end{aligned}
\]
where the last inequality uses $H^s\hookrightarrow L^\infty$, since $s>n/2$ and then
\[
\begin{aligned}
 \|\nabla\delta p\|_{H^{\alpha-1}}
    &\lesssim \|(\delta Z^+\cdot\nabla)z^-
          +(\delta Z^-\cdot\nabla)z^+\|_{H^{\alpha-1}}\\
    &\lesssim \|\delta Z^+\|_2\|\nabla z^-\|_{H^{s -1}}
 +\|\delta Z^-\|_2\|\nabla z^+\|_{H^{s -1}}\\
    &\lesssim \|\delta Z^+\|_2\|z^-\|_{H^s}
 +\|\delta Z^-\|_2\|z^+\|_{H^s} \,.\\
\end{aligned}
\]
Combining these 
\[
 \|\delta p(t)\|_{H_x^\alpha}
 \lesssim \bigl(\|z^+(t)\|_{H^s}+\|z^-(t)\|_{H^s}\bigr)
 \bigl(\|\delta Z^+(t)\|_2+\|\delta Z^-(t)\|_2\bigr).
\]
By Lemma~\ref{lem:composition}, volume preservation, and
\eqref{eq:bochner-embedding}, the preceding estimate gives
\[
\begin{aligned}
 \|\delta\mathsf p(t)\|_{L^\infty_rL^2_{y'}}
 &\lesssim\|\delta\mathsf p(t)\|_{H^\alpha_y}
 \lesssim\|\delta p(t)\|_{H^\alpha_x}\\
 &\lesssim\bigl(\|z^+(t)\|_{H^s}+\|z^-(t)\|_{H^s}\bigr)
 \bigl(\|\delta Z^+(t)\|_2+\|\delta Z^-(t)\|_2\bigr)\\
 &\lesssim\bigl(\|z^+(t)\|_{H^s}+\|z^-(t)\|_{H^s}\bigr)
 \left(\int_{\Rn}(e_++e_-)(t,y)dy\right)^{1/2}.
\end{aligned}
\]
Taking the $L^2(I)$ norm and using the definition of $c_I$, we finally obtain
\begin{equation}\label{eq:uniqueness-pressure-trace}
\begin{aligned}
 \|\delta\mathsf p\|_{L^2(I;L^\infty_rL^2_{y'})}
 &\lesssim\mathcal E_I^{1/2}
 \left(\int_I
 \bigl(\|z^+(t)\|_{H^s}+\|z^-(t)\|_{H^s}\bigr)^2dt\right)^{1/2}\\
 &\lesssim c_I\mathcal E_I^{1/2}.
\end{aligned}
\end{equation}

\emph{Step 5: control of the integrated right-hand side of
\eqref{eq:interaction-flux-identity}.}
Denote the advection and pressure contributions by
\[
\begin{aligned}
 \mathscr A_I
 &:=-\sum_\pm\int_I\int_{\Rn}R_\mp(t,y_1)
 \delta\mathsf Z^{\pm,j}\delta\mathsf Z^{\mp,i}
 E_i(\mathsf Z^{\pm,j}+\delta\mathsf Z^{\pm,j})dy dt,\\
 \mathscr P_I
 &:=-\sum_\pm\int_I\int_{\Rn}R_\mp(t,y_1)
 \delta\mathsf Z^{\pm,j}E_j\delta\mathsf p\,dy dt.
\end{aligned}
\]
Since $\mathscr I(t_0)=0$, integrating
\eqref{eq:interaction-flux-identity} over $I$ gives
\begin{equation}\label{eq:integrated}
    \mathscr I(t_1)+2v_{\rm A}\Phi_I^2=\mathscr A_I+\mathscr P_I.
\end{equation}
We estimate these two terms separately.

\emph{The advection term $\mathscr A_I$.}
Since $0\le R_\pm\le\mathcal E_I$, the calculation in Step~3 gives
\begin{equation}\label{eq:advection-interaction-bound}
\begin{aligned}
 |\mathscr A_I|
 &\le\sum_\pm\int_I\left|
 \int_{\Rn}R_\mp(t,y_1)
 \delta\mathsf Z^{\pm,j}\delta\mathsf Z^{\mp,i}
 E_i(\mathsf Z^{\pm,j}+\delta\mathsf Z^{\pm,j})dy
 \right|dt\\
 &\le\mathcal E_I\sum_\pm\int_{I\times\Rn}
 |\delta\mathsf Z^+||\delta\mathsf Z^-|
 \left|\bigl(E_i(\mathsf Z^{\pm,j}
 +\delta\mathsf Z^{\pm,j})\bigr)_{i,j}\right|dy dt\\
 &\lesssim\mathcal E_Ic_I\Phi_I
 \lesssim c_I^2\Phi_I^2.
\end{aligned}
\end{equation}

\emph{The pressure term $\mathscr P_I$.}
Using $\Div_{\gmetric}\delta\mathsf Z^\pm=0$ and
$\partial_rR_\mp=\pm\rho_\mp$, integration by parts gives
\begin{equation}\label{eq:intbyparts}
    \begin{aligned}
 &-\sum_\pm\int_{\Rn}
 R_\mp(t,y_1)\delta\mathsf Z^{\pm,j}E_j\delta\mathsf p\,dy\\
 &\qquad=\int_{\Rn}\delta\mathsf p\Kmat_j^1
 \bigl(\rho_-(t,y_1)\delta\mathsf Z^{+,j}
       -\rho_+(t,y_1)\delta\mathsf Z^{-,j}\bigr)dy.
\end{aligned}
\end{equation}
This is justified by the boundedness of $R_\pm,\rho_\pm,\Kmat$
and $\delta\mathsf p,\grad_{\gmetric}\delta\mathsf p,
\delta\mathsf Z^\pm\in L^2$. We now recover $\Phi_I$ from the remaining line energies, by means of \eqref{eq:uniqueness-pressure-trace} from Step 4. 

We need the following preliminary computation. For each $t\in I$, Cauchy--Schwarz in $r$ and
$\int_{\Real}\rho_\pm(t,r)dr\le\mathcal E_I$ give
\[
\begin{aligned}
 &\int_{\Real}
 \bigl(\rho_-\rho_+^{1/2}+\rho_+\rho_-^{1/2}\bigr)dr\\
 &\qquad=\int_{\Real}(\rho_+\rho_-)^{1/2}
 \bigl(\rho_-^{1/2}+\rho_+^{1/2}\bigr)dr\\
 &\qquad\le\left(\int_{\Real}\rho_+\rho_-dr\right)^{1/2}
 \left[\left(\int_{\Real}\rho_-dr\right)^{1/2}
       +\left(\int_{\Real}\rho_+dr\right)^{1/2}\right]\\
 &\qquad\le 2\mathcal E_I^{1/2}
 \left(\int_{\Real}\rho_+\rho_-dr\right)^{1/2}.
\end{aligned}
\]
Using this, \eqref{eq:intbyparts} above, the uniform bound on $\Kmat$,
H\"older's inequality in $y'$, and
$\|\delta\mathsf Z^\pm(t,r,\cdot)\|_{L^2_{y'}}
=(2\rho_\pm(t,r))^{1/2}$, we can estimate:
\begin{equation}\label{eq:pressurefluxbound}
\begin{aligned}
 |\mathscr P_I|
 &\lesssim\sum_\pm\int_I\int_{\Rn}
 \rho_\mp(t,y_1)|\delta\mathsf p(t,y)|
 |\delta\mathsf Z^\pm(t,y)|dy dt\\
 &\le\sum_\pm\int_I\int_{\Real}\rho_\mp(t,r)
 \|\delta\mathsf p(t,r,\cdot)\|_{L^2_{y'}}
 \|\delta\mathsf Z^\pm(t,r,\cdot)\|_{L^2_{y'}}drdt\\
 &\lesssim\int_I\|\delta\mathsf p(t)\|_{L^\infty_rL^2_{y'}}
 \int_{\Real}\bigl(\rho_-\rho_+^{1/2}
                   +\rho_+\rho_-^{1/2}\bigr)drdt\\
 &\lesssim\mathcal E_I^{1/2}
 \int_I\|\delta\mathsf p(t)\|_{L^\infty_rL^2_{y'}}
 \left(\int_{\Real}\rho_+\rho_-dr\right)^{1/2}dt\\
 &\le\mathcal E_I^{1/2}
 \|\delta\mathsf p\|_{L^2(I;L^\infty_rL^2_{y'})}
 \left(\int_I\int_{\Real}\rho_+\rho_-drdt\right)^{1/2}\\
 &\lesssim c_I\mathcal E_I\Phi_I
 \lesssim c_I^2\Phi_I^2.
\end{aligned}
\end{equation}
The last steps use Cauchy--Schwarz in $t$,
\eqref{eq:uniqueness-pressure-trace} and \eqref{eq:energy-by-interaction}.

\emph{Step 6: conclusion.} From \eqref{eq:integrated}, using
$\mathscr I(t_1)\geq0$ and the two bounds 
\eqref{eq:advection-interaction-bound}, \eqref{eq:pressurefluxbound},
we obtain
\[
 2v_{\rm A}\Phi_I^2
 \le\mathscr I(t_1)+2v_{\rm A}\Phi_I^2
 \le Cc_I^2\Phi_I^2.
\]
Since $c_I\to0$ as $t_1\downarrow t_0$, we may choose $I$ so that
$Cc_I^2<2v_{\rm A}$. Thus $\Phi_I=0$, and
\eqref{eq:energy-by-interaction} gives $\mathcal E_I=0$.
The two solutions therefore agree on $I$.

Finally, define
\[
 \mathcal S:=\left\{t\in[0,T_*]:
 \widetilde v=v\ \text{and}\ \widetilde B=B\ \text{on }[0,t]\right\}.
\]
The set $\mathcal S$ is nonempty and closed by continuity.  Openness follows from the previous local argument and the solutions must then agree on $[0,T_*]$. 
\end{proof}

\section{Legenda of Symbols}\label{ssec:table}

The most common symbols are collected below. Fonts are used as consistently as possible: plain italic
letters denote physical fields, scalar variables and the flow; sans-serif letters denote
vector fields along $X$, label-space tensors, null-frequency variables and explicit wave propagators; calligraphic letters
denote operators and function spaces; fraktur letters denote regularity indices, auxiliary
parameters or kernels.
The letter $D$ carries the self-adjoint Fourier derivatives $D_t$, $D_{\parallel}$, $D_{\pm}$, boldface distinguishes the fixed
background field and the pullback metric.

\begingroup\small
\setlength{\LTleft}{0pt}\setlength{\LTright}{0pt}
\begin{table}[H]
  \centering
  \small
  \renewcommand{\arraystretch}{0.85}
  
  \begin{tabular}{@{}>{\raggedright\arraybackslash}p{0.205\textwidth}>{\raggedright\arraybackslash}p{0.565\textwidth}>{\raggedright\arraybackslash}p{0.155\textwidth}@{}}
    \hline
\emph{symbol} & \emph{meaning} & \emph{introduced} \\ \hline
$\Bzero\equiv\Binf$ & fixed initial field; Alfv\'en speed $v_{\rm A}:=|\Binf|$ & \S\ref{sec:introduction} \\
$v,\ B,\ b;\ \vLag,\ \BLag,\ \bLag$ & velocity, total magnetic field, perturbation $b=B-\Binf$, and the corresponding vector fields along $X$ & \S\ref{sec:introduction} \\
$s,\ \sigma,\ \varepsilon$ & $\frac{n+1}2<s\le\frac n2+1$; $\sigma=s-1$; $\varepsilon=1-\breg$ & \S\ref{sec:introduction} \\
$\sreg,\ \preg,\ \breg$ & isotropic, parallel and modulation regularity indices; $\breg\in(\frac12,1)$ in the main argument & \S\ref{sec:introduction} \\
$x,\ y;\ X,\ w$ & Eulerian coordinate, material label, flow and displacement $X-\mathrm{id}$ & \S\ref{sec:introduction} \\
$\mathcal C_Xf=f\circ X$ & precomposition with the flow & \S\ref{sec:wave-hodge} \\
$d=d_{y}$ & exterior differential and componentwise label differential of $X,w,F$ and $\Box X$ & \S\ref{sec:introduction} \\
$\nabla,\ \nabla^{y}$ & flat Eulerian and flat label connections & \S\ref{sec:introduction} \\
$\eframe_i,\ E_i$ & Cartesian vector field $\partial_{x^i}\circ X$ along $X$ and its genuine pullback $\Kmat^a_i\partial_a$ to label space & \S\ref{sec:wave-hodge}, \eqref{eq:component-derivative} \\
$\Ufield$ & spatially exact spacetime fixed-point field; at a fixed point, $\Ufield=dw$ on $[0,T]$ & \eqref{eq:fixedpoint-class}, \eqref{eq:potential-reconstruction} \\
$\gmetric[\Ufield],\ \Kmat[\Ufield],\ J,\ \mu$ & pullback metric, inverse deformation matrix, Jacobian determinant and metric measure & \S\ref{sec:wave-hodge}, \eqref{eq:datum} \\
$\Mmat$ & deviation $\Mmat=\Kmat-\mathrm{Id}$ & \eqref{eq:KMdef} \\
$\mathcal B,\ D_{\parallel}$ & $\Binf \cdot \nabla^{y}$, of symbol $\mathrm i\lambda$; $D_{\parallel}=-\mathrm i\mathcal B$ self-adjoint & \eqref{eq:Bmult}, \eqref{eq:self-adjoint-derivatives} \\
$\mathcal A^{\pm},\ \Box$ & $\partial_t\pm\mathcal B$, and $\mathcal A^{+}\mathcal A^{-}=\partial_t^2-\mathcal B^2$ & \S\ref{sec:introduction} \\
$D_{\pm},\ D_t$ & $-\mathrm i\mathcal A^{\pm}=D_t\pm D_{\parallel}$ and $-\mathrm i\partial_t$ & \eqref{eq:self-adjoint-derivatives} \\
$\lambda,\ \tau,\ \tau_\pm$ & symbols of $D_{\parallel}$ and $D_t$, with $\tau_\pm=\tau\pm\lambda$ & \eqref{eq:Bmult}, \eqref{eq:joint-res} \\
$z_1,\ z_2,\ z$ & the two factor null frequencies and their sum in $(\tau_+,\tau_-)$ coordinates & \eqref{eq:null-coordinate-plancherel} \\
$\Lambda,\ \langle D_{\parallel}\rangle$ & $(1-\Delta_{y})^{1/2}$; the multiplier $\langle\lambda\rangle$ & \S\ref{sec:introduction} \\
$m,\ \Omega$ & modulation weight and wave weight & \eqref{eq:modulation}, \eqref{eq:Omega} \\
$\Box^{\mathrm{reg}}$ & regularised wave operator $(1+\mathcal A^+)(1+\mathcal A^-)$ & \eqref{eq:regularized-wave} \\
$Y^{\sreg,\preg},\ Y^{\sreg,\preg}_{\breg}$ & fixed-time and modulation-weighted spacetime scales & \eqref{eq:Y} \\
$H^{\sreg,\preg}_{\breg}$ & wave-adapted spacetime Hilbert norm & \eqref{eq:Hsymbol} \\
$Q_0,\ \mathcal Q[\Ufield]$ & bilinear Alfv\'en null form and the nonlinear scalar pressure datum & \eqref{eq:Q0}, \eqref{eq:datum} \\
$F[\Ufield]$ & canonical Hodge representative in $L^4$, giving the pressure force along $X$ & \eqref{eq:FU} \\
$\mathcal P_{\gmetric}f,\ d\mathcal P_{\gmetric}f$ & canonical Hodge representative and its distributional label derivative & \eqref{eq:pressure-gradient}, \eqref{eq:elliptic-Y} \\
$\varrho$ & $\log|\det dX|$ & \S\ref{sec:wave-hodge} \\
$\Phi(D_{+},D_{-})$ & joint Fourier multiplier with symbol $\Phi(\tau_+,\tau_-)$ & \eqref{eq:calculus} \\
$\nu$ & $d\tau_+ d\tau_-/(2v_{\rm A})$, the null-coordinate measure & \eqref{eq:null-coordinate-plancherel} \\
$\Cprop,\ \Sprop$ & free cosine and sine propagators for $\Box$ & \eqref{eq:propagator} \\
$c_\perp,\ c_\parallel$ & transverse Sobolev integral and one-dimensional null integral & \eqref{eq:embed-frac}, \eqref{eq:core-const} \\
$C_{\rm lin}$ & the constant of the linear estimate & \eqref{eq:linear} \\
$\mathcal U_{R,\delta},\ R$ & fixed-point class and radius & \eqref{eq:fixedpoint-class} \\
$\rpar,\ Y^{s-1,\rpar}$ & additional longitudinal regularity exponent of the datum and its fixed-time scale & \eqref{eq:Y}, Thm.~\ref{thm:main} \\
$\omega_{v_0}$ & the longitudinal high-frequency modulus of the datum & \eqref{eq:datum-modulus} \\
$E,\ H_{\rm c}$ & renormalised energy and cross helicity & \eqref{eq:energy} \\ \hline
\end{tabular}
\end{table}
\endgroup

\bibliographystyle{amsplain}
\bibliography{mhd-constant-field}

\end{document}